\documentclass[11pt]{article}

\usepackage{amsmath,amssymb,amsfonts,amsthm,enumitem,pstricks,pst-node,float,xy,mathrsfs,accents,hhline}
\xyoption{all}

\numberwithin{equation}{section}
\newtheorem{thm}{Theorem}
\newtheorem{prp}{Proposition}[section] 
\newtheorem{lmm}[prp]{Lemma}  
\newtheorem{crl}[prp]{Corollary}
\theoremstyle{definition}
\newtheorem{dfn}{Definition}
\newtheorem{eg}{Example}
\newtheorem{rmk}[prp]{Remark}

\DeclareFontFamily{U}{mathx}{\hyphenchar\font45}
\DeclareFontShape{U}{mathx}{m}{n}{
      <5> <6> <7> <8> <9> <10>
      <10.95> <12> 
      mathx10
      }{}
\DeclareSymbolFont{mathx}{U}{mathx}{m}{n}
\DeclareFontSubstitution{U}{mathx}{m}{n}
\DeclareMathAccent{\widecheck}{0}{mathx}{"71}

\def\ov#1{\overline{#1}}
\def\tn#1{\textnormal{#1}}
\def\sf#1{\textsf{#1}}
\def\BE#1{\begin{equation}\label{#1}}
\def\EE{\end{equation}}
\def\eref#1{(\ref{#1})}
\def\lr#1{\langle#1\rangle}
\def\blr#1{\big\langle#1\big\rangle}
\def\bblr#1{\bigg\langle#1\bigg\rangle}
\def\wc#1{\widecheck{#1}}
\def\wh#1{\widehat{#1}}
\def\wt#1{\widetilde{#1}}
\def\sm#1{\begin{small}#1\end{small}}

\def\lra{\longrightarrow}

\def\C{\mathbb C}
\def\E{\mathbb E}
\def\cH{\mathcal H}
\def\cM{\mathcal M}
\def\chM{\wc\cM}
\def\fM{\mathfrak M}
\def\cN{\mathcal N}
\def\cO{\mathcal O}
\def\P{\mathbb P}
\def\Q{\mathbb Q}
\def\R{\mathbb R}
\def\S{\mathbb S}
\def\T{\mathbb T}

\def\Th{\Theta}
\def\cU{\mathcal U}
\def\chU{\wc\cU}
\def\Z{\mathbb Z}
\def\cZ{\mathcal Z}

\def\al{\alpha}
\def\de{\delta}
\def\ep{\epsilon}
\def\io{\iota}
\def\ka{\kappa}
\def\la{\lambda}
\def\na{\nabla}
\def\om{\omega}
\def\si{\sigma}
\def\ups{\upsilon}
\def\vph{\varphi}
\def\ze{\zeta}

\def\De{\Delta}
\def\Ga{\Gamma}
\def\Si{\Sigma}

\def\fI{\mathfrak i}
\def\fj{\mathfrak j}
\def\s{\mathbf s}

\def\i{\infty}
\def\dbar{\bar\partial}
\def\eset{\emptyset}
\def\1{\mathbf 1}

\def\psb{\psi_{\textnormal{b}}}
\def\pst{\psi_{\textnormal{t}}}
\def\lb{\la_{\textnormal{b}}}
\def\lt{\la_{\textnormal{t}}}

\def\bfm{\mathbf m}
\def\nd{\textnormal{d}}
\def\ev{\textnormal{ev}}
\def\gl{\textnormal{gl}}
\def\GW{\textnormal{GW}}
\def\Hol{\textnormal{Hol}}
\def\hor{\textnormal{hor}}
\def\id{\textnormal{id}}
\def\Im{\textnormal{Im}}
\def\Obs{\textnormal{Obs}}
\def\ord{\textnormal{ord}}
\def\PD{\textnormal{PD}}
\def\pt{\textnormal{pt}}
\def\st{\textnormal{st}}
\def\vir{\tn{vir}}
\def\vrt{\tn{vrt}}

\begin{document}

\title{Absolute vs.~Relative Gromov-Witten Invariants}
\author{Mohammad F.~Tehrani and Aleksey Zinger\thanks{Partially supported by NSF grant 0846978}}
\date{\small May 3, 2014. Updated: \today}
\maketitle

\begin{abstract}

\noindent
In light of recent attempts to extend the Cieliebak-Mohnke approach for
constructing Gromov-Witten invariants to positive genera,
we compare the absolute and relative Gromov-Witten invariants
of compact symplectic manifolds 
when the symplectic hypersurface contains no relevant holomorphic curves.
We show that these invariants are then the same, except in 
a narrow range of dimensions of the target and genera of the domains,
and provide examples when they fail to be the same.
\end{abstract}

\tableofcontents

\section{Introduction}
\label{intro_sec}

\noindent
Gromov-Witten invariants of a compact symplectic manifold $(X,\om)$
are certain, often delicate, counts of $J$-holomorphic curves in~$X$;
they play prominent roles in symplectic topology, algebraic geometry, and 
string theory.
For a \textsf{symplectic hypersurface}~$V$ in~$(X,\om)$,
i.e.~a closed symplectic submanifold of real codimension~2,
relative Gromov-Witten invariants of~$(X,\om,V)$ count $J$-holomorphic curves in~$X$ 
with specified contacts with~$V$.
If $V$ contains no (non-constant) $J$-holomorphic curves that could possibly contribute
to a specific absolute invariant of~$X$, 
one could hope that such an absolute invariant equals the corresponding relative invariant
with the basic contact condition, divided by the number of orderings of the contact points.
We show that this is indeed the case, except in a narrow range of dimensions 
of the target and genera of the domains; 
see Theorem~\ref{main_thm} and Remarks \ref{mainthm_rmk}-\ref{IPcond_rmk}.
Examples~\ref{deg0_eg}-\ref{P4_eg} illustrate the three cases when 
the absolute and relative invariants can fail to be equal.
We also overview the geometric approach to constructing genus~0 Gromov-Witten invariants
suggested in~\cite{CM} and the attempts to extend this approach to higher genera
in~\cite{G} and~\cite{IPvfc}.\\

\noindent
For $g,k\!\in\!\Z^{\ge0}$,  we denote by $\ov\cM_{g,k}$ the Deligne-Mumford moduli space 
of stable $k$-marked genus~$g$ connected nodal curves.
If $2g\!+\!k\!<\!3$, $\ov\cM_{g,k}$ is empty with this definition, though 
it is often convenient to formally take it to be a point in these cases,
as done when we set up notation for GW-invariants below.
If $g,k\!\in\!\Z^{\ge0}$, $A\!\in\!H_2(X;\Z)$, and  
$J$ is an an almost complex structure on~$X$ compatible with (or tamed by)~$\om$,
let $\ov\fM_{g,k}(X,A)$ denote the moduli spaces of stable $J$-holomorphic $k$-marked maps 
from connected nodal curves of genus~$g$. 
If in addition $V\!\subset\!X$ is a symplectic hypersurface, 
$\s\!\equiv\!(s_1,\ldots,s_{\ell})$ is an $\ell$-tuple of positive integers such~that 
\BE{bsumcond_e} s_1+\ldots+s_{\ell}=A\cdot V,\EE
and $J$ is compatible with~$V$ in a suitable sense, let $\ov\fM_{g,k;\s}^V(X,A)$ 
denote the moduli spaces of  stable $J$-holomorphic $(k\!+\!\ell)$-marked maps 
from connected nodal curves of genus~$g$ 
that have contact with~$V$ at the last $\ell$ marked points of orders $s_1,\ldots,s_{\ell}$,
respectively.
These moduli spaces are introduced in \cite{LR,IPrel,Jun1} under certain assumptions on~$J$
and reviewed in Section~\ref{RelDfn_sec}.
The expected dimensions of these two moduli spaces are given~by
\BE{virdim_e}\begin{split}
\dim^{\vir}\ov\fM_{g,k}(X,A)&=2\big(\lr{c_1(X),A}+(n\!-\!3)(1\!-\!g)+k\big),\\
\dim^{\vir}\ov\fM_{g,k;\s}^V(X,A)&=2\big(\lr{c_1(X),A}+(n\!-\!3)(1\!-\!g)+k+\ell(\s)\!-\!|\s|\big),
\end{split}\EE
where $\ell(\s)\!\equiv\!\ell$  and $|\s|\!\equiv\!s_1\!+\!\ldots\!+\!s_{\ell}$.
In particular, these dimensions are the same~if 
$$\s=\1_{\ell}\equiv\big(\underset{\ell}{\underbrace{1,\ldots,1}}\big)\,, $$
i.e.~the tuple~$\s$ imposes no contact conditions on degree~$A$ $J$-holomorphic curves,
beyond what a generic such curve can be expected to satisfy.\\

\noindent 
For each $i\!=\!1,\ldots,k$, let 
\BE{evidfn_e}\ev_i\!: \ov\fM_{g,k}(X,A),\ov\fM_{g,k;\s}^V(X,A) \lra X\EE
be the $i$-th evaluation map.
It sends the equivalence class of a  $J$-holomorphic map $u\!:\Si\!\lra\!X$
from a genus~$g$ nodal curve~$\Si$ to $u(x_i)\!\in\!X$, where  
$x_i\!\in\!\Si$ is the $i$-th marked point.
Let
\BE{stdfn_e}\st\!: \ov\fM_{g,k}(X,A),\ov\fM_{g,k;\s}^V(X,A)\lra \ov\cM_{g,k}\EE
denote the forgetful morphism to the Deligne-Mumford space.
If $2g\!+\!k\!\ge\!3$, 
it sends the equivalence class of a $J$-holomorphic map $u\!:\Si\!\lra\!X$
from a marked genus~$g$ nodal curve~$\Si$ to 
the equivalence class of the stable $k$-marked genus~$g$ nodal curve~$\Si'$
obtained from $(\Si,x_1,\ldots,x_k)$ by contracting the unstable components
(spheres with one or two special, i.e.~nodal or marked, points);
see Figure~\ref{st_fig}.\\

\begin{figure}
\begin{pspicture}(-4,-5)(10,0)
\psset{unit=.3cm}
\psellipse[linewidth=.08](0,-4)(5,2)\pscircle[linewidth=.08](7,-4){2}  
\pscircle*(5,-4){.25}
\psarc[linewidth=.08](-2.5,-7){3.16}{65}{115}
\psarc[linewidth=.08](-2.5,-1){3.16}{245}{295}
\psarc[linewidth=.08](2.5,-7){3.16}{65}{115}
\psarc[linewidth=.08](2.5,-1){3.16}{245}{295}
\rput(5,-2){$u$}\rput(15,-3){$\st$}
\psline[linewidth=.12]{->}(12,-4)(18,-4)
\psellipse[linewidth=.08](25,-4)(5,2)
\psarc[linewidth=.08](22.5,-7){3.16}{65}{115}
\psarc[linewidth=.08](22.5,-1){3.16}{245}{295}
\psarc[linewidth=.08](27.5,-7){3.16}{65}{115}
\psarc[linewidth=.08](27.5,-1){3.16}{245}{295}
\psellipse[linewidth=.08](0,-14)(5,2)\pscircle[linewidth=.08](7,-14){2}  
\pscircle[linewidth=.08](7,-10){2} 
\pscircle*(5,-14){.25}
\pscircle*(7,-12){.25}\pscircle*(5.59,-8.59){.25}\pscircle*(8.41,-8.59){.25}  
\rput(4.7,-8.3){$x_1$}\rput(9.3,-8.3){$x_2$}
\psarc[linewidth=.08](-2.5,-17){3.16}{65}{115}
\psarc[linewidth=.08](-2.5,-11){3.16}{245}{295}
\psarc[linewidth=.08](2.5,-17){3.16}{65}{115}
\psarc[linewidth=.08](2.5,-11){3.16}{245}{295}
\rput(5,-12){$u$}\rput(15,-13){$\st$}
\psline[linewidth=.12]{->}(12,-14)(18,-14)
\psellipse[linewidth=.08](25,-14)(5,2)\pscircle[linewidth=.08](32,-14){2}  
\pscircle*(30,-14){.25}\pscircle*(33.41,-12.59){.25}\pscircle*(33.41,-15.41){.25}
\rput(34.2,-12.2){$x_1$}\rput(34.4,-15.8){$x_2$}
\psarc[linewidth=.08](22.5,-17){3.16}{65}{115}
\psarc[linewidth=.08](22.5,-11){3.16}{245}{295}
\psarc[linewidth=.08](27.5,-17){3.16}{65}{115}
\psarc[linewidth=.08](27.5,-11){3.16}{245}{295}
\end{pspicture}
\caption{Examples of the stabilization morphism~\eref{stdfn_e}.}
\label{st_fig}
\end{figure}
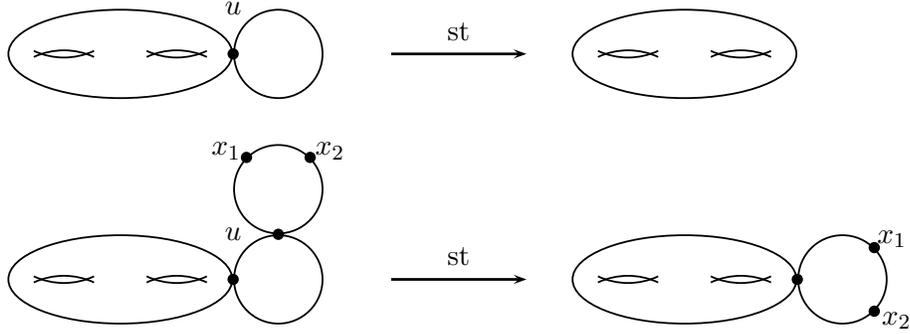

\noindent
Along with the virtual class for $\ov\fM_{g,k}(X,A)$,
constructed in \cite{RT2} in ``semi-positive" cases,
in \cite{BF} in the algebraic case, and in~\cite{FO,LT} in the general case, 
the first morphisms in~\eref{evidfn_e} and~\eref{stdfn_e} give rise
to the (\sf{absolute}) \sf{GW-invariants} of~$(X,\om)$:
\BE{absGWdfn_e}
\GW_{g,A}^X(\ka;\al_1,\ldots,\al_k)
\equiv\bblr{\!\st^*\ka\prod_{i=1}^k\!\ev_i^*\al_i,[\ov\fM_{g,k}(X,A)]^{\vir}\!}
\quad\forall\,\ka\!\in\!H^*(\ov\cM_{g,k}),\,\al_i\!\in\!H^*(X),\EE
where $H^*$ denotes the cohomology with $\Q$-coefficients.
The number above vanishes unless
\BE{AbsCond_e}\deg\ka+\sum_{i=1}^k\deg\al_i=
2\big(\lr{c_1(X),A}+(n\!-\!3)(1\!-\!g)+k\big)\,.\EE
Along with the virtual class for $\ov\fM_{g,k;\s}^V(X,A)$, 
the second morphisms in~\eref{evidfn_e} and~\eref{stdfn_e} give rise
to the \sf{relative GW-invariants} of~$(X,\om,V)$:
\BE{relGWdfn_e}
\GW_{g,A;\s}^{X,V}(\ka;\al_1,\ldots,\al_k)
\equiv\bblr{\!\st^*\ka\prod_{i=1}^k\!\ev_i^*\al_i,[\ov\fM_{g,k;\s}^V(X,A)]^{\vir}\!}
\quad\forall\,\ka\!\in\!H^*(\ov\cM_{g,k}),\,\al_i\!\in\!H^*(X).\EE
Such a virtual class is constructed in \cite{IPrel} in ``semi-positive" cases
and in \cite{Jun1} in the algebraic case and is used in~\cite{LR} in the general case;
see Section~\ref{RelDfn_sec} for more details.
The number in~\eref{relGWdfn_e} vanishes unless
$$\deg\ka+\sum_{i=1}^k\deg\al_i=
2\big(\lr{c_1(X),A}+(n\!-\!3)(1\!-\!g)+k+\ell(\s)\!-\!|\s|\big)\,.$$
The numbers in~\eref{absGWdfn_e} and~\eref{relGWdfn_e} are (graded-) symmetric 
and linear in the inputs~$\al_i$.
By the latter property, they give rise to well-defined numbers
$$\GW_{g,A}^X(\ka;\al),\GW_{g,A;\s}^{X,V}(\ka;\al)\in\Q
\qquad\forall \,\ka\!\in\!H^*(\ov\cM_{g,k}),\,\al\!\in\!H^*(X)^{\otimes k}\,.$$
The numbers 
$$\GW_{g,A}^X(\al)\equiv \GW_{g,A}^X(1;\al)
\qquad\hbox{and}\qquad
\GW_{g,A;\s}^{X,V}(\al)=\GW_{g,A;\s}^{X,V}(1;\al)$$
are called \sf{primary} GW-invariants or \sf{GW-invariants with primary insertions}.
In some cases, the numbers~\eref{absGWdfn_e} and~\eref{relGWdfn_e} can be described
as signed counts of concrete geometric objects, 
$J$-holomorphic or $(J,\nu)$-holomorphic maps; see 
Sections~\ref{RelDfn_sec} and~\ref{CM_sec}. 

\begin{rmk}\label{psiclass_rmk}
The numbers~\eref{absGWdfn_e} and~\eref{relGWdfn_e} do not cover 
GW-invariants that arise from natural classes on  $\ov\fM_{g,k}(X,A)$ and 
$\ov\fM_{g,k;\s}^V(X,A)$,
such as $\psi$-classes (which are generally different from the $\psi$-classes on $\ov\cM_{g,k}$
pulled back by the morphism~\eref{stdfn_e}) and the euler classes of obstruction bundles of various kinds;
both types of classes are central to GW-theory.
The geometric constructions of the numbers~\eref{absGWdfn_e} and~\eref{relGWdfn_e} 
reviewed in Sections~\ref{RelDfn_sec} and~\ref{CM_sec} are not compatible with 
such classes.
\end{rmk}

\begin{dfn}\label{hollow_dfn}
Let $(X,\om)$ be a compact symplectic manifold, $g\!\in\!\Z^{\ge0}$, and $A\!\in\!H_2(X;\Z)$.
A~symplectic hypersurface $V\!\subset\!X$ is \sf{$(g,A)$-hollow}
if there exists an $\om|_V$-tame almost complex structure~$J_V$ on~$V$ such that 
every non-constant $J_V$-holomorphic map 
$u\!:\Si\!\lra\!V$ from a smooth connected Riemann surface~$\Si$ satisfies
$$g(\Si)>g, \qquad\hbox{or}\qquad \lr{u^*\om,\Si}\!>\!\om(A), \qquad\hbox{or}\qquad 
\lr{u^*\om,\Si}=\om(A),~u_*[\Si]\!\neq\!A.$$
\end{dfn}

\begin{thm}\label{main_thm}
Suppose $(X,\om)$ is a compact symplectic manifold of real dimension~$2n$, 
$g,k\!\in\!\Z^{\ge0}$, \hbox{$A\!\in\!H_2(X;\Z)$}, and
$V\!\subset\!X$ is a $(g,A)$-hollow symplectic hypersurface such that $A\!\cdot\!V\!\ge\!0$.
If 
\BE{AReqcond_e}(g,A)\neq (1,0) \qquad\hbox{and}\qquad (n\!-\!5)g(g\!-\!1)\ge0\,,\EE
then the absolute GW-invariants~\eref{absGWdfn_e} and
the basic corresponding relative GW-invariants~\eref{relGWdfn_e} agree: 
\BE{AReq_e}
\GW_{g,A}^X(\ka;\al)=\frac{1}{(A\cdot V)!}\GW_{g,A;\1_{A\cdot V}}^{X,V}(\ka;\al)
\qquad\forall~\ka\in H^*(\ov\cM_{g,k}),~\al\in H^*(X)^{\otimes k}\,.\EE
This identity also holds if  $\ka\!=\!1$, $A\!\neq\!0$, and either $g\!=\!2$ or $n\!\neq\!4$. 
\end{thm}

\begin{rmk}\label{mainthm_rmk}
By Theorem~\ref{main_thm}, the absolute GW-invariants with primary insertions, i.e.~$\ka\!=\!1$,
and the corresponding relative invariants in degree $A\!\neq\!0$ may fail to
be equal only if $n\!=\!4$ and $g\!\ge\!3$ at the same time;
the possibility of such a failure is illustrated by Example~\ref{P4_eg}.
With non-trivial constraints~$\ka$, the two invariants in degree $A\!\neq\!0$ may fail to
be equal only if $1\!\le\!n\!\le\!4$ and $g\!\ge\!2$ at the same time;
the possibility of such a failure is illustrated by Example~\ref{P1_eg}.
Example~\ref{deg0_eg}, which is motivated by \cite[Example~12.5]{IPsum},
illustrates the possibility of failure of~\eref{AReq_e} with $A\!=\!0$.
\end{rmk}

\begin{rmk}\label{AG_rmk}
In Section~\ref{mainpf_sec}, we give two versions of essentially the same proof
of Theorem~\ref{main_thm}.
The~first version is a direct comparison of the two invariants.
It is particularly suitable for considering the independence of
the geometrically constructed curve counts 
of the chosen Donaldson divisor in \cite{CM,G,IPvfc}; see Section~\ref{CM_sec}.
The argument involves several cases; in all, but one of them, 
the conclusion is established by a dimension-counting argument.
In the exceptional case, when $\ka\!=\!1$, $n\!=\!3$, and $g\!\ge\!3$,
we also use the fact that $\la_g^2\!=\!0$ on~$\ov\cM_g$; see \cite[(5.3)]{Mumford}.
The second version of the proof is a formal application of the symplectic sum formula for 
GW-invariants, as in the setup introduced in \cite[Section~2.2]{MP}, 
successfully applied in the genus~0 case in~\cite{HLR},
and used in the attempted proof of \cite[Theorem~11.1]{IPvfc}.
As indicated by~\cite{IPvfc}, establishing~\eref{AReq_e} in this way 
leads to the analogue of~\eref{AReq_e} for virtual classes,
at least in the algebraic category; see Section~\ref{VMext_subs}.
The crucial step in the proof is that we start by taking a generic regularization for
the maps to~$V$, i.e.~a horizontal deformation of the parameters~$(J,\nu)$ along~$V$,
before deforming the parameter~$\nu$ in the normal direction to~$V$. 
The order of the deformations is reversed in \cite[Section~12]{IPvfc},
which makes the horizontal directions not even defined and 
crucially misses out the opportunity to quickly settle most cases
of Theorem~\ref{mainthm_rmk}.
The argument in \cite[Section~11]{IPvfc}
instead misinterprets \cite[(9)]{FP} to arrive at the conclusion of 
Lemma~\ref{relGW_lmm} in Section~\ref{VMext_subs} without the restrictions in~\eref{AReqcond_e2}
and the conclusion of Corollary~\ref{VC_crl} without the restrictions in~\eref{AReqcond_e3}
or the projective assumptions on~$X$ and~$A$.
\end{rmk}

\begin{rmk}\label{IPcond_rmk}
It is sufficient to verify the condition of Definition~\ref{hollow_dfn}
for $J$-holomorphic maps \hbox{$u\!:\Si\!\lra\!V$} that are simple in
the sense of \cite[Section~2.5]{MS2}.
By \cite[Section~3.2]{MS2}, moduli spaces of such maps have the expected dimensions
for a generic $\om_V$-tame (or compatible) almost complex structure~$J_V$ on~$V$.
Thus, by the first equation in~\eref{virdim_e}, $V$ is $(g,A)$-hollow if 
\BE{IPcond_e1}A'\cdot V> \lr{c_1(X),A'}+(n\!-\!4)(1\!-\!g')\EE
for all $g'\!\in\!\Z^{\ge0}$ with $g'\!\le\!g$ and $A'\!\in\!H_2(X;\Z)$ with 
$\om(A')\!\le\!\om(A)$ such that $A'$ can be represented by 
a $J_0$-holomorphic curve for some fixed $\om$-tame almost complex structure~$J_0$ on~$X$. 
By Gromov's Compactness Theorem, the number of such classes $A'$ is finite.
Since $\om(A')\!>\!0$ for all such classes, \eref{IPcond_e1} can be achieved by 
taking~$V$ to be Poincare dual to a sufficiently high multiple of 
a rational symplectic form close to~$\om$.
Such~$V$, called \sf{Donaldson hypersurfaces}, always exist by~\cite{D} and are central to 
the construction of genus~0 curve counts in~\cite{CM} and its attempted extensions
to positive genera in~\cite{G} and~\cite{IPvfc}; see Section~\ref{CM_sec}.
\end{rmk}

\begin{rmk}\label{hollowdfn_rmk}
For the purposes of the direct proof of Theorem~\ref{main_thm} in Section~\ref{DirComp_subs},
it is sufficient to assume that there exist an almost complex structure~$J_V$ on~$V$
and an arbitrarily small perturbation~$\nu$ on~$V$ as in Section~\ref{RelDfn_sec} so 
that every $(J_V,\nu)$-holomorphic map $u\!:\Si\!\lra\!V$ 
from a smooth connected Riemann surface~$\Si$ satisfies 
$$u_*[\Si]=0, \quad\hbox{or}\quad   g(\Si)>g,  \quad\hbox{or}\quad 
\lr{u^*\om,\Si}\!>\!\om(A), \quad\hbox{or}\quad 
\lr{u^*\om,\Si}=\om(A),~u_*[\Si]\!\neq\!A.$$
For the purposes of the proof of Theorem~\ref{main_thm} via the symplectic sum formula
in Section~\ref{SympSum_subs}, it is sufficient to assume the GW-invariants of~$V$
of genus~$g'$ and in the class~$A'$ vanish whenever $A'\!\neq\!0$, $g'\!\le\!g$,
and $\om(A')\!\le\!\om(A)$.  
\end{rmk}

\noindent
The next three examples illustrate different cases when~\eref{AReq_e} fails to hold. 
They are justified in Section~\ref{eg_sec}.

\begin{eg}\label{deg0_eg}
Suppose $(X,\om)$ is a compact symplectic manifold of real dimension~$2n$ and
$V\!\subset\!X$ is a symplectic hypersurface.
Let $\fj\!\in\!H^2(\ov\cM_{1,1})$ be the Poincare dual of a generic point and $\al\!\in\!H^2(X)$.
The genus~1 degree~0 GW-invariants of~$(X,\om)$ and~$(X,\om,V)$ satisfy
\begin{gather}
\label{deg0_e1}
\GW_{1,0}^X(\fj;1)=\frac{\chi(X)}{2}
= \frac{1}{0!}\GW_{1,0;()}^{X;V}(\fj;1)+\frac{\chi(V)}{2}\,,\\
\label{deg0_e2}
\GW_{1,0}^X(\al)=-\frac{\lr{\al\,c_{n-1}(X),X}}{24}= 
\frac{1}{0!}\GW_{1,0;()}^{X,V}(\al)-\frac{\lr{\al|_V\,c_{n-2}(V),V}}{24},
\end{gather}
where $\chi(\cdot)$ is the euler characteristic and $()$ in the subscript is the length~0
contact vector (and thus gives $0!$ in the denominators).
\end{eg}

\begin{eg}\label{P1_eg}
Denote by $\P^1$ the one-dimensional complex projective space with 
the standard symplectic form and by $V_{\de}\!\subset\!\P^1$ 
the symplectic hypersurface consisting of $\de\!\in\!\Z^{\ge0}$ distinct points.
Let $\pt\!\in\!H^2(\P^1)$ be the Poincare dual of a point and 
$\ka\!\in\!H^2(\ov\cM_{2,2})$ be the Poincare dual of the divisor
whose generic element consists of two components, one of genus~2 and the other of genus~0;
see the bottom right diagram in Figure~\ref{st_fig}.
The genus~2 degree~1 GW-invariants of~$\P^1$ and~$(\P^1,V_{\de})$ satisfy
\BE{P1_e}
\frac{1}{240} =\GW_{2,1}^{\P^1}(\ka^4;\pt,\pt)=
\frac{1}{\de!}\GW_{2,1;\1_{\de}}^{\P^1,V_{\de}}(\ka^4;\pt,\pt)+\frac{\de}{1,152}\,.\EE
\end{eg}

\begin{eg}\label{P4_eg}
Denote by $\P^4$ the four-dimensional complex projective space with 
the standard symplectic form and by $V_{\de}\!\subset\!\P^4$ 
a smooth complex hypersurface of degree~$\de$.
Let $\pt\!\in\!H^8(\P^4)$ be the Poincare dual of a point.
The genus~3 degree~1 primary GW-invariants of~$\P^4$ and~$(\P^4,V_{\de})$ satisfy
\BE{P4_e}
 -\frac{37}{82,944}
=\GW_{3,1}^{\P^4}(\pt)=
\frac{1}{\de!}\GW_{3,1;\1_{\de}}^{\P^4,V_{\de}}(\pt)+
\frac{\de(\de^2\!-\!5\de\!+\!8)}{72,576}\,.\EE
\end{eg}

\begin{rmk}\label{P4eg_rmk}
The proof in Section~\ref{P4eg_subs} of the second equality in~\eref{P4_e}  applies to 
primary GW-invariants of~$\P^4$ and~$(\P^4,V_{\de})$ in degree~$d$
as long as $V_{\de}$ contains no curves of genus at most~3 and 
degree at most~$d$ that pass through the constraints. 
In these cases, the last term in~\eref{P4_e} should be multiplied 
by the genus~0 degree~$d$ absolute GW-invariant with an extra point insertion.
The condition on~$V_{\de}$ in particular excludes the $d\!=\!1$ GW-invariants 
with primary insertions~$(\P^1,\P^2)$ if $\de\!=\!1$.
\end{rmk}

\noindent
We review the definitions of absolute and relative invariants in Section~\ref{RelDfn_sec},
focusing on the geometric differences for the requirements on generic parameters~$(J,\nu)$
determining the two types of invariants.
These differences are fundamental to establishing Theorem~\ref{main_thm} in 
Section~\ref{mainpf_sec} and the claims of Examples \ref{deg0_eg}-\ref{P4_eg}
in Section~\ref{eg_sec}.
In Section~\ref{CM_sec}, we review the Cieliebak-Mohnke approach to 
constructing GW-invariants and relate a key issue in this approach 
to Theorem~\ref{main_thm} and Examples \ref{deg0_eg}-\ref{P4_eg}.\\

\noindent
This note was inspired by the discussions regarding \cite[Theorem~11.1]{IPvfc}
and the related aspects of~\cite{G} at and following the SCGP workshop
on constructing the virtual cycle in GW-theory.
We would like to thank the SCGP for organizing and hosting this very 
enlightening workshop and the authors of~\cite{IPvfc} and~\cite{G}
for bringing up important questions concerning relative GW-invariants.
We are also grateful to C.-C.~Liu and D.~Maulik for sharing invaluable insights 
on  \cite[Theorem~11.1]{IPvfc}
and C.~Faber for providing intersection numbers for Deligne-Mumford moduli spaces.

\section{Review of GW-invariants}
\label{RelDfn_sec}

\noindent
Let $g,k\!\in\!\Z^{\ge0}$ be such that $2g\!+\!k\!\ge\!3$, 
\BE{PrymCov_e}\chM_{g,k} \lra \ov\cM_{g,k} \EE
be the branched cover of the Deligne-Mumford space of stable $k$-marked genus~$g$ 
curves by the associated moduli space of Prym structures constructed in~\cite{Lo}, and 
$$\pi_{g,k}\!:\chU_{g,k}\!\lra\!\chM_{g,k}$$
be the corresponding universal curve.
A \textsf{$k$-marked genus~$g$  nodal curve with a Prym structure} is 
a connected compact nodal $k$-marked Riemann surface $(\Si,z_1,\ldots,z_k)$ of arithmetic genus~$g$
together with a holomorphic map $\st_{\Si}\!:\Si\!\lra\!\chU_{g,k}$
which surjects on a fiber of~$\pi_{g,k}$ and takes the marked points of~$\Si$
to the corresponding marked points of the fiber.\\

\noindent
If  $J$ is an almost complex structure on a smooth manifold~$X$,
$A\!\in\!H_2(X;\Z)$, and 
\BE{nucond0_e}\nu \in \Ga_{g,k}(X,J)\equiv 
\Ga\big(\chU_{g,k}\!\times\!X,\pi_1^*(T^*\chU_{g,k})^{0,1}\!\otimes_{\C}\!\pi_2^*(TX,J)\big),\EE
a \textsf{$k$-marked genus~$g$  degree~$A$ $(J,\nu)$-map} is a tuple $(\Si,z_1,\ldots,z_k,\st_{\Si},u)$
such that $(\Si,z_1,\ldots,z_k,\st_{\Si})$ is a genus~$g$ $k$-marked nodal curve with a Prym structure
and $u\!:\Si\!\lra\!X$ is a smooth (or $L^p_1$, with $p\!>2$) map such that 
$$u_*[\Si]=A \qquad\hbox{and}\qquad
\dbar_{J,\fj}u\big|_z\equiv \frac12\big(\nd u+J\circ \nd u\circ\fj\big)=
\nu(\st_{\Si}(z),u(z))\quad\forall\,z\!\in\!\Si,$$
where $\fj$ is the complex structure on~$\Si$.
Two such tuples are \textsf{equivalent} if they differ by a reparametrization
of the domain commuting with the maps to~$\chU_{g,k}$.\\

\noindent
Suppose $(X,\om)$ is a compact symplectic manifold and $J$ is 
an $\om$-tame almost complex structure.
By \cite[Corollary~3.9]{RT2}, the space $\ov\fM_{g,k}(X,A;J,\nu)$ of 
equivalence classes~of $k$-marked  genus~$g$ degree~$A$  $(J,\nu)$-maps
is Hausdorff and compact  in Gromov's convergence topology.
By \cite[Theorem~3.16]{RT2}, for a generic~$\nu$
each stratum of $\ov\fM_{g,k}(X,A;J,\nu)$ consisting of simple (not multiply covered) 
maps of a fixed combinatorial type is a smooth manifold of the expected even dimension, 
which is less than the expected dimension of 
the subspace of simple maps with smooth domains
(except for this subspace itself).
By \cite[Theorem~3.11]{RT2}, the last stratum has a canonical orientation.
By \cite[Proposition~3.21]{RT2}, 
the images of  the strata of $\ov\fM_{g,k}(X,A;J,\nu)$ consisting of multiply covered maps 
under the morphism
\BE{RTcyc_e}
\st\!\times\!\ev_1\!\times\!\ldots\!\times\!\ev_k\!:
\ov\fM_{g,k}(X,A;J,\nu)\lra \ov\cM_{g,k}\!\times\!X^k\EE
are contained in images of maps from smooth even-dimensional manifolds of
dimension less than this stratum
if $\nu$ is generic and $(X,\om)$ is semi-positive in the sense of 
\cite[Definition~6.4.1]{MS2}.
Thus, \eref{RTcyc_e} is a pseudocycle.
Intersecting it with generic representatives for the Poincare duals of
the classes~$\ka$ and~$\al_i$ and dividing by the order of the covering~\eref{PrymCov_e}, 
we obtain the (absolute) GW-invariants~\eref{absGWdfn_e} 
of a semi-positive symplectic manifold $(X,\om)$ in the stable range, 
i.e.~with $(g,k)$ such that $2g\!+\!k\!\ge\!3$.
If $g\!=\!0$, the same reasoning applies with $\nu\!=\!0$ and yields
the same conclusion if $(X,\om)$ satisfies a slightly stronger condition
($c_1(A)\!>\!0$ instead of $c_1(A)\!\ge\!0$ in \cite[Definition~6.4.1]{MS2}).
For general symplectic manifolds~$(X,\om)$, the GW-invariants~\eref{absGWdfn_e} 
are defined in \cite{FO,LT} using Kuranishi structures (or finite-dimensional approximations)
and local perturbations~$\nu$ as in~\eref{nucond0_e}.\\

\noindent
Suppose in addition $V\!\subset\!X$ is a closed symplectic hypersurface and $J(TV)\!=\!TV$.
Thus, $J$ induces a complex structure~$\fI_{X,V}$ on (the fibers~of) 
the normal bundle
$$\pi_{X,V}\!: \cN_XV\equiv TX|_V\big/TV\lra V.$$
A connection~$\na^{\cN_XV}$ in $(\cN_XV,\fI_{X,V})$ induces a splitting of the exact sequence
\BE{NXVsplit_e}\begin{split}
0&\lra \pi_{X,V}^*\cN_XV\lra T(\cN_XV)
\stackrel{\nd\pi_{X,V}}{\lra} \pi_{X,V}^*TV\lra0
\end{split}\EE
of vector bundles over~$\cN_XV$ which restricts to the canonical splitting
over the zero section and is preserved by the multiplication by~$\C^*$;
see \cite[Lemma~1.1]{anal}.
For each trivialization 
$$\cN_XV|_U\approx U\!\times\!\C$$ 
over an open subset~$U$ of~$V$, there exists 
$\al\in \Ga(U;T^*V\!\otimes_{\R}\!\C)$
such that the image of $\pi_{X,V}^*TV$ corresponding to this splitting is given~by
$$T_{(x,w)}^{\hor}(\cN_XV)=\big\{(v,-\al_x(v)w)\!:\,v\!\in\!T_xV\big\}
\qquad\forall~(x,w)\!\in\!U\!\times\!\C.$$
The isomorphism $(x,w)\!\lra\!(x,w^{-1})$ of $U\!\times\!\C^*$ maps this vector space~to
\begin{equation*}\begin{split}
T_{(x,w^{-1})}^{\hor}\big((\cN_XV)^*\big)
&=\big\{(v,w^{-2}\al_x(v)w)\!:\,v\!\in\!T_xV\big\}\\
&=\big\{(v,\al_x(v)w^{-1})\!:\,v\!\in\!T_xV\big\}
\qquad\forall~(x,w)\!\in\!U\!\times\!\C^*.
\end{split}\end{equation*}
Thus, the splitting of \eref{NXVsplit_e} induced by a connection in~$(\cN_XV,\fI_{X,V})$ 
extends to a splitting of the exact sequence 
$$0\lra T^{\vrt}(\P_XV) \lra T(\P_XV)
\stackrel{\nd\pi_{X,V}}{\lra} \pi_{X,V}^*TV\lra0,$$
where  
\BE{PXVdfn_e}\pi_{X,V}\!:\P_XV\equiv \P\big(\cN_XV\oplus V\!\times\!\C\big)  \lra V\,;\EE
this splitting restricts to the canonical splittings over 
\BE{PXVsecdfn_e}\P_{X,\i}V\equiv\P(\cN_XV\!\oplus\!0)
\qquad\hbox{and}\qquad \P_{X,0}V\equiv\P(0 \oplus X\!\times\!\C)\EE
and is preserved by the multiplication by~$\C^*$.
Via this splitting, the almost complex structure $J_V\!\equiv\!J_X|_V$ and 
the complex structure $\fI_{X,V}$ in the fibers of~$\pi_{X,V}$ induce
an almost complex structure~$J_{X,V}$ on~$\P_XV$ which restricts to almost complex
structures on $\P_{X,\i}V$ and~$\P_{X,0}V$
and is preserved by the~$\C^*$-action.
Furthermore, the projection $\pi_{X,V}\!:\P_XV\!\lra\!V$ is $(J_V,J_{X,V})$-holomorphic.
By \cite[Lemma~2.2]{anal}, $\xi\!\in\!\Ga(V,\cN_XV)$ is 
$(J_{X,V},J|_V)$-holomorphic if and only if~$\xi$ lies in the kernel of 
the $\dbar$-operator on~$(\cN_XV,\fI_{X,V})$ corresponding to the connection used above.\\

\noindent
For each $m\!\in\!\Z^{\ge0}$, let 
\begin{gather}\label{XmVdfn_e} 
X_m^V=\big(X\sqcup\{1\}\!\times\!\P_XV\sqcup\ldots\sqcup
\{m\}\!\times\!\P_XV\big)/\!\!\sim\,,\qquad\hbox{where}\\
\notag
x\sim 1\!\times\!\P_{X,\i}V|_x\,,~~~
r\!\times\!\P_{X,0}V|_x\sim (r\!+\!1)\!\times\!\P_{X,\i}V|_x
\quad  \forall\,x\!\in\!V,~r=1,\ldots,m\!-\!1;
\end{gather}
see Figure~\ref{relcurve_fig}.
Define
$$q_m\!:X_m^V\lra X \qquad\hbox{by}\qquad
q_m(x)=\begin{cases}x,&\hbox{if}~x\!\in\!X;\\
\pi_{X,V}([v,w]),&\hbox{if}~x\!=\!(r,[v,w])\!\in\!r\!\times\!\P_XV\,.
\end{cases}$$
We denote by $J_m$ the almost complex structure on~$X_m^V$ so that 
$$J_m|_X=J \qquad\hbox{and}\qquad 
J_m|_{\{r\}\times\P_XV}=J_{X,V} \quad \forall~r=1,\ldots,m.$$
For each $(c_1,\ldots,c_m)\!\in\!\C^*$, define
\BE{Thdfn_e}\Th_{c_1,\ldots,c_m}\!:X_m^V\lra X_m^V \qquad\hbox{by}\quad
\Th_{c_1,\ldots,c_m}(x)=\begin{cases}x,&\hbox{if}~x\!\in\!X;\\
(r,[c_rv,w]),&\hbox{if}~x\!=\!(r,[v,w])\!\in\!r\!\times\!\P_XV.
\end{cases}\EE
This diffeomorphism is biholomorphic with respect to~$J_m$ and
preserves the fibers of the projection $\P_XV\!\lra\!V$
and the sections~$\P_{X,0}V$ and~$\P_{X,\i}V$.\\

\noindent
Suppose $J(TV)\!=\!TV$ and $J$ is $\om$-tame.
We denote by~$\na$ the Levi-Civita connection of the metric~$g_J$ on~$X$
determined by $(\om,J)$ as in \cite[(2.1.1)]{MS2},
by~$\wt\na$ the corresponding $J_X$-linear connection, as above \cite[(3.1.3)]{MS2},
and by~$\wh\na$ the connection given~by
$$\wh\na_v\ze = \wt\na_v\ze - \frac14\big\{ \na_{J\ze}J+J\na_{\ze}J\big\}(v)
\qquad\forall~\ze\in\Ga(X;TX),~v\in TX.$$
By the next paragraph, the $\dbar$-operator
$$\wh\na^{0,1}\!: \Ga(X;TX)\lra \Ga\big(X;T^*X^{0,1}\!\otimes_{\C}\!TX\big),
\qquad \ze\lra \frac12 \big(\na_{\cdot}\ze+J\na_{J\cdot}\ze\big)$$
restricts to an operator
$$\wh\na^{0,1}\!: \Ga(V;TV)\lra \Ga\big(V;T^*V^{0,1}\!\otimes_{\C}\!TV\big),$$
and thus descends to a $\dbar$-operator 
$$\Ga(V;\cN_XV)\lra \Ga\big(V;T^*V^{0,1}\!\otimes_{\C}\!\cN_XV\big)$$
corresponding to some connection 
$\na^{\cN_XV}$ in $(\cN_XV,\fI_{X,V})$; see \cite[Section~2.3]{anal}.
Let $J_{X,V}$ denote the complex structure on~$\P_XV$ induced by~$J_V$
and~$\na^{\cN_XV}$ as  in the paragraph above the previous one;
it depends only on the above $\dbar$-operator and not on 
the connection~$\na^{\cN_XV}$ realizing~it.\\

\noindent
If in addition $u\!:(\Si,\fj)\!\lra\!(X,J)$ is $(J,\fj)$-holomorphic, 
i.e.~$\dbar_{J,\fj}u\!=\!0$, the linearization of the $\dbar_{J,\fj}$-operator at~$u$ 
is given~by
\begin{gather}
D_u\!: \Ga(\Si,u^*TX)\lra \Ga^{0,1}_{J,\fj}(\Si;u^*TX)
\equiv\Ga\big(\Si,(T^*\Si,j)^{0,1}\!\otimes_{\C}\!u^*(TX,J)\big)
,\notag\\
\label{Dudfn_e}
D_u\xi=\frac12\big(\wh\na^u\xi+\{u^*J\}\circ\wh\na^u\xi\circ\fj)
+\frac14N_J^u(\xi,\nd u),
\end{gather}
where $\wh\na^u$ and $N_J^u$ are the pull-backs of the connection~$\wh\na$  
and of the Nijenhuis tensor~$N_J$ of~$J$ normalized as in \cite[p18]{MS2},
respectively, by~$u$; see \cite[(3.1.7)]{MS2}.
If in addition $u(\Si)\!\subset\!V$, 
$$D_u\big(\Ga(\Si,u^*TV)\big)\subset \Ga^{0,1}_{J,\fj}(\Si,u^*TV),$$
because the restriction of $D_u$ to $\Ga(\Si;u^*TV)$ is the linearization
of the $\dbar_{J,\fj}$-operator at~$u$  for the space of maps to~$V$.
Thus, $D_u$ descends to a first-order differential operator
\BE{DuNXV_e} D_u^{\cN_XV}\!: \Ga(\Si,u^*\cN_XV)\lra \Ga^{0,1}_{J,\fj}(\Si,u^*\cN_XV).\EE
By~\eref{Dudfn_e}, this operator is $\C$-linear if 
\BE{NijenCond_e}N_J(v,w)\in T_xV \qquad\forall~v,w\!\in\!T_xX,~x\!\in\!V.\EE
Under this assumption, $\xi\!\in\!\Ga(\Si,u^*\cN_XV)$ is a $(J_{X,V},\fj)$-holomorphic map
if and only if  \hbox{$\xi\!\in\!\ker D_u^{\cN_XV}$}.\\

\begin{figure}
\begin{pspicture}(38,-.8)(11,3.2)
\psset{unit=.3cm}
\psline[linewidth=.1](48,-2)(80,-2)
\rput(51,6.8){\sm{$V$}}\rput(51,5.2){\sm{$\P_{X,\i}V$}}
\rput(51,2.8){\sm{$\P_{X,0}V$}}\rput(51,1.2){\sm{$\P_{X,\i}V$}}
\rput(51,-1.2){\sm{$\P_{X,0}V$}}
\rput(46,8){$X$}\rput(46,4){$1\!\times\!\P_XV$}\rput(46,0){$2\!\times\!\P_XV$}
\psarc[linewidth=.04](60,9){3}{180}{0}\pscircle*(60,6){.3}
\psarc[linewidth=.04](64,9){1}{0}{180}
\psarc[linewidth=.04](68,9){3}{180}{0}\pscircle*(68,6){.3}
\psarc[linewidth=.04](60,-1){3}{0}{180}
\psarc[linewidth=.04](68,-1){3}{0}{180}
\psarc[linewidth=.04](64,-1){1}{180}{0}
\psarc[linewidth=.04](56,-1){1}{180}{0}
\psline[linewidth=.1](48,2)(80,2)
\pscircle*(57,-1){.3}\rput(57.9,-.8){\sm{$z_1$}}
\pscircle*(60,2){.3}\rput(56,-2.8){\sm{$z_2$}}
\pscircle*(64,-2){.3}\rput(64,-2.8){\sm{$z_3$}}
\pscircle*(68,2){.3}
\pscircle[linewidth=.04](60,4){2}
\psarc[linewidth=.04](68,4){2}{90}{270}\psarc[linewidth=.04](68,3){3}{30}{90}
\psarc[linewidth=.04](68,3){1}{270}{0}\psarc[linewidth=.04](70,3){1}{0}{180}
\psarc[linewidth=.04](72,-1){1}{180}{0}
\psline[linewidth=.1](48,6)(80,6)
\pscircle*(56,-2){.3}\pscircle*(72,-2){.3}
\rput(72,-2.8){\sm{$z_4$}}
\end{pspicture}
\caption{The image of a relative map with $k\!=\!1$ and $\s\!=\!(2,2,2)$
to the space $X_2^V$.}
\label{relcurve_fig}
\end{figure}
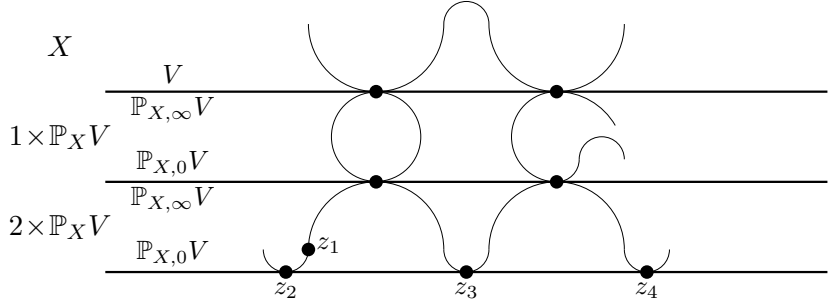

\noindent
If $J(TV)\!\subset\!V$, $\Si$ is a smooth connected Riemann surface, and
$u\!:\Si\!\lra\!X$ is a $J$-holomorphic map such that $u(\Si)\!\not\subset\!V$,
then $u^{-1}(V)$ is an isolated set of points~$z_i$;
see the beginning of \cite[Section~5.1]{SympSum}.
Furthermore, $u$ has a well-defined order of contact with~$V$ at
each $z_i\!\in\!u^{-1}(V)$, $\ord_{z_i}^Vu\!\in\!\Z^+$;
if $\Si$ is compact, 
$$\sum_{z_i\in u^{-1}(V)}\!\!\!\!\!\!\ord_{z_i}^Vu = u_*[\Si]\cdot V\,.$$  
If $A\!\in\!H_2(X;\Z)$, $g,k,\ell\!\in\!\Z^{\ge0}$, and
$\s\!=\!(s_1,\ldots,s_{\ell})\!\in\!(\Z^+)^{\ell}$  is a tuple satisfying~\eref{bsumcond_e},
let
\BE{relmoddfn_e}\fM_{g,k;\s}^V(X,A)\subset \ov\fM_{g,k+\ell}(X,A)\EE
denote the subset of equivalence of stable $J$-holomorphic maps~$u$ from 
marked genus~$g$  nodal curves $(\Si,z_1,\ldots,z_{k+\ell})$ such that 
$$u^{-1}(V)=\big\{z_{k+1},\ldots,z_{k+\ell}\big\}\
\qquad \hbox{and}\qquad 
\ord_{z_{k+i}}^Vu=s_i\quad\forall\,i=1,\ldots,\ell.$$
If $J$ satisfies~\eref{NijenCond_e}, we denote by
\BE{relmoddfn_e2}\ov\fM_{g,k;\s}^V(X,A)\supset \fM_{g,k;\s}^V(X,A)\EE
the space of equivalence classes of stable $J_{X,V}$-holomorphic maps $u\!:\Si\!\lra\!X_m^V$,
with $m\!\in\!\Z^{\ge0}$, from connected marked genus~$g$ nodal curves $(\Si,z_1,\ldots,z_{k+\ell})$
such~that the restriction of~$u$ to each irreducible component of~$X_m^V$ is contained 
in either~$X$ or in $\{r\}\!\times\!\P_XV$ for some $r\!=\!1,\ldots,m$,
but not in~$V$ or $\{r\}\!\times\!\P_{X,0}V$ for any~$r$,
$$q_{m*}u_*[\Si]=A,  \qquad \ord_{z_{k+i}}^{\{m\}\times\P_{X,0}V}u=s_i
\quad\forall~i=1,\ldots,\ell,$$
and the orders of contacts of the two branches at each node on
$V$, $\{r\}\times\P_{X,0}V$, or $\{r\}\times\P_{X,0}V$ agree;
see Figure~\ref{relcurve_fig}.
Two maps~$u$ as above are equivalent if they different by an isomorphism of marked domains
and a composition with an isomorphism~\eref{Thdfn_e};
see \cite[Section~4.2]{SympSum} for more details.\\

\noindent
The relative moduli spaces $\ov\fM_{g,k;\s}^V(X,A)$ are introduced in~\cite{LR} 
in a somewhat different formulation and under a stronger assumption on~$J$ 
than~\eref{NijenCond_e}, which essentially requires it to be given 
via the Symplectic Neighborhood Theorem \cite[Theorem~3.30]{MS1}
and makes the setup very amenable for the gluing needed to construct 
a virtual class.
In~\cite{IPrel}, the relative moduli spaces are re-introduced,
again in a somewhat different formulation from the previous paragraph,
with $\om$-compatible~$J$ satisfying~\eref{NijenCond_e}. 
The relative non-amenability of this setup with the gluing is not material
in cases when the relative invariants~\eref{relGWdfn_e} can be defined 
geometrically, as in the next paragraph.
By \cite[Section~3.2]{LR} and  \cite[Section~6]{IPrel}, the spaces 
$\ov\fM_{g,k;\s}^V(X,A)$ are compact; they are also Hausdorff.\\

\noindent
With notation as in~\eref{nucond0_e} and $J$ as in the previous two paragraphs, let
$$\Ga_{g,k}^V(X,J)\subset \Ga_{g,k}(X,J)$$
denote the subspace of elements~$\nu$ such~that 
\BE{nuVrestr_e}\nu|_{\wc\cU_{g,k}\times V}\in  \Ga_{g,k}(V,J|_V),
\quad
\wt\na_w\nu+J\wt\na_{Jw}\nu\in (T^*\wc\cU_{g,k})^{0,1}\!\otimes_{\C}\!T_xV 
\quad\forall~w\!\in\!T_xX,~x\!\in\!V.\EE
The first condition in~\eref{nuVrestr_e} insures that every $(J,\nu)$-holomorphic map
$u\!:\Si\!\lra\!X$ has well-defined order of contact with~$V$
at all points of~$u^{-1}(V)$ not contained
in an irreducible component of~$\Si$ mapped into~$V$.
The second condition in~\eref{nuVrestr_e} implies that the linearization 
of the $\dbar_{J,\fj}\!-\!\nu$ operator at $u\!:\Si\!\lra\!V$ induces a $\C$-linear map 
$$D_u^{\cN_XV}\!: \Ga(\Si,u^*\cN_XV)\lra \Ga^{0,1}_{J,\fj}(\Si,u^*\cN_XV)$$
for every $(J,\nu)$-holomorphic map $u\!:\Si\!\lra\!V$.
The moduli spaces
$$\fM_{g,k;\s}^V(X,A;J,\nu)\subset \ov\fM_{g,k;\s}^V(X,A;J,\nu)$$
can then be defined analogously to~\eref{relmoddfn_e} and~\eref{relmoddfn_e2}.
The component maps into the rubber layers $\{r\}\!\times\!\P_XV$ are then 
$(J_{X,V},\nu')$-holomorphic, with
\begin{gather*}
\nu'\in\Ga_{g',k'}(\P_XV,J), \\
\{\nu'|_w\}(v)=\big(\{\wt\na_w\nu\}(v),\nu(v)\big)\in
T_w^{\vrt}\cN_XV\oplus T_w^{\hor}\cN_XV
\qquad\forall~w\in\cN_XV,~v\in T\wc\cU_{g',k'}\,.
\end{gather*}
By the same reasoning as for~$J_{X,V}$, $\nu'$ given by the second line above 
extends over~$\P_{X,\i}V$, is $\C^*$-equivariant,
and satisfies~\eref{nuVrestr_e} with $(X,V)$ replaced by $(\P_XV,\P_{X,0}V)$
and $(\P_XV,\P_{X,\i}V)$.\\

\noindent
By \cite[Proposition~7.3]{IPrel}, the space $\ov\fM_{g,k;\s}^V(X,A;J,\nu)$ is compact.
By \cite[Lemma~7.5]{IPrel}, if $\nu$ is generic each stratum of $\ov\fM_{g,k;\s}^V(X,A;J,\nu)$
consisting of simple maps of a fixed combinatorial type is a smooth manifold of 
the expected even dimension, which is less than the expected dimension of 
the subspace of simple maps with smooth domains (except for this subspace itself).
By \cite[Theorem~7.4]{IPrel}, the last stratum has a canonical orientation.
As explained in \cite[Section~4.3]{SympSum}, the images of the strata of 
$\ov\fM_{g,k;\s}^V(X,A)$ consisting of multiply covered maps
under the morphism
\BE{IPcyc_e}
\st\!\times\!\ev_1\!\ldots\!\times\!\ev_k\!\times\!\ev_{k+1}\!\ldots\!\times\!\ev_{k+\ell}\!:
\ov\fM_{g,k;\s}^V(X,A;J,\nu)\lra \ov\cM_{g,k+\ell}\!\times\!X^k\!\times\! V^{\ell}\EE
are contained in images of maps from smooth even-dimensional manifolds of
dimension less than the main stratum
if $\nu$ is generic, subject to the conditions~\eref{NijenCond_e} and~\eref{nuVrestr_e}, 
$(V,\om|_V)$ is semi-positive, and $(X,\om,V)$ is semi-positive in the sense 
of \cite[Definition~4.7(1)]{SympSum}.
Such strata do not even exist if the domains of all elements of
$\ov\fM_{g,k;\s}^V(X,A)$ possibly contributing to 
the number~\eref{relGWdfn_e} are stable for some~$J$,
as happens in Section~\ref{CM_sec}.
By the proof of \cite[Proposition~8.2]{IPvfc}, all relevant domains are stable
for a generic~$J$ if
\BE{IPcond_e2} A'\cdot V\ge \lr{c_1(X),A'}+\frac12\dim_{\R}X+2g\EE
for all $A'\!\in\!H_2(X;\Z)$ with 
$\om(A')\!\le\!\om(A)$ such that $A'$ can be represented by 
a $J$-holomorphic curve. 
In the above cases, \eref{IPcyc_e} is thus a pseudocycle.
Intersecting it with generic representatives for the Poincare duals of
the cohomology classes~$\ka$ on~$\ov\cM_{g,k+\ell}$,
$\al_1,\ldots,\al_k$ on~$X$, and~$\al_{k+1},\ldots,\al_{k+\ell}$ on~$V$ 
and dividing by the order of the covering~\eref{PrymCov_e}, 
we obtain the relative GW-invariant
$$\GW_{g,k;\s}^{X,V}\big(\ka;\al_1,\ldots,\al_k;\al_{k+1},\ldots,\al_{k+\ell}\big)
=\GW_{g,k;\s}^{X,V}\big(\ka;\al_1\!\otimes\!\ldots\!\otimes\!\al_k;
\al_{k+1}\!\otimes\!\ldots\!\otimes\!\al_{k+\ell}\big).$$
The relative GW-invariant~\eref{relGWdfn_e} is the above invariant with~$\ka$
pulled back from $\ov\cM_{g,k}$ by the forgetful morphism from $\ov\cM_{g,k+\ell}$
and $\al_{k+i}\!=\!1$ for all $i\!=\!1,\ldots,\ell$. 
If $g\!=\!0$, the same reasoning applies with $\nu\!=\!0$ and yields the same conclusion
if $(X,\om,V)$ satisfies the slightly stronger condition of 
\cite[Definition~4.7(2)]{SympSum}.
For general triple~$(X,\om,V)$, the relative GW-invariants~\eref{relGWdfn_e} 
are defined similarly to \cite{FO,LT} using Kuranishi structures (or finite-dimensional approximations)
and local perturbations~$\nu$ as in~\eref{nuVrestr_e}.

\section{Proof of Theorem~\ref{main_thm}}
\label{mainpf_sec}

\noindent
A generic $(J,\nu)$-holomorphic map contributing to the absolute 
GW-invariant~\eref{absGWdfn_e} has intersection number $A\!\cdot\!V$ with~$V$.
One would thus expect it to meet~$V$ at $A\!\cdot\!V$ distinct points.
The different orderings of these points would ideally give rise to $(A\!\cdot\!V)!$ 
distinct relative maps contributing to the relative GW-invariant~\eref{relGWdfn_e}.
However, a regular pair~$(J,\nu)$ determining the number~\eref{absGWdfn_e}
may not satisfy the conditions~\eref{NijenCond_e} and~\eref{nuVrestr_e} 
required of the pairs~$(J,\nu)$ determining the number~\eref{relGWdfn_e},
while a generic pair satisfying~\eref{NijenCond_e} and~\eref{nuVrestr_e}
may not be regular for the purposes of determining the number~\eref{absGWdfn_e}. 
Thus, there is no \`a priori reason for the identity~\eref{AReq_e} to hold in general.
Below we give two versions of nearly the same proof of Theorem~\ref{main_thm}:
first by a direct comparison and then by formally applying the symplectic sum formula.

\subsection{By direct comparison}
\label{DirComp_subs}

\noindent
The restriction~\eref{NijenCond_e} on~$J$ (or even the stronger one in~\cite{LR})
is not material, as we can simply fix one admissible~$J$ and then choose a suitable~$\nu$
to compute the GW-invariants~\eref{absGWdfn_e} and~\eref{relGWdfn_e}.
We~start by choosing a generic $\nu|_V\!\in\!\Ga_{g',k'}(V,J)$ 
and then extend it to~$X$ so that it satisfies the second condition in~\eref{nuVrestr_e}.
A generic such extension~$\nu$ determines the {\it relative} GW-invariant~\eref{relGWdfn_e}.
It counts the $(J,\nu)$-maps that pass through generic representatives of the Poincare
duals of~$\ka$ and~$\al_i$ have images in~$X$ with no components mapped into~$V$.
Dropping the contact marked points, we obtain a regular element of $\ov\fM_{g,k}(X,A;J,\nu)$
which contributes to the {\it absolute} GW-invariant~\eref{absGWdfn_e}.
However, because $\nu$ may not be generic as far as the absolute invariants are concerned,
$\ov\fM_{g,k}(X,A;J,\nu)$ may contain other elements~$u$ which meet generic representatives
of the Poincare duals of~$\ka$ and~$\al_i$.
Any such~$u$ must have at least some components mapped into~$V$,
as all other components can be regularized with~$\nu$ subject to 
the condition~\eref{nuVrestr_e}.\\

\noindent
Spaces~$\fM_{\Ga}$ of maps as at the end of the previous paragraph
can be represented by decorated connected bipartite graphs~$\Ga$ with vertices~$v$ 
\begin{enumerate}[label=$\bullet$,leftmargin=*]
\item alternating between those representing the topological components~$\Si_v$
of the domain of the maps into~$V$ and into~$X$ (without being contained in~$V$),
\item labeled by pairs indicating the genus~$g_v$ of and the degree~$A_v$ of the map on~$\Si_v$,
 and
\item decorated by disjoint subsets of $\{1,\ldots,k\}$, indicating the marked points
carried by~$\Si_v$;
\end{enumerate}
see Figure~\ref{bigraph_fig}.
Since $\fM_{\Ga}$ is contained in $\ov\fM_{g,k}(X,A;J,\nu)$,
$$g_{\Ga}+\sum_{v\in\Ga}g_v=g, \qquad \sum_{v\in\Ga}A_v=A\in H_2(X;\Z)\,,
\quad\hbox{and}\quad \sum_{v\in\Ga}k_v=k,$$
where $v\!\in\!\Ga$ means that $v$ is a vertex in~$\Ga$, $g_{\Ga}$ is the genus 
of the graph~$\Ga$  (number of edges minus the number vertices plus~1),
and $k_v$ is the number of original marked points attached to a vertex $v\!\in\!\Ga$
(the number of the original marked points carried by the topological component~$\Si_v$
of~$\Si$).
We~denote by~$\Ga_V$ the set of vertices of~$\Ga$ corresponding to the components mapped into~$V$
and by $\Ga_X$ the set of remaining vertices.
For each $v\!\in\!\Ga$, let $\ell_v\!\in\!\Z^{\ge0}$ denote the number of edges
leaving~$v$ (the number of nodes joining~$\Si_v$ to other topological components of~$\Si$).
The stability condition on the elements of  $\ov\fM_{g,k}(X,A)$ implies that
$k_v\!+\!\ell_v\!\ge\!3$   for each $v\!\in\!\Ga$ with $(g_v,A_v)\!=\!(0,0)$.\\

\noindent 
If the domains of all relevant elements of $\ov\fM_{g,k}(X,A)$ are stable, 
as is the case in Section~\ref{CM_sec}, 
the above perturbations~$\nu$ can be chosen globally as elements of $\Ga_{g,k}^V(X,J)$.
Otherwise, the same general principle applies by using compatible Kuranishi structures
for maps to~$X$ and to~$V$.
Theorem~\ref{main_thm} is established by showing that  the subspace 
$$\fM_{\Ga}(\ka;\al)\subset \fM_{\Ga}\subset \ov\fM_{g,k}(X,A;J,\nu)$$
of the elements that are of type~$\Ga$ and meet generic representatives
of the Poincare duals of~$\ka$ and~$\al$ is empty for a generic~$\nu$ satisfying
\eref{nuVrestr_e} unless~$\Ga$ is 
the one-vertex graph of maps to~$X$, as in the first diagram in Figure~\ref{bigraph_fig}.
We can assume that~$\ka$ and~$\al$ satisfy~\eref{AbsCond_e}.\\

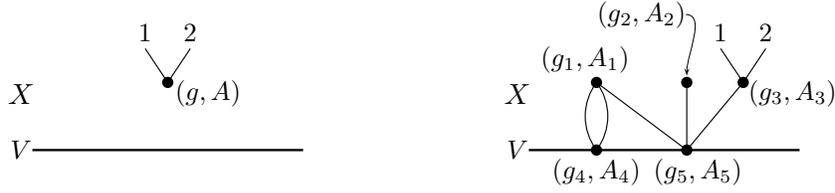
\begin{figure}
\begin{pspicture}(-.2,1.3)(11,4.2)
\psset{unit=.3cm}
\psline[linewidth=.1](10,6)(22,6)\rput(9.5,8.5){$X$}\rput(9.5,6){\sm{$V$}}
\pscircle*(16,9){.25}
\psline[linewidth=.05](16,9)(15,10.5)\psline[linewidth=.05](16,9)(17,10.5)
\rput(15,11.2){\sm{1}}\rput(17,11.2){\sm{2}}
\rput(17.8,8.5){\sm{$(g,A)$}}
\psline[linewidth=.1](32,6)(44,6)\rput(31.5,8.5){$X$}\rput(31.5,6){\sm{$V$}}
\pscircle*(35,6){.25}\pscircle*(39,6){.25}
\pscircle*(35,9){.25}\pscircle*(41.5,9){.25}\pscircle*(39,9){.25}
\psline[linewidth=.05](39,6)(35,9)
\psline[linewidth=.05](39,6)(39,9)\psline[linewidth=.05](39,6)(41.5,9)
\psarc[linewidth=.05](33,7.5){2.5}{-36.9}{36.9}\psarc[linewidth=.05](37,7.5){2.5}{143.1}{216.0}
\psline[linewidth=.05](41.5,9)(40.5,10.5)\psline[linewidth=.05](41.5,9)(42.5,10.5)
\rput(40.5,11.2){\sm{1}}\rput(42.5,11.2){\sm{2}}
\rput(35,5){\sm{$(g_4,A_4)$}}\rput(39.5,5){\sm{$(g_5,A_5)$}}
\rput(34.5,10){\sm{$(g_1,A_1)$}}\rput(43.7,8.5){\sm{$(g_3,A_3)$}}
\rput(37,12){\rnode{A}{\sm{$(g_2,A_2)$}}}\pnode(39,9.3){B}
\nccurve[angleA=0,angleB=90,linewidth=.03]{->}{A}{B}
\end{pspicture}
\caption{Bipartite graphs~$\Ga$ representing elements of $\ov\fM_{g,2}(X,A;J,\nu)$.}
\label{bigraph_fig}
\end{figure}

\noindent
Since $V$ is assumed to be $(g,A)$-hollow in Theorem~\ref{main_thm},
we can use the Symplectic Neighborhood Theorem \cite[Theorem~3.30]{MS1}
to choose an $\om$-tame almost complex structure~$J$ on~$X$
so that $J(TV)\!\subset\!TV$,  $J_V\!\equiv\!J|_V$ 
satisfies the conditions of Definition~\ref{hollow_dfn},
and $J$ satisfies~\eref{NijenCond_e}.
Thus, the degree~$A_v$ of the restriction of any element of~$\fM_{\Ga}$
to a topological component~$\Si_v$ of the domain mapped into~$V$ is zero.
If the genus~$g_v$ of such~$\Si_v$ is zero, the restriction of any element~$u$ 
of~$\fM_{\Ga}$ to~$\Si_v$ is regular as a map into~$X$ and stays so after
a small generic deformation~$\nu$ as in the previous paragraph.
If $g_v\!=\!0$ for all $v\!\in\!\Ga_V$, $\fM_{\Ga}$ consists of regular maps into~$X$
for a generic~$\nu$ satisfying~\eref{nuVrestr_e} and thus has the expected dimension.
Since this dimension is smaller than the virtual dimension of $\ov\fM_{g,k}(X,A)$,
unless $\Ga_V\!=\!\eset$, $\fM_{\Ga}(\ka;\al)\!=\!\eset$.
In particular, if $g\!=\!0$, all $(J,\nu)$-maps for a generic~$\nu$ satisfying~\eref{nuVrestr_e}
are regular as maps to~$X$ and transverse to~$V$.
Thus, the sets of stable maps contributing to the numbers on the two sides
of~\eref{AReq_e} are the same in this case,
up to the orderings of the $A\!\cdot\!V$ intersection points with~$V$.
This establishes the $g\!=\!0$ case of~\eref{AReq_e}.\\

\noindent
If $n\!\ge\!5$, 
$$\dim^{\vir}\ov\fM_{g',0}(V,0)=2(n\!-\!4)(1\!-\!g')<0 \qquad\forall~g'\ge2.$$
In these cases, we can choose deformations~$\nu$ satisfying~\eref{nuVrestr_e} 
so that $\fM_{\Ga}\!=\!\eset$ if $g_v\!\ge\!2$ for any $v\!\in\!\Ga_V$.
For the purposes of establishing the $g\!\ge\!1$ cases of~\eref{AReq_e},
it thus remains to consider the spaces~$\fM_{\Ga}$ 
so that $g_v\!\in\!\{0,1\}$ for all $v\!\in\!\Ga_V$.
Denote by $\Ga_{v;1}\!\subset\!\Ga_V$ the subset of vertices so that $g_v\!=\!1$.
In the next paragraph, we show that 
\BE{g1loss_e}\dim\fM_{\Ga}\le \dim^{\vir}\ov\fM_{g,k}(X,A)-
2\!\!\!\sum_{v\in\Ga_{v;1}}\!\!\!\ell_v\,\EE
for a generic~$\nu$ satisfying~\eref{nuVrestr_e},  
if either $n\!\ge\!5$ or $g_v\!\le\!1$ for all $v\!\in\!\Ga_V$
(in particular, if $g\!=\!1$).
Thus, $\fM_{\Ga}(\ka;\al)\!=\!\eset$ in these cases if $\Ga$ is not the basic one-vertex
graph as in the first diagram in Figure~\ref{bigraph_fig},
and so~\eref{AReq_e} again holds.\\

\noindent
Removing the vertices of~$\Ga_{v;1}$ from~$\Ga$ and replacing the edges leading to them
by the marked points on the remaining vertices,
we obtain graphs $\Ga_i$, with $i\!=\!1,\ldots,N$ for some $N\!\in\!\Z^+$,
representing subspaces~$\fM_{\Ga_i}$ of the moduli spaces $\ov\fM_{g_i,k_i+\ell_i}(X,A_i)$
with
$$\sum_{i=1}^N(g_i\!-\!1)+\sum_{v\in\Ga_{v;1}}\!\!\!\ell_v=g\!-\!1, \quad
\sum_{i=1}^NA_i=A,\quad
\sum_{i=1}^Nk_i+\sum_{v\in\Ga_{v;1}}\!\!\!k_v =k\,,\quad
\sum_{i=1}^N\ell_i=\!\sum_{v\in\Ga_{v;1}}\!\!\!\ell_v\,,$$
where $k_i\!\in\!\Z^{\ge0}$ is the number of the original marked points 
carried by the component~$\Ga_i$.
The moduli spaces $\ov\fM_{1,k_v+\ell_v}(V,0;J,\nu)$ corresponding to $v\!\in\!\Ga_{v;1}$
are of dimension $2(k_v\!+\!\ell_v)\!\in\!\Z^+$ for a generic choice of~$\nu|_V$.
Since~$\fM_{\Ga_i}$ contains no component of positive genus mapped into~$V$,
it has the expected dimension for a generic extension of~$\nu|_V$ 
satisfying~\eref{nuVrestr_e}.
Taking into account the matching conditions at the nodes joining elements of $\fM_{\Ga_i}$
to elements of $\ov\fM_{1,k_v+\ell_v}(V,0;J,\nu)$, we find~that 
\begin{equation*}\begin{split}
\dim\fM_{\Ga}&\le \sum_{i=1}^N\!\dim\fM_{\Ga_i}
+\sum_{v\in\Ga_{v;1}}\!\!\!\dim\ov\fM_{1,k_v+\ell_v}(V,0;J,\nu)
-2n\!\!\sum_{v\in\Ga_{v;1}}\!\!\!\ell_v\\
&\le  2\sum_{i=1}^N\!\big(\lr{c_1(X),A_i}\!+\!(n\!-\!3)(1\!-\!g_i)\!+\!k_i\!+\!\ell_i\big)
+2\!\!\sum_{v\in\Ga_{v;1}}\!\!\!(k_v\!+\!\ell_v)
-2n\!\!\sum_{v\in\Ga_{v;1}}\!\!\!\ell_v\\
&=2\big(\lr{c_1(X),A}\!+\!(n\!-\!3)(1\!-\!g)\!+\!k\big)
+2(n\!-\!3\!+\!1\!+\!1\!-\!n)\!\!\sum_{v\in\Ga_{v;1}}\!\!\!\ell_v\,.
\end{split}\end{equation*}
Along with the first equation in~\eref{virdim_e},
this establishes~\eref{g1loss_e} and concludes 
the proof of the first claim of Theorem~\ref{main_thm}.

\begin{rmk}\label{deg0g1_rmk}
A regular genus~1 degree~0 $(J,\nu)$-map into $V$ may not be regular as a $(J,\nu)$-map into~$X$.
However, the space of such maps has the expected dimension for the target~$X$
because  this dimension is {\it the same} as the expected dimension for the target~$V$
in the $g\!=\!1$ case.
Thus, a boundary stratum of $(J,\nu)$-maps with only $g\!=\!0,1$ components contained in~$V$
is of smaller dimension than the main stratum of maps into~$X$.
However, the space of $(J,\nu)$-maps from smooth genus~1 domains into~$V$
has the same dimension as the main stratum;
this is precisely what makes Example~\ref{deg0_eg} possible.
\end{rmk}

\noindent
Suppose next that $\ka\!=\!1$ and $g\!\ge\!2$ in~\eref{AReq_e}, i.e.~only 
the primary insertions are considered.
Given a bipartite graph~$\Ga$ describing a subspace~$\fM_{\Ga}$
of $\ov\fM_{g,k}(X,A)$ as in Figure~\ref{bigraph_fig}, 
let $\Ga_0$ be the decorated bipartite graph obtained by replacing
the genus labels  of all vertices $v\!\in\!\Ga_V$ with~0.
Thus, $\fM_{\Ga_0}$ is a subspace of $\ov\fM_{g_0,k}(X,A)$ for some $g_0\!<\!g$,
unless $g_v\!=\!0$ for all $v\!\in\!\Ga_V$ (in which case $\Ga_0\!=\!\Ga$ and thus $g_0\!=\!g$).
If $n\!=\!1,2$ and $g_0\!<\!g$,
$$\dim^{\vir}\ov\fM_{g_0,k}(X,A)<\dim^{\vir}\ov\fM_{g,k}(X,A)$$
by the first equation in~\eref{virdim_e}.
Thus, for a generic $\nu\!\in\!\Ga_{g_0,k}^V(X,J)$, 
$\fM_{\Ga_0}(1;\al)\!=\!\eset$ in this case, and so
$\nu\!\in\!\Ga_{g,k}^V(X,J)$ can be chosen so~that 
$\fM_{\Ga}(1;\al)\!=\!\eset$ whenever $g_v'\!>\!0$ for any $v\!\in\!\Ga_V$.
This establishes the $n\!=\!1,2$ cases of the last claim of Theorem~\ref{main_thm}.\\

\noindent
If $g\!\ge\!2$ in~\eref{AReq_e} and $n\!=\!3$,
\BE{n3dims_e}\dim^{\vir}\ov\fM_{g_0,k}(X,A)=\dim^{\vir}\ov\fM_{g,k}(X,A)\,.\EE
For any $v\!\in\!\Ga_V$ with $g_v\!\ge\!1$, 
$$\ov\fM_{g_v,k_v+\ell_v}(V,0)=\ov\cM_{g_v,k_v+\ell_v}\times V\,;$$
the obstruction bundle for this moduli space is 
\BE{ObsBun_e}\pi_1^*\E^*\!\otimes\!\pi_2^*TV \lra \ov\cM_{g_v,k_v+\ell_v}\!\times\!V\,,\EE
where $\E\!\lra\!\ov\cM_{g_v,k_v+\ell_v}$ is the rank~$g_v$ Hodge vector bundle
of holomorphic differentials; it has chern classes $\la_i\!\equiv\!c_i(\E)$.
For $g_v\!\ge\!2$, it is the pull-back of the Hodge vector bundle over~$\ov\cM_g$
by the forgetful morphism;
if $g_v\!=\!1$, it is the pull-back of the Hodge line bundle over~$\ov\cM_{1,1}$.
By \cite[(5.3)]{Mumford} in the first case and for dimensional reasons in the second case,
\BE{lag2_e} \la_{g_v}^2=0\in H^{4g_v}\big(\ov\cM_{g_v,k_v+\ell_v}\big).\EE
Since the obstruction bundle is given by~\eref{ObsBun_e}, 
\BE{VirClass_e} \big[\ov\fM_{g_v,k_v+\ell_v}(V,0;J,\nu)\big]=
e\big(\pi_1^*\E^*\!\otimes\!\pi_2^*TV\big)\cap
\big[\ov\cM_{g_v,k_v+\ell_v}\!\times\!V\big]\EE
for a generic $\nu\!\in\!\Ga_{g_v,k_v}^V\big(X,J)$.
By~\eref{n3dims_e}, $\fM_{\Ga_0}(1;\al)$ consists of isolated maps meeting~$V$
transversality at finitely many points~$p_j$
for such a choice of~$\nu$ (if $\fM_{\Ga_0}(1;\al)$ is not empty).
These points include the nodes where irreducible components of elements of~$\fM_{\Ga_0}(1;\al)$
meet the elements of $\ov\fM_{g_v,k_v+\ell_v}(V,0;J,\nu)$ with $v\!\in\!\Ga_V$.
By~\eref{VirClass_e} and~\eref{lag2_e}, the homology class represented by the subspace of 
the latter passing through~$p_j$ is 
$$e\big(\E^*\!\otimes\!T_{p_j}V\big)\cap \big[\ov\cM_{g_v,k_v+\ell_v}\big]
= \la_{g_v}^2\cap \big[\ov\cM_{g_v,k_v+\ell_v}\big]=0.$$
Thus, the contribution of $\fM_{\Ga}(1;\al)$ to the left-hand side of~\eref{AReq_e}
is the degree of a zero-cycle which vanishes
in the homology  and thus is~0, if $g_v\!\ge\!1$ for any $v\!\in\!\Ga_V$.
This establishes the $\ka\!=\!1$, $n\!=\!3$, and $g\!\ge\!2$ case of~\eref{AReq_e}.\\

\noindent
The remaining case of Theorem~\ref{main_thm} is $\ka\!=\!1$, $n\!=\!4$, $g\!=\!2$,
and $A\!\neq\!0$ (otherwise both sides of~\eref{AReq_e} vanish for dimensional reasons). 
By the previous discussion, it is sufficient to show that $\fM_{\Ga}(1;\al)\!=\!\eset$
for a generic~$\nu$ satisfying~\eref{nuVrestr_e} if $g_v\!=\!2$ for some $v\!\in\!\Ga_V$.
This assumption implies that $g_{v'}\!=\!0$ for all $v'\!\in\!\Ga_V\!-\!v$ and 
\BE{n4dims_e}\dim^{\vir}\ov\fM_{g_0,k}(X,A)=\dim^{\vir}\ov\fM_{g,k}(X,A)+4\,.\EE
By the first equation in~\eref{virdim_e}, the virtual dimension 
of $\ov\fM_{2,0}(V,0)$ is~0.
Thus, we can choose a deformation~$\nu$ satisfying~\eref{nuVrestr_e} so that the image
of all elements of $\ov\fM_{g_0,k_v+\ell_v}(X,A;J,\nu)$ is contained in arbitrary small
neighborhoods of finitely many points of~$V$.
By~\eref{n4dims_e}, for a generic such~$\nu$ there are no elements of 
$\fM_{\Ga_0}(1;\al)$ that pass through these images,
since each point in $V\!\subset\!X$ imposes a condition of real codimension~6 on maps to~$X$.
Thus, $\fM_{\Ga}(1;\al)\!=\!\eset$ for a generic~$\nu$ satisfying~\eref{nuVrestr_e} in this case
as well.

\subsection{Via the symplectic sum formula}
\label{SympSum_subs}

\noindent
We next give a proof of Theorem~\ref{main_thm} by applying the symplectic sum 
formula to the symplectic decomposition
\BE{Xdecomp_e}X=X\underset{V=\P_{X,\i}V}{\#}\P_XV\,,\EE
with $\P_{X,\i}V\!\subset\!\P_XV$ as in \eref{PXVdfn_e} and~\eref{PXVsecdfn_e}.
The $\P^1$-bundle $\P_XV\!\lra\!V$ carries a symplectic form induced from~$\om|_V$
in a way well-defined up to symplectic deformation equivalence;
see the beginning of Section~\ref{VMext_subs}.\\

\noindent
According to the symplectic sum formula, the left-hand side of~\eref{AReq_e}
is a weighted count of $k$-marked genus~$g$ degree~$A$ $(J,\nu)$-maps~$u$
into
\BE{X1Vdfn_e} X_1^V\equiv X\underset{V=\P_{X,\i}V}{\cup}\P_XV\EE
that have the same contact order with the common hypersurface~$V$ at the two branches
of each node, take no smooth point of the domain to~$V$, and meet 
generic representatives of the Poincare duals of~$\ka$ and~$\al_i$.
The \sf{degree} of such~$u$ is the class in~$X$ represented by the composition of~$u$
with the natural projection
\BE{qproj_e} q\!:  X\underset{V=\P_{X,\i}V}{\cup}\P_XV\lra X\,;\EE
its \sf{weight} is the product of the contacts with the common hypersurface
(counted once for each pair of contacts from the two sides).\\

\noindent
Spaces $\fM_{\Ga}(\ka;\al)$ of such maps to~$X_1^V$ can be represented by the same 
kind of connected bipartite graphs~$\Ga$ as in Section~\ref{DirComp_subs}
with an additional decoration $d_e\!\in\!\Z^+$ for 
each edge~$e$; see Figure~\ref{bigraph_fig2}, where edge labels~1 are not
explicitly indicated.
The subset~$\Ga_V$ of vertices now describes the topological components~$\Si_v$
of the domain~$\Si$ that are mapped to~$\P_XV$;
the additional decorations~$d_e$ specify the orders of contacts with~$V$
of the branches of the nodes associated with the edges.
The stability condition on~$\Ga$ described before now applies only to the vertices 
$v\!\in\!\Ga_X$.
The composition of an element~$u$ in such a space $\fM_{\Ga}(\ka;\al)$ with~$q$
produces an element of the space $\fM_{\bar\Ga}(\ka;\al)$ considered above
with $\bar\Ga$ obtained from~$\Ga$ by dropping the edge labels and contracting
off the unstable vertices $v\!\in\!\Ga_V$ and the edges leaving from them.\\

\noindent
Breaking a graph~$\Ga$  as in the previous paragraph at the mid-point of each edge, 
we obtain 
the relative moduli spaces
$$\ov\fM_{g_v,k_v;\s_v}^V(X,A_v) \qquad\hbox{and}\qquad 
\ov\fM_{g_v,k_v;\s_v}^{\P_{X,\i}V}(\P_XV,A_v(\s_v))$$
with $v\!\in\!\Ga_X$ and $v\!\in\!\Ga_V$, respectively, where 
$\s_v$ is the tuple given by the labels on the edges and
$A_v(\s_v)$ is the sum of the push-forward of~$A_v$
under the inclusion $\P_{X,0}V\!\lra\!\P_XV$ and $|\s_v|$ fiber classes.
The left-hand side of~(1.9) is the sum over all admissible graphs~$\Ga$
of the weighted products of the corresponding relative invariants with 
the relative primary insertions given by the usual Kunneth decomposition 
of the diagonal in~$V^2$ at each node; 
see the second-to-last equation on page~201 in~\cite{Jun2}
and equations~(5.4), (5.7), and~(5.8) in~\cite{LR}.
Since the intersection points of elements of $\fM_{\Ga}(\ka;\al)$ are unordered,
while the contact points of the corresponding relative invariants are ordered,
the contribution from each graph~$\Ga$ should be divided by the number of
orderings of the intersection points.\\

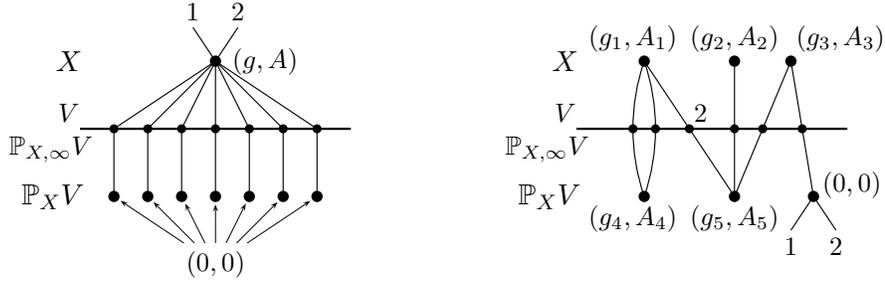
\begin{figure}
\begin{pspicture}(0,0)(11,4)
\psset{unit=.3cm}
\psline[linewidth=.1](10,6)(22,6)\rput(9.5,9){$X$}\rput(8.8,3){$\P_XV$}
\rput(9.5,6.7){\sm{$V$}}\rput(8.7,5.1){\sm{$\P_{X,\i}V$}}
\pscircle*(16,9){.25}
\psline[linewidth=.05](16,9)(15,10.5)\psline[linewidth=.05](16,9)(17,10.5)
\rput(15,11.2){\sm{1}}\rput(17,11.2){\sm{2}}\rput(18.2,9){\sm{$(g,A)$}}
\pscircle*(11.5,6){.2}\pscircle*(13,6){.2}\pscircle*(14.5,6){.2}\pscircle*(16,6){.2}
\pscircle*(17.5,6){.2}\pscircle*(19,6){.2}\pscircle*(20.5,6){.2}
\psline[linewidth=.05](16,9)(11.5,6)\psline[linewidth=.05](16,9)(13,6)
\psline[linewidth=.05](16,9)(14.5,6)\psline[linewidth=.05](16,9)(16,6)
\psline[linewidth=.05](16,9)(17.5,6)\psline[linewidth=.05](16,9)(19,6)
\psline[linewidth=.05](16,9)(20.5,6)
\psline[linewidth=.05](11.5,6)(11.5,3)\psline[linewidth=.05](13,6)(13,3)
\psline[linewidth=.05](14.5,6)(14.5,3)\psline[linewidth=.05](16,6)(16,3)
\psline[linewidth=.05](17.5,6)(17.5,3)\psline[linewidth=.05](19,6)(19,3)
\psline[linewidth=.05](20.5,6)(20.5,3)
\pscircle*(11.5,3){.25}\pscircle*(13,3){.25}\pscircle*(14.5,3){.25}\pscircle*(16,3){.25}
\pscircle*(17.5,3){.25}\pscircle*(19,3){.25}\pscircle*(20.5,3){.25}
\rput(16,0){\rnode{A}{\sm{$(0,0)$}}}
\pnode(11.5,3){B1}\pnode(13,3){B2}\pnode(14.5,3){B3}\pnode(16,3){B4}
\pnode(17.5,3){B5}\pnode(19,3){B6}\pnode(20.5,3){B7}
\ncline[linewidth=.03,nodesep=3pt]{->}{A}{B1}
\ncline[linewidth=.03,nodesep=3pt]{->}{A}{B2}
\ncline[linewidth=.03,nodesep=3pt]{->}{A}{B3}
\ncline[linewidth=.03,nodesep=3pt]{->}{A}{B4}
\ncline[linewidth=.03,nodesep=3pt]{->}{A}{B5}
\ncline[linewidth=.03,nodesep=3pt]{->}{A}{B6}
\ncline[linewidth=.03,nodesep=3pt]{->}{A}{B7}
\psline[linewidth=.1](32,6)(44,6)\rput(31.5,9){$X$}\rput(30.8,3){$\P_XV$}
\rput(31.5,6.7){\sm{$V$}}\rput(30.7,5.1){\sm{$\P_{X,\i}V$}}
\pscircle*(35,3){.25}\pscircle*(35,9){.25}
\pscircle*(39,3){.25}\pscircle*(41.5,9){.25}\pscircle*(39,9){.25}
\psline[linewidth=.05](39,3)(35,9)
\psline[linewidth=.05](39,3)(39,9)\psline[linewidth=.05](39,3)(41.5,9)
\pscircle*(37,6){.2}\pscircle*(40.25,6){.2}\pscircle*(39,6){.2}
\psarc[linewidth=.05](26.25,6){9.25}{-18.9}{18.9}\psarc[linewidth=.05](43.75,6){9.25}{161.1}{198.9}
\pscircle*(35.5,6){.2}\pscircle*(34.5,6){.2}
\psline[linewidth=.05](41.5,9)(42.5,3)\pscircle*(42.5,3){.25}\pscircle*(42,6){.2}
\psline[linewidth=.05](42.5,3)(41.5,1.5)\psline[linewidth=.05](42.5,3)(43.5,1.5)
\rput(41.5,.8){\sm{1}}\rput(43.5,.8){\sm{2}}
\rput(34.5,2){\sm{$(g_4,A_4)$}}\rput(39,2){\sm{$(g_5,A_5)$}}
\rput(34.5,10){\sm{$(g_1,A_1)$}}\rput(43.7,10){\sm{$(g_3,A_3)$}}
\rput(44.2,3.5){\sm{$(0,0)$}}
\rput(39,10){\sm{$(g_2,A_2)$}}\rput(37.5,6.7){\sm{$2$}}
\end{pspicture}
\caption{Bipartite graphs~$\Ga$ representing elements of $\ov\fM_{g,2}(X_1^V,A;J,\nu)$ 
with $A\!\cdot\!V\!=\!7$.}
\label{bigraph_fig2}
\end{figure}

\noindent
Some care is needed in translating the constraints~$\ka$ and~$\al_i$ in~\eref{absGWdfn_e}
into constraints for the relative invariants of~$(X,V)$ and~$(\P_XV,\P_{X,\i}V)$.
If $v\!\in\!\Ga_X$, the corresponding relative invariant of~$(X,V)$ keeps 
the insertion~$\al_i$ at the absolute marked point corresponding to~$i$,
if it is carried by~$\Si_v$.
If $v\!\in\!\Ga_V$, the corresponding relative invariant of $(\P_XV,\P_{X,\i}V)$ 
gets the insertion $\pi_{X,V}^*(\al_i|_V)$ at the absolute marked point corresponding to~$i$,
where $\pi_{X,V}\!:\P_XV\!\lra\!V$ is the projection map.
Denote~by
$$\st_v\!: \ov\fM_{g_v,k_v;\s_v}^V(X,A_v)\lra \ov\cM_{g_v,k_v+\ell(\s_v)}
\quad\hbox{or}\quad
\st_v\!: \ov\fM_{g_v,k_v;\s_v}^{\P_{X,\i}V}(\P_XV,A_v(\s_v))\lra  \ov\cM_{g_v,k_v+\ell(\s_v)}$$
the stabilization map, depending on whether $v\!\in\!\Ga_X$ or $v\!\in\!\Ga_V$,
respectively; in the unstable range, the target of this map is one point.
Let
$$\gl_{\Ga}\!: \prod_{v\in\Ga}\!\ov\cM_{g_v,k_v+\ell(\s_v)}\lra \ov\cM_{g,k}$$
be the morphism given by identifying pairs of points corresponding to 
the same edge in~$\Ga$.
In particular,
$$\gl_{\Ga}\circ \prod_{v\in\Ga}\!\st_v
=\st\!\circ\!\io_{\Ga}: \ov\fM_{\Ga}\lra \ov\cM_{g,k}\,,$$
where $\io_{\Ga}\!:\ov\fM_{\Ga}\!\lra\!\ov\fM_{g,k}(X_1^V,A)$ is the inclusion map.
By the Kunneth formula,
$$\gl_{\Ga}^*\ka=\sum_j\bigotimes_{v\in\Ga}\!\ka_{j;v}
\in \bigotimes_{v\in\Ga}\!H^*\big(\ov\cM_{g_v,k_v+\ell(\s_v)}\big)
=H^*\bigg( \prod_{v\in\Ga}\!\ov\cM_{g_v,k_v+\ell(\s_v)}\bigg)$$
for some $\ka_{j;v}\!\in\!H^*(\ov\cM_{g_v,k_v+\ell(\s_v)})$.
In the $\Ga$-summand in the symplectic sum decomposition for 
the absolute GW-invariant~\eref{absGWdfn_e},
the insertion~$\ka$ is replaced by the insertion~$\ka_{v;j}$ 
in the relative invariant corresponding to the vertex~$v$
and the resulting products are summed over all~$j$.
This is carried out in a specific case in Section~\ref{P1eg_subs}.\\

\noindent
Since $V$ is assumed to be $(g,A)$-hollow in Theorem~\ref{main_thm},
we can choose an almost complex structure~$J_V$ on~$V$ so that it satisfies 
the conditions of Definition~\ref{hollow_dfn}. 
Using a connection in~$\cN_XV$ as in Section~\ref{RelDfn_sec},
we can extend~$J_V$ to an almost complex structure~$J$ on~$\P_XV$ so
that the condition \eref{NijenCond_e} is satisfied and
the projection $\pi_{X,V}\!:\P_XV\!\lra\!V$ is $(J_V,J)$-holomorphic.
Using the same connection, we can extend any $\nu\!\in\!\Ga_{g,k}(V,J_V)$
to  
$$\pi_{X,V}^*\nu\in\Ga_{g,k}^{\P_{X,\i}V}(\P_XV,J)$$
so that $\pi_{X,V}\!\circ\!u\!:\Si\!\lra\!V$ is $(J_V,\nu)$-holomorphic 
whenever $u\!:\Si\!\lra\!\P_XV$ is $(J,\pi_{X,V}^*\nu)$-holomorphic.\\

\noindent
By the previous paragraph, we can assume that the degree~$A_v$ of 
the composition of the restriction of any element of~$\fM_{\Ga}$
to a topological component~$\Si_v$ of the domain mapped into $\P_XV$ with~$\pi_{X,V}$ is zero,
i.e.~all relevant relative invariants of $(\P_XV,\P_{X,\i}V)$ lie in
the fiber classes~$d_vF$ with $d_v\!\in\!\Z^{\ge0}$.
A~key point of the paragraph above the previous one is that 
the class integrated over the relative moduli space $\ov\fM_{g_v,k_v;\s_v}^{\P_{X,\i}V}(\P_XV,dF)$
corresponding to the vertex~$v$ is pulled back by the projection map
\BE{vphdfn_e}\vph\equiv\st\!\times\!\pi_{X,V}\!: 
\ov\fM_{g_v,k_v;\s_v}^{\P_{X,\i}V}(\P_XV,d_vF)\lra\ov\cM_{g_v,k_v+\ell(\s_v)}\!\times\!V\,.\EE
In particular, if 
\BE{dimhigh_e}\dim^{\vir}\ov\fM_{g_v,k_v;\s_v}^{\P_{X,\i}V}(\P_XV,d_vF)
>\dim\big(\ov\cM_{g_v,k_v+\ell(\s_v)}\!\times\!V\big),\EE
then the relative invariant corresponding to the vertex $v\!\in\!\Ga_V$ vanishes
and such bipartite graph~$\Ga$ does not contribute to the left-hand side of~\eref{AReq_e}.\\

\noindent
By the second equation in~\eref{virdim_e} and the condition $|\s_v|\!=\!d_v$, 
\eref{dimhigh_e} is equivalent~to
$$d_v+(n\!-\!3)(1\!-\!g_v)+k_v+\ell(\s_v)
>n\!-\!1+\begin{cases}
0,&\hbox{if}~g_v\!=\!0,~\!k_v\!+\!\ell(\s_v)\!\le\!2;\\
3g_v\!-\!3\!+\!k_v\!+\!\ell(\s_v),&\hbox{otherwise}.
\end{cases}$$
If $g_v\!=\!0$, either $d_v\!\in\!\Z^+$ (and thus $\ell(\s_v)\!\in\!\Z^+$)
or $k_v\!\ge\!3$ for stability reason.
Thus, the relative invariant corresponding to a vertex $v\!\in\!\Ga_V$ with $g_v\!=\!0$
is zero unless $d_v\!=\!1$, $k_v\!=\!0$, and $\s_v\!=\!(1)$. 
In this remaining case, the only nonzero relative invariant~is
$$\GW_{0,F;(1)}^{\P_XV,\P_{X,\i}V}\big(1,1;\PD_V([\pt])\big)=1.$$
In particular, the contribution to the left-hand side of~\eref{AReq_e} from the simplest graph,
i.e.~as in the first diagram in Figure~\ref{bigraph_fig2},~is
\BE{RelContr_e}\frac{1}{(A\cdot V)!}\GW_{g,A;\1_{A\cdot V}}^{X,V}\big(\ka;\al;1^{A\cdot V}\big)
\equiv  \frac{1}{(A\cdot V)!}\GW_{g,A;\1_{A\cdot V}}^{X,V}\big(\ka;\al\big)\,.\EE
All other nonzero contributions to the left-hand side of~\eref{AReq_e} can come
only from graphs~$\Ga$ such that $(d_v,k_v,\s_v)\!=\!(1,0,(1))$ for all $v\!\in\!\Ga_V$
with $g_v\!=\!0$ and $g_v\!\in\!\Z^+$ for some $v\!\in\!\Ga_V$.
Since there are no such graphs if $g\!=\!0$, this concludes the proof
of the $g\!=\!0$ case of~\eref{AReq_e}.\\

\noindent
We next show that the relative invariants corresponding to $v\!\in\!\Ga_V$
with $g_v\!\in\!\Z^+$ also vanish under the assumptions of~\eref{AReqcond_e}.
If $g_v\!\ge\!2$ and $\nu\!\in\!\Ga_{g_v,0}(V,J_V)$, 
the composition with the projection~$\pi_{X,V}$ induces a continuous~map 
\BE{fMfibr_e}\pi_{X,V}\!:\ov\fM_{g_v,k_v;\s_v}^{\P_{X,\i}V}\big(\P_XV,d_vF;J,\pi_{X,V}^*\nu\big)
\lra \ov\fM_{g_v,0}\big(V,0;J_V,\nu\big).\EE
Since 
$$\dim\ov\fM_{g_v,0}\big(V,0;J_V,\nu\big)
=\dim^{\vir}\ov\fM_{g_v,0}\big(V,0\big)
=(n\!-\!4)(1\!-\!g_v) \qquad\forall~g_v\!\ge\!2$$
for a generic $\nu\!\in\!\Ga_{g_v,0}(V,J_V)$,
the moduli spaces in~\eref{fMfibr_e} are empty if $g_v\!\ge\!2$ and $n\!\ge\!5$.
In particular, the relative invariants vanish in these cases.\\

\noindent
If $g_v\!=\!1$, then $d_v,\ell(\s_v)\!\in\!\Z^+$ by the first assumption in~\eref{AReqcond_e}. 
For a generic $\nu\!\in\!\Ga_{1,1}(V,J_V)$,
$$\pi_{X,V}\!:\ov\fM_{1,k_v;\s_v}^{\P_{X,\i}V}\big(\P_XV,d_vF;J,\pi_{X,V}^*\nu\big)
\lra \ov\fM_{1,1}\big(V,0;J_V,\nu\big)$$
is then a fibration with typical fiber $\ov\fM_{1,k_v;\s_v}^{\pt}(\P^1,d)_{\fj}$,
where the subscript $\fj$ denotes the moduli space with $\fj$ fixed on $\ov\cM_{1,1}$.
Since the obstruction bundle for $\ov\fM_{1,1}(V,0)$ is given by~\eref{ObsBun_e},
\BE{g1_e}
\big[\ov\fM_{1,1}\big(V,0;J_V,\nu)\big]=
e\big(\pi_1^*\E^*\!\otimes\!\pi_2^*TV\big)\cap   \big[\ov\cM_{1,1}\!\times\!V\big]
=\{\fj\}\!\times\!V_1+\ov\cM_{1,1}\!\times\!V_0,\EE
where $V_0,V_1\!\subset\!V$ are some cycles  of $\R$-dimensions~0 and~2,
respectively, and $\fj$ is a fixed element of~$\cM_{1,1}$.
Since
$$\dim^{\vir}\ov\fM_{1,k_v;\s_v}^{\pt}(\P^1,d_v) = d_v+k_v\!+\!\ell(\s_v)
>k_v+\ell(\s_v)=\dim\ov\cM_{1,k_v+\ell(\s_v)},$$
the integral of the pull-back of any class by~\eref{vphdfn_e} 
vanishes on the last term in~\eref{g1_e}.
Since
$$\dim^{\vir}\ov\fM_{1,k_v;\s_v}^{\pt}(\P^1,d_v)_{\fj} = d_v\!-\!1+k_v\!+\!\ell(\s_v)
>k_v\!+\!\ell(\s_v)-1=\dim\ov\cM_{1,k_v+\ell(\s_v);\fj}\,,$$
the integral of the pull-back of any class by~\eref{vphdfn_e} 
vanishes on the first term on the RHS of~\eref{g1_e} as well.\\

\noindent
In summary, the only graph~$\Ga$ that contributes to the left-hand side of~\eref{AReq_e}
via the symplectic sum formula applied to the decomposition~\eref{Xdecomp_e}
under the assumptions~\eref{AReqcond_e} is the graph with 
$$|\Ga_X|=1 \qquad\hbox{and}\qquad 
(g_v,d_v,k_v,\s_v)=(0,1,0,(1))\quad\forall~v\!\in\!\Ga_V;$$
see the first diagram in Figure~\ref{bigraph_fig2}.
Since its contribution is given by~\eref{RelContr_e}, we have established the first claim
of Theorem~\ref{main_thm}.\\ 

\noindent
Suppose next that $\ka\!=\!1$ and $g\!\ge\!2$ in~\eref{AReq_e}, i.e.~only 
the primary insertions are considered.
The relative invariants of $(X,V)$ that enter into the symplectic sum formula
then count curves that meet generic Poincare duals of all the constraints~$\al_i$.
If $n\!=\!1,2$ and $g_0\!<\!g$, 
\BE{n3dims_e2}\dim^{\vir}\ov\fM_{g_0,k;\s_0}^V(X,A)<\dim^{\vir}\ov\fM_{g,k}(X,A)\EE
by~\eref{virdim_e}.
Thus, these relative invariants vanish if $n\!=\!1,2$ and
the total genus of the vertices in~$\Ga_X$ is less than~$g$.
This happens in particular if $g_v\!>\!0$ for any $v\!\in\!\Ga_V$.
Along with the paragraph containing~\eref{RelContr_e}, 
this establishes the $n\!=\!1,2$ cases of the last claim of Theorem~\ref{main_thm}.\\

\noindent
Suppose $g_v\!\ge\!2$ for some $v\!\in\!\Ga_V$ and $n\!=\!3$.
The dimensions of the two moduli spaces in~\eref{n3dims_e2} are then the same.
The relative invariants of $(X,V)$ that enter into the symplectic sum formula
thus count curves that meet~$V$ at finitely many distinct points~$\{p_j\}$.
Since the obstruction bundle for $\ov\fM_{g_v,0}(V,0)$ is given by~\eref{ObsBun_e},
the homology class of the subspace of elements of $\ov\fM_{g_v,0}(V,0;J_V,\nu)$ that pass through~$p_j$
$$\big[\ov\fM_{g_v,0}(V,0;J_V,\nu)|_{p_j}\big]=
e\big(\E^*\!\otimes\!T_{p_j}V\big)\cap
\big[\ov\cM_{g_v,0}\!\times\!\{p_j\}\big]
=\la_{g_v}^2\cap \big[\ov\cM_{g_v,0}\big]=0;$$
see~\eref{lag2_e}.
Thus, by~\eref{fMfibr_e}, the genus~$g_v$ relative invariants of $(\P_XV,\P_{X,\i}V)$
with a relative point insertion vanish in this case as well.\\

\noindent
The remaining case of Theorem~\ref{main_thm} is $\ka\!=\!1$, $n\!=\!4$, and $g\!=\!2$.
Since $A\!\neq\!0$ in this case, $d_v,\ell(\s_v)\!\in\!\Z^+$.
For a generic $\nu\!\in\!\Ga_{2,0}(V,J_V)$, the target in~\eref{fMfibr_e}
is a finite set of points, while the dimension of the fiber~is
$$d_v+1-g_v+k_v+\ell(\s_v)\ge 1-1+0+1=1.$$
Thus, the genus~2 relative invariants of $(\P_XV,\P_{X,\i}V)$
with only primary insertions from~$V$ vanish.
This concludes the proof of the last claim of Theorem~\ref{main_thm}.

\subsection{Extension to virtual cycles}
\label{VMext_subs}
  
\noindent
In the process of establishing the first claim of Theorem~\ref{main_thm} above,
we showed that the relative invariants in the fiber classes of $\P^1$-bundles
often vanish.
This, more technical, conclusion is summarized, in Lemma~\ref{relGW_lmm} below.
It leads to a version of Theorem~\ref{main_thm} for virtual moduli cycles;
see Corollary~\ref{VC_crl}.\\

\noindent
Let $(V,\om)$ be a compact symplectic manifold, $\pi_L\!:L\!\lra\!V$ be a complex line bundle,
and 
$$\pi_{L,V}\!: \P_L\equiv \P(L\!\oplus\!V\!\times\!\C)\lra V$$
be the bundle projection map.
Given a Hermitian metric~$\rho$ (square of the norm) and 
a $\rho$-compatible connection~$\na$ in~$L$,
let $\al$ denote the connection 1-form on the $\rho$-circle bundle in~$L$ and
its extension to $L\!-\!V$ via the retraction given by $v\!\lra\!v/|v|$.
The closed 2-form
$$\wt\om\equiv \pi_{X,V}^*\om-\ep\,\nd\bigg(\frac{\al}{1\!+\!\rho^2}\bigg)$$
on $L\!-\!V$ extends to a closed 2-form on~$\P_L$, which is symplectic 
if $\ep\!>\!0$ is sufficiently small;
we will take the symplectic deformation equivalence class of this form to 
be the default one.
Let 
$$\P_{L,\i}=\P(L\!\oplus\!0)\subset \P_L\,.$$
The projection map
$$\vph\equiv\st\!\times\!\pi_{L,V}\!: 
\ov\fM_{g,k;\s}^{\P_{L,\i}}(\P_L,dF)\lra\ov\cM_{g,k+\ell(\s)}\!\times\!V,$$
where $F\!\in\!H_2(\P_L;\Z)$ is the fiber class,
induces a push-forward on the virtual class:
$$\vph_*\big[\ov\fM_{g,k;\s}^{\P_{L,\i}}(\P_L,dF)\big]^{\vir}
\in H_*\big(\ov\cM_{g,k+\ell(\s)}\!\times\!V\big).$$
By the Poincare duality applied on $\ov\cM_{g,k+\ell(\s)}\!\times\!V$,
this push-forward is determined by the evaluation of cohomology classes pulled back
from $\ov\cM_{g,k+\ell(\s)}\!\times\!V$ by~$\vph$ on the virtual class
of $\ov\fM_{g,k;\s}^{\P_{L,\i}}(\P_L,dF)$.
Thus, Section~\ref{SympSum_subs} establishes the following statement.

\begin{lmm}\label{relGW_lmm}
Let $(V,\om)$ be a compact symplectic manifold of real dimension~$2(n\!-\!1)$ 
and $L\!\lra\!V$ be a complex line bundle.
If $g,d,k\!\in\!\Z^{\ge0}$ and $\s\!\in\!(\Z^+)^{\ell}$ are such~that 
\BE{AReqcond_e2}(g,d)\neq(1,0), 
\qquad\hbox{and}\qquad 
(n\!-\!5)g(g\!-\!1)\ge0\,,\EE
then 
$$\vph_*\big[\ov\fM_{g,k;\s}^{\P_{L,\i}}(\P_L,dF)\big]^{\vir}=
\begin{cases}[V],&\hbox{if}~(g,d,k,\s)\!=\!(0,1,0,(1));\\
0,&\hbox{otherwise}.
\end{cases}$$
\end{lmm}

\begin{crl}[D.~Maulik]\label{VC_crl}
Suppose $(X,\om)$ is a projective manifold of real dimension~$2n$, 
\hbox{$g,k\!\in\!\Z^{\ge0}$}, \hbox{$A\!\in\!H_2(X;\Z)$}, and
$V\!\subset\!X$ is a $(g,A)$-hollow projective hypersurface such that $A\!\cdot\!V\!\ge\!0$.
If 
\BE{AReqcond_e3}(g,A)\neq (1,0) \qquad\hbox{and}\qquad (n\!-\!5)g(g\!-\!1)\ge0\,,\EE
then 
\BE{AReq_e2}
\big[\ov\fM_{g,k}(X,A)\big]^{\vir}
=\frac{1}{(A\cdot V)!}\,f_*\big[\ov\fM_{g,k;\1_{A\cdot V}}(X,A)\big]^{\vir}\,,\EE
where $f$ is the morphism between the moduli spaces dropping the relative marked points. 
\end{crl}

\begin{proof}
Let $\De\!\subset\!\C$ denote a small disk around the origin, 
$\cZ$ be the blowup of $\De\!\times\!X$ along $0\!\times\!V$,
and $\pi\!:\cZ\!\lra\!\De$ be the projection map.
Thus, 
$$\cZ_{\la}=X\quad\forall~\la\in\De^*\!\equiv\!\De\!-\!0
\qquad\hbox{and}\qquad \cZ_0\equiv \pi^{-1}(0)=X_1^V\,,$$
with notation as in~\eref{X1Vdfn_e}.\\

\noindent
As summarized in \cite[Section~0]{Jun2}, there are moduli stacks 
$\ov\fM_{g,k}(X_1^V,A)$ and $\ov\fM_{g,k}(\cZ,A)$.
The former carries a virtual class so that 
\BE{faminc_e}\big[\ov\fM_{g,k}(X_1^V,A)\big]^{\vir}=\big[\ov\fM_{g,k}(\cZ_{\la},A)\big]^{\vir}
=\big[\ov\fM_{g,k}(X,A)\big]^{\vir}\EE
under the inclusion into $\ov\fM_{g,k}(\cZ,A)$.
In the case of the given family $\cZ\!\lra\!\De$, \eref{faminc_e} can be written~as
\BE{qpush_e}q_*\big[\ov\fM_{g,k}(X_1^V,A)\big]^{\vir}
=\big[\ov\fM_{g,k}(X,A)\big]^{\vir}\,,\EE
with $q$ as in~\eref{qproj_e}.
By the last formula on page~201 in~\cite{Jun2},
\BE{VCsplit_e}\big[\ov\fM_{g,k}(X_1^V,A)\big]^{\vir}
=\sum_{\Ga}\frac{\bfm(\Ga)}{\ell(\Ga)!}
\Phi_{\Ga*}\De^{!}\big( 
\big[\ov\fM(X,\Ga_X)\big]^{\vir}\!\times\!\big[\ov\fM(\P_XV,\Ga_V)^{\vir}\big]\big).\EE
This sum is taken over the same bipartite graphs~$\Ga$ as in Section~\ref{SympSum_subs}.
For such a graph~$\Ga$, $\bfm(\Ga)$ is the product of the edge labels 
(of contacts with the common divisor~$V$) and 
$\ell(\Ga)$ is the number of edges (of contacts with~$V$). 
In the notation of Section~\ref{SympSum_subs}, 
the two moduli spaces appearing on the right-hand side of~\eref{VCsplit_e} are
$$\prod_{v\in\Ga_X}\!\!\ov\fM_{g_v,k_v;\s_v}^V(X,A_v) \qquad\hbox{and}\qquad
\prod_{v\in\Ga_V}\!\!\ov\fM_{g_v,k_v;\s_v}^{\P_{X,\i}V}\big(\P_XV,A_v(\s_v)\big)\,,$$
respectively.
The symbol~$\De^!$ indicates the cap product with
the product over the edges of~$\Ga$ of the pull-back of the diagonal $\De_V\!\subset\!V^2$
by the evaluation maps at the relative marked points corresponding to the same edge,
while~$\Phi_{\Ga}$ is the morphism given by identifying these marked points.\\

\noindent
Since $V$ is $(g,A)$-hollow, the only possible nonzero summands in~\eref{VCsplit_e} 
correspond to~$\Ga$ with $A_v\!=\!0$ for all $v\!\in\!\Ga_V$.
For such~$\Ga$, the relative evaluation maps are given by the composition
with the projection map $\pi_{X,V}\!:\P_XV\!\lra\!V$,
while $q\!\circ\!\Phi_{\Ga}$ factors through $\id\!\times\!\vph$. 
Combining~\eref{qpush_e} and~\eref{VCsplit_e}, we thus obtain
\BE{VCsplit_e2}\big[\ov\fM_{g,k}(X,A)\big]^{\vir}=
\sum_{\Ga}\frac{\bfm(\Ga)}{\ell(\Ga)!}
\Phi_{\Ga*}\De^{!}\big( 
\big[\ov\fM(X,\Ga_X)\big]^{\vir}\!\times\!\vph_*\big[\ov\fM(\P_XV,\Ga_V)^{\vir}\big]\big),\EE
with the sum now taken over bipartite graphs~$\Ga$ as in Section~\ref{SympSum_subs}
with  $A_v\!=\!0$ for all $v\!\in\!\Ga_V$.
For graphs~$\Ga$, the restrictions~\eref{AReqcond_e3} imply 
the restrictions~\eref{AReqcond_e2} for all $(g,d)\!=\!(g_v,d_v)$ with $v\!\in\!\Ga_V$.
By Lemma~\ref{relGW_lmm}, the last term in~\eref{VCsplit_e2} thus vanishes except for 
the basic graph~$\Ga$ with $|\Ga_X|\!=\!1$,  $(g_v,A_v,k_v)\!=\!(0,0,0)$ for all $v\!\in\!\Ga_V$,
and all edge labels equal to~1, i.e.~as in the first diagram in Figure~\ref{bigraph_fig2}.
The summand in~\eref{VCsplit_e2} corresponding to this basic graph gives~\eref{AReq_e2}.
\end{proof}

\begin{rmk}\label{Jun_rmk}
The equality~\eref{faminc_e} is established in~\cite{Jun2} for a general flat degeneration
$\pi\!:\cZ\!\lra\!\De$, with~$\cZ_0$ consisting of two smooth varieties joined along
a smooth hypersurface, only after summing 
over all classes~$A$ 
of the same degree with respect to an ample line bundle over~$\cZ$.
However, in the given case, the relevant summands on the two sides of~\eref{faminc_e}
lie in different spaces and thus must be equal pairwise. 
\end{rmk}

\begin{rmk}\label{SympVC_rmk}
The conclusion of Corollary~\ref{VC_crl} should apply to any compact 
symplectic manifold~$(X,\om)$ and $(g,A)$-hollow symplectic hypersurface~$V$.
Unfortunately, the above proof of Corollary~\ref{VC_crl} makes use of 
the symplectic sum (degeneration) formula for virtual fundamental cycles
(not just numbers) in GW-theory,
which is not even claimed in the symplectic category in any work we are aware~of.
In particular, \cite{IPsum} is concerned only with equating GW-invariants
(viewed as numbers), contrary to the claim just above \cite[(11.4)]{IPvfc}.
\end{rmk}

\section{Details on the counter-examples}
\label{eg_sec}

\noindent
In Sections~\ref{deg0eg_subs}-\ref{P4eg_subs}, 
we establish the claims made in Examples~\ref{deg0_eg}-\ref{P4_eg},
respectively; see Section~\ref{intro_sec}.
In the case of Example~\ref{deg0_eg}, we give two computations of
the relative invariants.
In the cases of Examples~\ref{P1_eg} and~\ref{P4_eg},
we include localization computations of the $\de\!=\!0,1$ numbers
as consistency checks;
the localization computations for Example~\ref{P4_eg}
are separated off into Section~\ref{LocComp_subs}.\\

\noindent
In Sections~\ref{P1eg_subs}-\ref{LocComp_subs}, we use some degree~1 relative GW-invariants
of $(\P^1,\i)$ and rubber relative invariants of $(\P^1,\i,0)$ with respect to 
the standard $\C^*$-action.
In principle, all such invariants are computed in \cite{FaberP,OP}.
As it is not completely trivial to extract actual numbers from 
the generating series in \cite{FaberP,OP}, 
we include alternative computations for the few numbers relevant to our purposes.

\subsection{Genus 1 degree 0 invariants}
\label{deg0eg_subs}

\noindent
Let $(X,\om,V)$ be as in Example~\ref{deg0_eg}.
Fix an $\om$-tame almost complex structure~$J$ on~$X$ so that $J(TV)\!=\!TV$
and the Nijenhuis condition~\eref{NijenCond_e} holds.\\

\noindent
For a complex structure~$\fj$ on a smooth 1-marked genus~1 Riemann surface 
$(\Si,x_1)$, the space of degree~0 holomorphic maps $u\!:\Si\!\lra\!X$
consists of the constant maps and so is canonically isomorphic to~$X$.
The obstruction bundle (i.e.~the bundle of the cokernels of the linearizations~$D_{J,u}$
of the $\dbar_{J,\fj}$-operator at~$u$) is isomorphic to 
$\cH^{0,1}_{\fj}\!\otimes_{\C}\!TX$,
where $\cH^{0,1}_{\fj}$ is the complex one-dimensional space
of anti-holomorphic one-forms on~$\Si$.
Thus,
\BE{g1Obs_e}TX\approx \Obs\lra \Hol_{\fj}(X,0)\approx X.\EE
By definition, the 1-marked genus~1 degree~0 fixed $\fj$ absolute GW-invariant
with primary insertion $1\!\in\!H^*(X)$ is the (signed) number of solutions
$u\!:\Si\!\lra\!X$ of
\BE{Jnueq_e}\dbar_{J,\fj}u\big|_z=\nu\big(z,u(z)\big)\quad\forall~z\in\Si, \qquad
u_*[\Si]=0\in H_2(X;\Z),\EE
for a generic element
$$\nu\in \Ga_{\fj}(X,J)\equiv\Ga\big(\Si\!\times\!X,(T^*\Si)^{0,1}\!\otimes_{\C}\!TX\big).$$
The projection~$\bar\nu$ of this element to the cokernel of $D_{J,u}$ for 
each $u\!\in\!\Hol_{\fj}(X,0)$ induces a transverse section of 
the obstruction bundle~\eref{g1Obs_e}.
The solutions of~\eref{Jnueq_e} correspond to the zeros of~$\bar\nu$,
as the obstruction to solving~\eref{Jnueq_e} vanishes at these points.
Thus, the number of solutions of~\eref{Jnueq_e} is
$$\blr{e(\Obs),\Hol_{\fj}(X,0)}=\blr{c_1(TX),X}=\chi(X).$$
If $\fj\!\in\!H^2(\ov\cM_{1,1})$ is the Poincare dual of a generic element,
the absolute GW-invariant $\GW_{1,0}^X(\fj;1)$ is half this number,
because the group of automorphisms of a generic element of~$\ov\cM_{1,1}$ is~$\Z_2$.
This establishes the first equality in~\eref{deg0_e1}.\\

\noindent
For maps of degree $A\!=\!0$, $A\!\cdot\!V\!=\!0$ and so the only compatible contact vector
is the length~0 vector, which we denote by~$()$.
By definition, the 1-marked genus~1 degree~0 fixed $\fj$ GW-invariant relative to~$V$ 
with contact vector~$()$ and primary insertion $1\!\in\!H^*(X)$ is the number of 
solutions $u\!:\Si\!\lra\!X$ of
\BE{Jnueq_e2}\dbar_{J,\fj}u\big|_z=\nu\big(z,u(z)\big)\quad\forall~z\in\Si, \qquad
u_*[\Si]=0\in H_2(X;\Z), \qquad u(\Si)\not\subset V,\EE
for a generic $\nu\!\in\!\Ga_{\fj}^V(X,J)$, where
$$\Ga_{\fj}^V(X,J)\subset \Ga_{\fj}(X,J)$$
is the subspace of elements $\nu$ such~that
$$\nu|_{\Si\times V}\in  \Ga_{\fj}(V,J|_V), \quad
\wt\na_w\nu+J\wt\na_{Jw}\nu\in (T^*\Si)^{0,1}\!\otimes_{\C}\!T_xV 
\quad\forall~w\!\in\!T_xX,~x\!\in\!V.$$
By the first assumption above, the number of maps $u\!:\Si\!\lra\!V\!\subset\!X$
that satisfy the first two conditions in~\eref{Jnueq_e2} and fail the third is~$\chi(V)$.
The total number of  maps $u\!:\Si\!\lra\!X$
that satisfy the first two conditions in~\eref{Jnueq_e2} is $\chi(X)$,
as in the previous paragraph.
Thus, the number of solutions of~\eref{Jnueq_e2} is $\chi(X)\!-\!\chi(V)$. 
Similarly to the previous paragraph,
the relative GW-invariant $\GW_{1,0;()}^{X,V}(\fj;1)$ is half this number.
This establishes the second equality in~\eref{deg0_e1}.\\

\noindent
With $\al$ as in~\eref{deg0_e2}, let $Y\!\subset\!X$ be a generic representative
of the Poincare dual of~$\al$.
Since every degree~0 $J$-holomorphic map is constant,
$$\ov\fM_{1,1}(X,0)=\ov\cM_{1,1}\!\times\! X.$$
Similarly to the previously case, the obstruction bundle in this case is isomorphic~to 
\BE{g1Obs_e2}\Obs=\pi_1^*\E^*\!\otimes\!\pi_2^*TX\lra \ov\cM_{1,1}\!\times\!X,\EE
where $\E\!\lra\!\ov\cM_{1,1}$ is the Hodge line bundle of holomorphic differentials.
Its first chern class, $\la\!\equiv\!c_1(\E)$, satisfies
\BE{M11la_e}\blr{\la,\ov\cM_{1,1}}=\frac{1}{24}\,.\EE
By definition, the 1-marked genus~1 degree~0 absolute GW-invariant
with primary insertion~$\al$ is the (signed) number of solutions
$u\!:(\Si,\fj)\!\lra\!X$ of
\BE{Jnueq_e3}\dbar_{J,\fj}u\big|_z=\nu\big(z,u(z)\big)\quad\forall~z\in\Si, \qquad
u_*[\Si]=0\in H_2(X;\Z),\qquad u(x_1)\in Y,\EE
where $\Si$ is a smooth torus, $x_1\!\in\!\Si$ is the marked point,
$\nu$ is a generic element of
$$\Ga_{1,1}(X,J)\equiv\Ga\big(\cU_{1,1}\!\times\!X,(T^*\cU_{1,1})^{0,1}\!\otimes_{\C}\!TX\big),$$
and $\cU_{1,1}\!\lra\!\ov\cM_{1,1}$ is the universal curve.
Similarly to the case considered above, $\nu$ induces a transverse section~$\bar\nu$
of the obstruction bundle~\eref{g1Obs_e2}.
The solutions of~\eref{Jnueq_e3} correspond to the zeros of~$\bar\nu$
with $u(x_1)\!\in\!Y$.
Thus, 
\BE{GWg1X_e}\begin{split}
\GW_{1,0}^X(\al)&=\blr{e(\Obs),\ov\cM_{1,1}\!\times\!Y}
=-\blr{\la,\ov\cM_{1,1}}\blr{c_{n-1}(X),Y}\\
&=-\frac{1}{24}\blr{\al\,c_{n-1}(X),X}\,.
\end{split}\EE
This establishes the first equality in~\eref{deg0_e2}.\\

\noindent
By definition, the 1-marked genus~1 degree~0 GW-invariant
relative to~$V$ with contact vector~$()$ and
primary insertion~$\al$ is the number of solutions
$u\!:(\Si,\fj)\!\lra\!X$ of
\BE{Jnueq_e4}\dbar_{J,\fj}u\big|_z=\nu\big(z,u(z)\big)\quad\forall~z\in\Si, \quad
u_*[\Si]=0\in H_2(X;\Z),\quad u(x_1)\in Y,\quad u(\Si)\not\subset V,\EE
for a generic $\nu\!\in\!\Ga_{1,1}^V(X,J)$, where
$\Ga_{1,1}^V(X,J)\!\subset\!\Ga_{1,1}(X,J)$
is the subspace of elements $\nu$ satisfying~\eref{nuVrestr_e}.
By the first assumption in~\eref{nuVrestr_e} and previous paragraph, 
the number of maps $u\!:\Si\!\lra\!V\!\subset\!X$
that satisfy the first three conditions in~\eref{Jnueq_e4} and fail the fourth~is
\BE{GWg1V_e}\GW_{1,0}^V(\PD_V(V\!\cap\!Y))
=-\frac{1}{24}\blr{\PD_V(V\!\cap\!Y)\,c_{n-2}(Y),Y}
=-\frac{1}{24}\blr{\al|_V\,c_{n-2}(Y),Y}\,.\EE
Since the total number of  maps $u\!:\Si\!\lra\!X$
that satisfy the first three conditions in~\eref{Jnueq_e4} is $\GW_{1,0}^X(\al)$,
$\GW_{1,0;()}^{X,V}(\al)$ is the difference of~\eref{GWg1X_e} and~\eref{GWg1V_e},
as claimed in the second equality in~\eref{deg0_e2}.

\begin{rmk}\label{g1_rmk}
Strictly speaking, the arguments in the last two paragraphs should be applied
to the universal curve~$\chU_{1,1}$ over the moduli space $\chM_{1,1}$
of 1-marked genus~1 curves with Prym structures in place of~$\cU_{1,1}$
and the resulting numbers should then be divided by the order of 
the covering~\eref{PrymCov_e} with $(g,k)\!=\!(1,1)$.
This nuance is taken into account by $\E\!\lra\!\ov\cM_{1,1}$ being
a line orbi-bundle over an orbifold with the chern class given by~\eref{M11la_e}.
\end{rmk}

\noindent
The absolute invariant $\GW_{1,0}^X(\fj;1)$ can also be computed using 
the same framework as $\GW_{1,0}^X(\al)$.
If $\si\!\in\!\ov\cM_{1,1}$ represents the Poincare dual of~$\fj$,
\eref{GWg1X_e} becomes
$$\GW_{1,0}^X(\fj;1)=\blr{e(\Obs),[\si]\!\times\!X}
=\blr{1,[\si]}\blr{c_n(X),X}=\frac12 \chi(X).$$
Below we recall a similar framework for computing the relative invariants
in the algebraic category, based on \cite[Section~8]{Jun3},
and note that it applies equally well in the symplectic category.\\

\noindent
If $X$ is a complex manifold and $V\!\subset\!X$ is a complex submanifold,
the sheaf $\cO_X(TX)$ of holomorphic vector fields contains
the subsheaf $\cO_X(TX(-\log V))$ of vector fields with values in~$TV$ 
along~$V$.
If $(z_1,\ldots,z_n)$ is a  coordinate chart on $U\!\subset\!X$ such that 
$U\!\cap\!V$ is the slice $z_n\!=\!0$,
$\cO_U(TX(-\log V))$ is freely generated by the vector fields
$$\frac{\partial}{\partial z_1},\,\ldots,\,\frac{\partial}{\partial z_{n-1}},\,
z_n\frac{\partial}{\partial z_n}\,.$$
Thus, $\cO_X(TX(-\log V))$ is a locally free sheaf of rank~$n$,
i.e.~the sheaf of a holomorphic sections of a holomorphic vector bundle $TX(-\log V)$ of rank~$n$.
In the symplectic category, such a vector bundle can be constructed using 
the Symplectic Neighborhood Theorem \cite[Theorem~3.30]{MS1};
the resulting complex vector bundle is well-defined up to equivalence 
by the deformation equivalence class of~$\om$ as a symplectic form on~$X\!\supset\!V$.

\begin{lmm}\label{TXlogV_lmm}
Suppose $(X,\om)$ is a compact symplectic manifold of real dimension~$2n$ and 
$V\!\subset\!X$ is a compact symplectic manifold.
If $\al\!\in\!H^{2(n-k)}(X)$, then
$$\blr{\al\,c_k(TX(-\log V)),X}=\blr{\al\,c_k(X),X}-\blr{\al\,c_{k-1}(V),V}\,.$$
\end{lmm}

\begin{proof}
By definition, there is a short exact sequence of sheaves
\BE{TXlogV_e1}
0\lra \cO_X(TX(-\log V))\lra \cO_X(TX)\lra \cO_V(V)\lra 0,\EE
where the second non-trivial homomorphism is the restriction to~$V$ followed 
by the projection to the normal bundle~$\cN_XV$,
which equals to the restriction of the line bundle $\cO_X(V)$ to~$V$.
Combining~\eref{TXlogV_e1} with the short exact sequence
$$0\lra \cO_X\lra \cO_X(V)\lra \cO_V(V)\lra 0,$$
we find that
\begin{equation*}\begin{split}
c\big(\cO_X(TX(-\log V))\big)
&=c\big(\cO_X(TX)\big)\,c\big(\cO_V(V)\big)^{-1}
=c\big(\cO_X(TX)\big)\,\big(c(\cO_X(V))c(\cO_X)^{-1}\big)^{-1}\\
&=c(X)\,\big(1\!+\!\PD_X(V)\big)^{-1}\,.
\end{split}\end{equation*}
Thus,
\BE{TXlogV_e5}\begin{split}
\blr{\al\,c_k(TX(-\log V)),X}
&=\blr{\al\,c_k(X),X}-\sum_{i=0}^{k-1}(-1)^i\blr{\al\,c_{k-1-i}(X)\,(\PD_XV)^{1+i},X}\\
&=\blr{\al\,c_k(X),X}-
\sum_{i=0}^{k-1}(-1)^i\blr{\al\,c_{k-1-i}(X)\,c_1(\cN_XV)^i,V}\,.
\end{split}\EE
Since $c(V)=c(TX)|_Vc(\cN_XV)^{-1}$, the claim follows from~\eref{TXlogV_e5}.
\end{proof}

\noindent 
In the projective setting, the analogue of the obstruction bundle~\eref{g1Obs_e2} 
for the relative moduli space~is
$$\Obs^V=\pi_1^*\E^*\!\otimes\!\pi_2^*TX(-\log V)\lra \ov\cM_{1,1}\!\times\!X\,;$$
see \cite[Section~8]{Jun3}.
In the symplectic setting, the substance of the first restriction in~\eref{nuVrestr_e}
is that~$\nu$ induces a section of~$\Obs^V$.
If $\nu$ is generic, subject to the conditions in~\eref{nuVrestr_e},
this section is transverse to the zero set everywhere and when restricted 
to $\ov\cM_{1,1}\!\times\!V$.
Thus, it has no zeros along~$V$ and the two relative invariants in Example~\ref{deg0_eg}
are given~by
\begin{equation*}\begin{split}
\GW_{1,0;()}^{X,V}(\fj;1)&=\blr{e(\Obs^V),[\si]\!\times\!X}
=\blr{1,[\si]}\blr{c_n(TX(-\log V)),X}\,\\
\GW_{1,0;()}^{X,V}(\al)&=\blr{e(\Obs^V),\ov\cM_{1,1}\!\times\!Y}
=-\blr{\la,\ov\cM_{1,1}}\blr{c_{n-1}(TX(-\log V)),Y}.
\end{split}\end{equation*}
The second equalities in~\eref{deg0_e1} and~\eref{deg0_e2} now follow 
from Lemma~\ref{TXlogV_lmm} and~\eref{M11la_e}.\\

\begin{figure}
\begin{pspicture}(0,0.7)(11,3.2)
\psset{unit=.3cm}
\psline[linewidth=.1](10,6)(22,6)\rput(9.5,9){$X$}\rput(8.8,3){$\P_XV$}
\rput(9.5,6.7){\sm{$V$}}\rput(8.7,5.1){\sm{$\P_{X,0}V$}}
\pscircle*(16,9){.25}\rput(18,9){\sm{$(1,0)$}}
\psline[linewidth=.1](32,6)(44,6)\rput(31.5,9){$X$}\rput(30.8,3){$\P_XV$}
\rput(31.5,6.7){\sm{$V$}}\rput(30.7,5.1){\sm{$\P_{X,0}V$}}
\pscircle*(38,3){.25}\rput(40,3){\sm{$(1,0)$}}
\end{pspicture}
\caption{The two bipartite graphs~$\Ga$ contributing to the genus~1 degree~0 GW-invariants 
of $X$ via the symplectic sum formula applied to~\eref{Xdecomp_e}.}
\label{deg0_fig}
\end{figure}
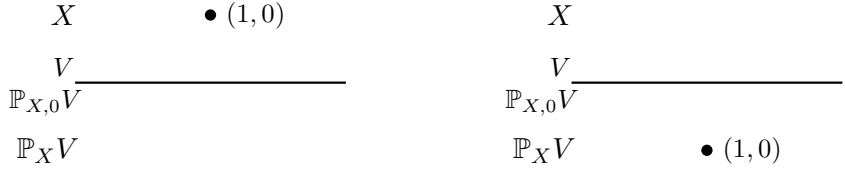

\noindent
Finally, we note that the two pairs of the GW-invariants in Example~\ref{deg0_eg}
are consistent with the symplectic sum formula as stated in the second-to-last equation
on page~201 in~\cite{Jun2} and applied to the decomposition~\eref{Xdecomp_e}.
Since the degree $A\!=\!0$ in this case, there are only two types of bipartite graphs~$\Ga$
as in Section~\ref{SympSum_subs} to sum~over: the two possible one-vertex graphs;
see Figure~\ref{deg0_fig}.
The symplectic sum formula in these cases gives
\begin{alignat}{1}
\label{deg0sum_e1}
\GW_{1,0}^X(\fj;1)&=\GW_{1,0;()}^{X,V}(\fj;1)+\GW_{1,0;()}^{\P_XV,\P_{X,\i}V}(\fj;1),\\
\label{deg0sum_e2}
\GW_{1,0}^X(\al)&=\GW_{1,0;()}^{X,V}(\al)+
\GW_{1,0;()}^{\P_XV,\P_{X,\i}V}\big(\pi_{X,V}^*(\al|_V)\big),
\end{alignat}
with $\P_{X,\i}V\!\subset\!\P_XV$ and $\pi_{X,V}\!:\P_{X,\i}V\!\lra\!V$
as in~\eref{PXVdfn_e} and~\eref{PXVsecdfn_e}.
According to the second equality in~\eref{deg0_e1},
\begin{equation*}\begin{split}
\GW_{1,0;()}^{X,V}(\fj;1)&=\frac{\chi(X)-\chi(V)}{2}, \\
\GW_{1,0;()}^{\P_XV,\P_{X,\i}V}(\fj;1)&=\frac{\chi(\P_XV)-\chi(\P_{X,\i}V)}{2}
=\frac{2\chi(V)-\chi(V)}{2}\,.
\end{split}\end{equation*}
Thus, \eref{deg0sum_e1} is consistent with the first equality in~\eref{deg0_e1}.
According to the second equality in~\eref{deg0_e2},
\begin{equation*}\begin{split}
\GW_{1,0;()}^{X,V}(\al)&=-\frac{\lr{\al\,c_{n-1}(X),X}-\lr{\al|_V\,c_{n-2}(V),V}}{24},\\
\GW_{1,0;()}^{\P_XV,\P_{X,\i}V}\big(\pi_{X,V}^*(\al|_V)\big)&=
-\frac{\lr{\pi_{X,V}^*(\al|_V)\,c_{n-1}(\P_XV),\P_XV}
-\lr{\al|_V\,c_{n-2}(V),V}}{24}\\
&=\frac{2\,\lr{\al|_V\,c_{n-2}(V),V}
-\lr{\al|_V\,c_{n-2}(V),V}}{24}\,.
\end{split}\end{equation*}
Thus, \eref{deg0sum_e2} is consistent with the first equality in~\eref{deg0_e2}.

\subsection{Genus 2 degree 1 invariants of~$\P^1$}
\label{P1eg_subs}

\noindent
We establish the second equality in~\eref{P1_e} by applying the symplectic sum formula, as stated 
in the second-to-last equation on page~201 in~\cite{Jun2}, to 
the absolute GW-invariant in~\eref{P1_e} via the decomposition~\eref{Xdecomp_e} with
$$X=\P^1, \qquad V=V_{\de}\equiv\{p_1,\ldots,p_{\de}\}, \qquad
\P_XV=\{1,\ldots,d\}\times\P^1\,.$$ 
We will make use of some top intersection numbers on the Deligne-Mumford spaces
$\ov\cM_2$ and $\ov\cM_{2,1}$, as summarized in Tables~\ref{M2and3_tbl} 
and~\ref{M21_tbl}.
The numbers for $\ov\cM_3$ and $\ov\cM_{3,1}$, appearing in Tables~\ref{M2and3_tbl} 
and~\ref{M31_tbl}, will be used in Section~\ref{P4eg_subs}.
These numbers can be obtained from C.~Faber's computer program,
which implements the formula described in~\cite{Faber}.\\

\begin{figure}
\begin{pspicture}(0,0)(11,4)
\psset{unit=.3cm}
\rput(9.5,9){$\P^1$}\rput(8.5,3){$\{1\}\times\!\P^1$}\rput(23.5,3){$\{7\}\times\!\P^1$}
\rput(10.7,6){\sm{$p_1$}}\rput(21.4,6){\sm{$p_7$}}
\pscircle*(16,9){.25}
\psline[linewidth=.05](16,9)(15,10.5)\psline[linewidth=.05](16,9)(17,10.5)
\rput(15,11.2){\sm{1}}\rput(17,11.2){\sm{2}}\rput(18.2,9){\sm{$(2,1)$}}
\pscircle*(11.5,6){.2}\pscircle*(13,6){.2}\pscircle*(14.5,6){.2}\pscircle*(16,6){.2}
\pscircle*(17.5,6){.2}\pscircle*(19,6){.2}\pscircle*(20.5,6){.2}
\psline[linewidth=.05](16,9)(11.5,6)\psline[linewidth=.05](16,9)(13,6)
\psline[linewidth=.05](16,9)(14.5,6)\psline[linewidth=.05](16,9)(16,6)
\psline[linewidth=.05](16,9)(17.5,6)\psline[linewidth=.05](16,9)(19,6)
\psline[linewidth=.05](16,9)(20.5,6)
\psline[linewidth=.05](11.5,6)(11.5,3)\psline[linewidth=.05](13,6)(13,3)
\psline[linewidth=.05](14.5,6)(14.5,3)\psline[linewidth=.05](16,6)(16,3)
\psline[linewidth=.05](17.5,6)(17.5,3)\psline[linewidth=.05](19,6)(19,3)
\psline[linewidth=.05](20.5,6)(20.5,3)
\pscircle*(11.5,3){.25}\pscircle*(13,3){.25}\pscircle*(14.5,3){.25}\pscircle*(16,3){.25}
\pscircle*(17.5,3){.25}\pscircle*(19,3){.25}\pscircle*(20.5,3){.25}
\rput(16,0){\rnode{A}{\sm{$(0,0)$}}}
\pnode(11.5,3){B1}\pnode(13,3){B2}\pnode(14.5,3){B3}\pnode(16,3){B4}
\pnode(17.5,3){B5}\pnode(19,3){B6}\pnode(20.5,3){B7}
\ncline[linewidth=.03,nodesep=3pt]{->}{A}{B1}
\ncline[linewidth=.03,nodesep=3pt]{->}{A}{B2}
\ncline[linewidth=.03,nodesep=3pt]{->}{A}{B3}
\ncline[linewidth=.03,nodesep=3pt]{->}{A}{B4}
\ncline[linewidth=.03,nodesep=3pt]{->}{A}{B5}
\ncline[linewidth=.03,nodesep=3pt]{->}{A}{B6}
\ncline[linewidth=.03,nodesep=3pt]{->}{A}{B7}
\rput(31.5,9){$\P^1$}\rput(30.5,3){$\{1\}\times\!\P^1$}\rput(45.5,3){$\{7\}\times\!\P^1$}
\rput(32.7,6){\sm{$p_1$}}\rput(43.4,6){\sm{$p_7$}}
\pscircle*(38,9){.25}
\psline[linewidth=.05](38,9)(37,10.5)\psline[linewidth=.05](38,9)(39,10.5)
\rput(37,11.2){\sm{1}}\rput(39,11.2){\sm{2}}\rput(40.2,9){\sm{$(0,1)$}}
\pscircle*(33.5,6){.2}\pscircle*(35,6){.2}\pscircle*(36.5,6){.2}\pscircle*(38,6){.2}
\pscircle*(39.5,6){.2}\pscircle*(41,6){.2}\pscircle*(42.5,6){.2}
\psline[linewidth=.05](38,9)(33.5,6)\psline[linewidth=.05](38,9)(35,6)
\psline[linewidth=.05](38,9)(36.5,6)\psline[linewidth=.05](38,9)(38,6)
\psline[linewidth=.05](38,9)(39.5,6)\psline[linewidth=.05](38,9)(41,6)
\psline[linewidth=.05](38,9)(42.5,6)
\psline[linewidth=.05](33.5,6)(33.5,3)\psline[linewidth=.05](35,6)(35,3)
\psline[linewidth=.05](36.5,6)(36.5,3)\psline[linewidth=.05](38,6)(38,3)
\psline[linewidth=.05](39.5,6)(39.5,3)\psline[linewidth=.05](41,6)(41,3)
\psline[linewidth=.05](42.5,6)(42.5,3)
\pscircle*(33.5,3){.25}\pscircle*(35,3){.25}\pscircle*(36.5,3){.25}\pscircle*(38,3){.25}
\pscircle*(39.5,3){.25}\pscircle*(41,3){.25}\pscircle*(42.5,3){.25}
\rput(37.25,0){\rnode{A}{\sm{$(0,0)$}}}\rput(42.5,0){\rnode{A2}{\sm{$(2,0)$}}}
\pnode(33.5,3){B1}\pnode(35,3){B2}\pnode(36.5,3){B3}\pnode(38,3){B4}
\pnode(39.5,3){B5}\pnode(41,3){B6}\pnode(42.5,3){B7}
\ncline[linewidth=.03,nodesep=3pt]{->}{A}{B1}
\ncline[linewidth=.03,nodesep=3pt]{->}{A}{B2}
\ncline[linewidth=.03,nodesep=3pt]{->}{A}{B3}
\ncline[linewidth=.03,nodesep=3pt]{->}{A}{B4}
\ncline[linewidth=.03,nodesep=3pt]{->}{A}{B5}
\ncline[linewidth=.03,nodesep=3pt]{->}{A}{B6}
\ncline[linewidth=.03,nodesep=3pt]{->}{A2}{B7}
\end{pspicture}
\caption{Bipartite graphs~$\Ga$ that contribute to the absolute GW-invariant in~\eref{P1_e}
via the symplectic sum decomposition with respect to~$V_7\!=\!\{p_1,\ldots,p_7\}$.}
\label{P1eg_fig}
\end{figure}

\noindent
Since the Poincare duals of the two primary insertions in~\eref{P1_e}
vanish on the divisor~$V_{\de}$
(the constraining points can be chosen to be distinct from the $\de$ points in~$V_{\de}$),
$k_v\!=\!0$ for all $v\!\in\!\Ga_V$ (the marked points stay on the $X$-side)
if $\Ga$ is a bipartite graph as in Section~\ref{SympSum_subs}
contributing  to the absolute GW-invariant in~\eref{P1_e}.
Furthermore, $A_v\!=\!1$ for the unique vertex $v\!\in\!\Ga_X$ 
and all edge labels are~1 in this case
(because the curve on the $X$-side is of degree~1 and so meets each 
point in the divisor with order~1).
By Section~\ref{SympSum_subs} or Lemma~\ref{relGW_lmm} (separately), 
there are thus only two types of graphs~$\Ga$ contributing to 
the absolute GW-invariant in~\eref{P1_e}: 
\begin{enumerate}[label=(\arabic*),leftmargin=*]

\item $(g_v,A_v,k_v)\!=\!(0,0,0)$ for all $v\!\in\!\Ga_V$ and 

\item  the $\de$ graphs with $(g_v,A_v,k_v)\!=\!(2,0,0)$ for
one element $v\!\in\!\Ga_V$ and
$(g_v,A_v,k_v)\!=\!(0,0,0)$ for the remaining $\de\!-\!1$ elements $v\!\in\!\Ga_V$; 

\end{enumerate}
see Figure~\ref{P1eg_fig}.
There are other bipartite graphs $\Ga$, but they all contain a vertex $v\!\in\!\Ga_V$
with $g_v\!=\!1$;
by Section~\ref{SympSum_subs} or Lemma~\ref{relGW_lmm}, such a graph does not contribute
to an absolute GW-invariant via the symplectic sum formula.
In the setup of Section~\ref{DirComp_subs}, such  graphs correspond to configurations with genus~1
components sinking into the divisor; since there are no higher-genus components sinking
into the divisor in this case, the argument in Section~\ref{DirComp_subs} also implies that
such a configuration does not contribute.\\

\noindent
Thus, by the symplectic sum formula,
\BE{P1eg_e}
\GW_{2,1}^{\P^1}(\ka^4;\pt,\pt)
=\frac{1}{\de!}\GW_{2,1;\1_{\de}}^{\P^1,V_{}}(\ka^4;\pt,\pt)
+\frac{\de}{\de!}\sum_i
\GW_{0,1;\1_{\de}}^{\P^1,V_{\de}}(\ka_i;\pt,\pt)
\GW_{2,1;(1)}^{\P^1,\pt}(\ka_i';1)\,,
\EE
with $\ka_i\!\in\!H^*(\ov\cM_{0,2+\de})$ and $\ka_i'\!\in\!H^*(\ov\cM_{2,1})$ given~by
$$\gl^*\ka^4=\sum_i \ka_i\otimes\ka_i'
\in H^*(\ov\cM_{0,2+\de})\otimes H^*(\ov\cM_{2,1}) = H^*(\ov\cM_{0,2+\de}\!\times\!\ov\cM_{2,1}),$$
where 
$$\gl\!:\ov\cM_{0,2+\de}\!\times\!\ov\cM_{2,1}\lra \ov\cM_{2,2}\,,$$
is the morphism obtained by forgetting the last $\de\!-\!1$ points on the genus~0 curve
and identifying the marked point of the genus~2 curve
with the third marked point on the genus~0 curve.
Since $\ka$ is the Poincare dual of the divisor represented by the bottom right diagram
in Figure~\ref{st_fig}, it follows~that 
\BE{ka4pullback_e}\sum_i \ka_i\otimes\ka_i'\equiv \gl^*\ka^4 = 1\otimes\psi_1^4,\EE
where $\psi_1\!\in\!H^*(\ov\cM_{2,1})$ is the chern class of 
the universal cotangent line bundle.
By Theorem~\ref{main_thm},
\BE{P1eg_e2}
\frac{1}{\de!}\GW_{0,1;\1_{\de}}^{\P^1,V_{\de}}(1;\pt,\pt)
\equiv \frac{1}{\de!}\GW_{0,1;\1_{\de}}^{\P^1,V_{\de}}(\pt,\pt)
=\GW_{0,1}^{\P^1}(\pt,\pt)=1.\EE
Combining~\eref{P1eg_e} with~\eref{ka4pullback_e}, \eref{P1eg_e2}, 
and Lemma~\ref{M21_lmm} below,
we conclude that 
$$\GW_{2,1}^{\P^1}(\ka^4;\pt,\pt)
=\frac{1}{\de!}\GW_{2,1;\1_{\de}}^{\P^1,V_{\de}}(\ka^4;\pt,\pt)
+\de\blr{\psi_1^4,\ov\cM_{2,1}}\,.$$ 
The second equality in~\eref{P1_e} now follows from the first column in Table~\ref{M21_tbl}.

\begin{table}
\begin{center}
\begin{tabular}{||c|c||cc||c|c|c|c|c|c|c|c||}
\hhline{==~~=======}
$\la_1^3$& $\la_1\la_2$&&& $\la_1^6$& $\la_1^4\la_2$& $\la_1^3\la_3$& 
$\la_1^2\la_2^2$& $\la_1\la_2\la_3$& $\la_2^3$&  $\la_3^2$\\
\cline{1-2}\cline{5-11}
$\frac{1}{2880}$& $\frac{1}{5760}$&&& $\frac{1}{90720}$& $\frac{1}{181440}$&
$\frac{1}{725760}$& $\frac{1}{362880}$&  $\frac{1}{1451520}$& $\frac{1}{725760}$& 0\\
\hhline{==~~=======}
\end{tabular}
\end{center}
\caption{The top intersections of $\la$-classes on $\ov\cM_2$ and $\ov\cM_3$.}
\label{M2and3_tbl}
\end{table}

\begin{table}
\begin{center}
\begin{tabular}{||c|c|c|c|c|c||}
\hline\hline
$\psi_1^4$& $\psi_1^3\la_1$& $\psi_1^2\la_1^2$& $\psi_1^2\la_2$& $\psi_1\la_1^3$& 
$\psi_1\la_1\la_2$\\
\hline
$\frac{1}{1152}$& $\frac{1}{480}$& $\frac{7}{2880}$& $\frac{7}{5760}$&
$\frac{1}{1440}$& $\frac{1}{2880}$\\
\hline\hline
\end{tabular}
\end{center}
\caption{The top intersections of $\la$-classes and $\psi_1$ on $\ov\cM_{2,1}$.}
\label{M21_tbl}
\end{table}

\begin{lmm}[C.-C.~Liu]\label{M21_lmm}
If $\st\!:\ov\fM_{2,0;(1)}^{\pt}(\P^1,1)\!\lra\!\ov\cM_{2,1}$ is the forgetful morphism
dropping the map to~$\P^1$, then
\BE{M21_e}\st_*\big[\ov\fM_{2,0;(1)}^{\pt}(\P^1,1)\big]^{\vir}
=\big[\ov\cM_{2,1}\big].\EE\\
\end{lmm}

\noindent
Since $\ov\cM_{2,1}$ is smooth (as an orbifold) and irreducible,
\eref{M21_e} is equivalent~to 
\BE{M21_e2} \blr{\st^*\si,[\ov\fM_{2,0;(1)}^{\pt}(\P^1,1)]^{\vir}}=1,\EE
where $\si\!\in\!H^8(\ov\cM_{2,1})$ is the Poincare dual of 
a generic element $(\Si,x_1)$ of~$\ov\cM_{2,1}$.
We give two proofs of~\eref{M21_e2} below.
The first argument applies the virtual localization theorems of \cite{GP,GV}
as in \cite[Chapter~27]{MirSym}.
The second proof applies the obstruction analysis of~\cite{ObsAnal}
as in \cite[Section~4]{g2n2and3}.

\begin{proof}[{\bf{\emph{Proof 1 of \eref{M21_e2}}}}]
We use the standard $(\C^*)$-action on~$\P^1$.
It has two fixed points, 
$$p_1=[1,0]  \qquad\hbox{and}\qquad p_2=[0,1],$$
and lifts linearly to an action on $\cO_{\P^1}(1)\!\lra\!\P^1$.
As in \cite[Chapter~27]{MirSym}, we let
$$\al_i=c_1\big(\cO_{\P^1}(1)\big)\big|_{p_i}\in 
H_{(\C^*)^2}^*\equiv H^*\big(B((\C^*)^2)\big)
=H^*(\P^{\i}\!\times\!\P^{\i})=\C[\al_1,\al_2]\,.$$
The fixed loci of the induced action on $\ov\fM_{2,0;(1)}^{\,p_2}(\P^1,1)$
consist of maps sending components of positive genus to either the fixed point~$p_1$
or the rubber~$\P^1$ attached to the fixed point~$p_2$.
The three graphs describing these fixed loci in the notation of \cite[Chapter~27]{MirSym}
are shown in Figure~\ref{P1loc_fig1}.
In these diagrams, the first vertex label indicates the corresponding fixed point of~$\P^1$,
while the second indicates the genus of the component taken there, if any.
The edge degree is~1 in all cases, corresponding to the degree~1 cover from~$\P^1\!\lra\!\P^1$. \\

\begin{figure}
\begin{pspicture}(-.2,1.2)(11,3)
\psset{unit=.3cm}
\pscircle*(16,9){.2}\rput(17.8,9){\sm{$(1,2)$}}
\pscircle*(16,5){.2}\rput(16.8,5){\sm{$2$}}
\psline[linewidth=.05](16,9)(16,5)
\psline[linewidth=.05](16,5)(15,4)\rput(14.5,4){$1$}
\pscircle*(26,9){.2}\rput(27.8,9){\sm{$(1,1)$}}
\pscircle*(26,5){.2}\rput(27.8,5){\sm{$(2,1)$}}
\psline[linewidth=.05](26,9)(26,5)
\psline[linewidth=.05](26,5)(25,4)\rput(24.5,4){$1$}
\pscircle*(36,9){.2}\rput(36.8,9){\sm{$1$}}
\pscircle*(36,5){.2}\rput(37.8,5){\sm{$(2,2)$}}
\psline[linewidth=.05](36,9)(36,5)
\psline[linewidth=.05](36,5)(35,4)\rput(34.5,4){$1$}
\end{pspicture}
\caption{The three graphs describing the $(\C^*)^2$-fixed loci
of $\ov\fM_{2,0;(1)}^{\,p_2}(\P^1,1)$}
\label{P1loc_fig1}
\end{figure}
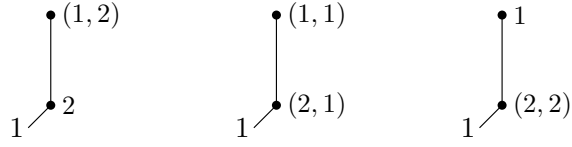

\noindent
The morphism~$\st$ takes the fixed locus represented by the middle diagram in Figure~\ref{P1loc_fig1}
to the closure in $\ov\cM_{2,1}$ of the locus consisting of two-component maps.
Thus, $\st^*\si$ vanishes on this locus and the middle diagram does not contribute
to~\eref{M21_e2} via the virtual localization theorem of~\cite{GP}.\\

\noindent 
The locus represented by the first diagram in Figure~\ref{P1loc_fig1}
is isomorphic to $\ov\cM_{2,1}$ and is cut down by $\st^*\si$ to a single point.
The space of deformations of this locus consists of moving the node and
of smoothing the node;
after restricting to the cut-down space, 
the equivariant chern class of both of these line bundles  equal to 
the equivariant chern class of $T\P^1$ at~$p_1$, 
which is $\al_1\!-\!\al_2$ in this case; see \cite[Exercise 27.1.3]{MirSym}.
The obstruction bundle after cutting down by $\st^*\si$ is 
$$H^1(\Si;T_{p_1}\P^1)=\big(H^0(\Si;T^*\Si\!\otimes\!T_{p_1}^*\P^1)\big)^*
\approx T_{p_1}\P^1\oplus T_{p_1}\P^1\,;$$
its equivariant euler class is $(\al_1\!-\!\al_2)^2$.
Thus, the contribution of the first diagram in Figure~\ref{P1loc_fig1} to~\eref{M21_e2}~is
$$\int_{\ov\cM_{2,1}}\st^*\si\,
\frac{(\al_1\!-\!\al_2)^2}{(\al_1\!-\!\al_2)\cdot(\al_1\!-\!\al_2)}=1\,;$$ 
see \cite[(7)]{GP} or \cite[Theorem~3.6]{GV}.\\

\noindent 
The locus represented by the last diagram in Figure~\ref{P1loc_fig1}
is isomorphic~to the (rubber) moduli space
$\ov\fM_{2,0;(1),(1)}^{0,\i}(\P^1,1)_{\sim}$
of relative morphisms to the non-rigid target~$(\P^1,0,\i)$ with the standard $\C^*$-action;
see \cite[Section]{GV}.
Since the virtual dimension of this moduli space is~3,
the restriction of~$\st^*\si$ to this fixed locus vanishes.
By \cite[Theorem~3.6]{GV}, the last diagram in Figure~\ref{P1loc_fig1}
thus does not contribute to~\eref{M21_e2}.
Combining this with the conclusion of the two previous paragraphs,
we obtain~\eref{M21_e2}.
\end{proof}

\begin{proof}[{\bf{\emph{Proof 2 of \eref{M21_e2}}}}]
Let $(\Si,\fj,x_1)$ be a generic element of $\ov\cM_{2,1}$ as before.
The number~\eref{M21_e2} is the number of solutions $u\!:\Si\!\lra\!\P^1$ of
\BE{Jnueq_e5}\dbar_{J,\fj}u\big|_z=\nu\big(z,u(z)\big)\quad\forall~z\in\Si, \qquad
u_*[\Si]=[\P^1]\in H_2(\P^1;\Z),\EE
for a generic $\nu\!\in\!\Ga_{\fj}^{\pt}(\P^1,J)$, where
$$\Ga_{\fj}^{\pt}(\P^1,J)\subset 
\Ga\big(\Si\!\times\!\P^1,(T^*\Si)^{0,1}\!\otimes_{\C}\!T\P^1\big)$$
is the subspace of elements $\nu$ such~that
\BE{nuR_e}\nu|_{\Si\times\pt}=0,\qquad
\na_w\nu+J\na_{Jw}\nu=0 \quad\forall~w\!\in\!T_{\pt}\P^1.\EE
The moduli space of degree~1 holomorphic maps $(\Si,\fj)\!\lra\!\P^1$ 
and its obstruction bundle are given~by
\BE{TarBndl_e}\cH_{\Si}^{0,1}\otimes T\P^1\approx\Obs\lra \Hol_{\fj}(\P^1,1)\approx \P^1\,,\EE
where $\cH_{\Si}^{0,1}$ is the space of harmonic $(0,1)$-forms on~$\Si$.\\

\noindent
The space of deformations of the domain of the elements in $\Hol_{\fj}(\P^1,1)$
is the product of the two tangent bundles
at the node, i.e. 
\BE{DomBndl_e} T_{x_1}\Si\otimes T\P^1\approx  T\P^1\lra \P^1\,.\EE
Each smoothing parameter $\ups$ in this line bundle determines an approximately
$(J,\fj)$-holomorphic map $u_{\ups}\!:\Si\!\lra\!\P^1$;
see \cite[Section~3.3]{ObsAnal}.
The first-order term of the projection $\pi^{0,1}_{\ups,-}\dbar_{J,\fj}u_{\ups}$ 
of $\dbar_{J,\fj}u_{\ups}$ to~$\Obs$ is given~by
$$\big\{ L(\ups_1\!\otimes\!\ups_2)\big\}(\psi)=\psi_{x_1}(\ups_1)
\big\{\nd_{x_2}u\big\}(\ups_2)\in T_{u(x_2)}\P^1 \qquad\forall~\psi\in\cH^{1,0}_{\Si}\,,$$
where $x_2\!\in\!\P^1$ is the node of the rational component of the domain of the map;
see \cite[Lemma~4.5]{g2n2and3}.
Since $L$ is injective in this case, the solutions of~\eref{Jnueq_e5}
correspond to the zeros of the section~of 
\BE{AffMap_e}\Obs/\Im\,L\lra \Hol_{\fj}(\P^1,1)\EE
induced by a generic $\nu$, excluding the one with $u(x_2)\!=\!\pt$;
see the proof of \cite[Corollary~4.7]{g2n2and3}.
Thus, the number of solutions of~\eref{Jnueq_e5} is
$$\blr{e(\Obs/\Im\,L),\Hol_{\fj}(\P^1,1)}-1
=\blr{c_1(\cH_{\Si}^{0,1}\!\otimes\!T\P^1)-c_1(T_{x_1}\Si\!\otimes\!T\P^1),\P^1}-1
=1.$$
This establishes~\eref{M21_e2}. 
\end{proof}

\noindent
We next use the virtual localization theorem of~\cite{GP}
to compute the absolute invariant and the $\de\!=\!1$ case of the relative invariant
in Example~\ref{P1_eg}.
We continue with the localization setup of the first proof of~\eref{M21_e2}
and compute
\BE{M21loc_e}\begin{split} 
&\int_{[\ov\fM_{2,2}(\P^1,1)]^{\vir}}\!\!\st^*\ka^4\,
\ev_1^*\cO_{\P^1}(1\!-\!\al_2)\,\ev_2^*\cO(1\!-\!\al_2) \qquad\hbox{and}\\
&\int_{[\ov\fM_{2,2;(1)}^{p_2}(\P^1,1)]^{\vir}}\!\!\st^*\ka^4\,
\ev_1^*\cO_{\P^1}(1\!-\!\al_2)\,\ev_2^*\cO_{\P^1}(1\!-\!\al_2)\,.
\end{split}\EE
The $(\C^*)^2$-fixed loci consist of maps sending the positive-genus components 
and the absolute marked points to the fixed points~$p_1$ and~$p_2$.
Since the equivariant chern class of $\cO_{\P^1}(1\!-\!\al_2)$ vanishes at~$p_2$,
the only graphs possibly contributing to the integrals in~\eref{M21loc_e}
must have both absolute marked points sent to~$p_1$.
Since the morphism~$\st$ takes fixed loci with a positive-genus component at
both fixed points to the closure in $\ov\cM_{2,2}$ of the locus consisting
of two genus~1 curves, 
$\st^*\ka^4$ vanishes on such fixed loci as well.
The two remaining graphs possibly contributing to each of the integrals
in~\eref{M21loc_e} are shown in Figure~\ref{P1loc_fig2}.\\

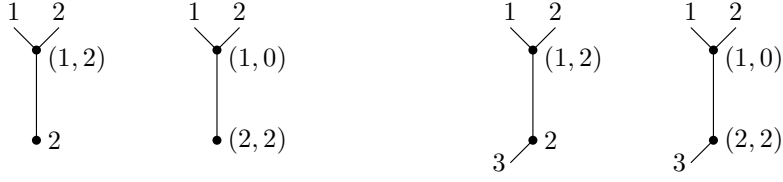
\begin{figure}
\begin{pspicture}(.6,1.2)(11,3)
\psset{unit=.3cm}
\pscircle*(14,9){.2}\rput(15.8,8.6){\sm{$(1,2)$}}
\pscircle*(14,5){.2}\rput(14.8,5){\sm{$2$}}
\psline[linewidth=.05](14,9)(14,5)
\psline[linewidth=.05](14,9)(13,10)\rput(13,10.7){\sm{$1$}}
\psline[linewidth=.05](14,9)(15,10)\rput(15,10.7){\sm{$2$}}
\pscircle*(22,9){.2}\rput(23.8,8.6){\sm{$(1,0)$}}
\pscircle*(22,5){.2}\rput(23.8,5){\sm{$(2,2)$}}
\psline[linewidth=.05](22,9)(22,5)
\psline[linewidth=.05](22,9)(21,10)\rput(21,10.7){\sm{$1$}}
\psline[linewidth=.05](22,9)(23,10)\rput(23,10.7){\sm{$2$}}
\pscircle*(36,9){.2}\rput(37.8,8.6){\sm{$(1,2)$}}
\pscircle*(36,5){.2}\rput(36.8,5){\sm{$2$}}
\psline[linewidth=.05](36,9)(36,5)
\psline[linewidth=.05](36,9)(35,10)\rput(35,10.7){\sm{$1$}}
\psline[linewidth=.05](36,9)(37,10)\rput(37,10.7){\sm{$2$}}
\psline[linewidth=.05](36,5)(35,4)\rput(34.5,4){\sm{$3$}}
\pscircle*(44,9){.2}\rput(45.8,8.6){\sm{$(1,0)$}}
\pscircle*(44,5){.2}\rput(45.8,5){\sm{$(2,2)$}}
\psline[linewidth=.05](44,9)(44,5)
\psline[linewidth=.05](44,9)(43,10)\rput(43,10.7){\sm{$1$}}
\psline[linewidth=.05](44,9)(45,10)\rput(45,10.7){\sm{$2$}}
\psline[linewidth=.05](44,5)(43,4)\rput(42.5,4){\sm{$3$}}
\end{pspicture}
\caption{The two pairs of graphs possibly contributing to the two integrals
in~\eref{M21loc_e}} 
\label{P1loc_fig2}
\end{figure}

\noindent
The locus represented by the first diagram in Figure~\ref{P1loc_fig2}
is isomorphic to $\ov\cM_{2,3}$.
The space of deformations of this locus consists of moving the node and
of smoothing the node; 
the equivariant chern classes of these line bundles are 
$\al_1\!-\!\al_2$ and $\al_1\!-\!\al_2\!-\!\psi_3$, respectively. 
The euler class of the obstruction bundle is given~by
$$e\big(\E^*\!\otimes\!T_{p_1}\P^1\big)
=\la_2-(\al_1\!-\!\al_2)\la_1+(\al_1\!-\!\al_2)^2\,.$$
By \cite[(7)]{GP}, the contribution of the first diagram in Figure~\ref{P1loc_fig2} 
to the first integral in~\eref{M21loc_e} is~thus
\BE{M21loc_e1a}\begin{split}
&\int_{\ov\cM_{2,3}}\!\!\!\st^*\ka^4\,(\al_1\!-\!\al_2)^2\,
\frac{\la_2-(\al_1\!-\!\al_2)\la_1+(\al_1\!-\!\al_2)^2}
{(\al_1\!-\!\al_2)(\al_1\!-\!\al_2\!-\!\psi_3)}
=\int_{\ov\cM_{2,3}}\!\!\!\st^*\ka^4\big(\la_2\!-\!\la_1\psi_3\!+\!\psi_3^2)\\
&\qquad
= 0+4\lr{\psi_1^3\la_1,\ov\cM_{2,1}}
-\lr{(f^*\psi_1)^3\psi_2^2,\ov\cM_{2,2}}
=4\lr{\psi_1^3\la_1,\ov\cM_{2,1}}-\lr{\psi_1^3\psi_2^2,\ov\cM_{2,2}}\,,
\end{split}\EE
where $f\!:\ov\cM_{2,2}\!\lra\!\ov\cM_{2,1}$ is the forgetful morphism.
The second equality above applies the dilaton equation \cite[Exercise~25.2.7]{MirSym}
to the middle term, while the last equality follows from \cite[Lemma~25.2.3]{MirSym}.\\

\noindent
The locus represented by the second diagram in Figure~\ref{P1loc_fig2}
is isomorphic to $\ov\cM_{2,1}$.
The space of deformations of this locus consists of moving and
smoothing the two nodes; 
the total equivariant euler class of the deformations~is
$$(\al_1\!-\!\al_2)(\al_2\!-\!\al_1)(\al_1\!-\!\al_2)(\al_2\!-\!\al_1\!-\!\psi_1).$$
The euler class of the obstruction bundle is given~by
$$e\big(\E^*\!\otimes\!T_{p_2}\P^1\big)
=\la_2-(\al_2\!-\!\al_1)\la_1+(\al_2\!-\!\al_1)^2\,.$$
By \cite[(7)]{GP}, the contribution of the second diagram in Figure~\ref{P1loc_fig2} 
to the first integral in~\eref{M21loc_e} is~thus
\BE{M21loc_e1b}\begin{split}
&\int_{\ov\cM_{2,1}}\!\!\!\st^*\ka^4\,(\al_1\!-\!\al_2)^2\,
\frac{\la_2-(\al_2\!-\!\al_1)\la_1+(\al_1\!-\!\al_2)^2}
{(\al_1\!-\!\al_2)^3(\al_1\!-\!\al_2\!+\!\psi_1)}
=\int_{\ov\cM_{2,1}}\!\!\!\psi_1^4=\frac{1}{1152}\,.
\end{split}\EE
Combining \eref{M21loc_e1a} and \eref{M21loc_e1b} with 
$$\lr{\psi_1^3\psi_2^2,\ov\cM_{2,2}}=\frac{29}{5760} $$
from C.~Faber's program, we obtain the first equality in~\eref{P1_e}.\\

\noindent
The contribution of the fixed locus of $\ov\fM_{2,2;(1)}^{p_2}(\P^1,1)]^{\vir}$
represented by third diagram in Figure~\ref{P1loc_fig2} 
to the second integral in~\eref{M21loc_e} is the same 
as of the first diagram to the the first integral in~\eref{M21loc_e}.
The locus represented by the fourth diagram in Figure~\ref{P1loc_fig2}
is isomorphic to $\ov\cM_{2,2}$.
The space of deformations of this locus consists of moving node at~$p_1$
and   smoothing both nodes; 
the total equivariant euler class of the deformations~is
$$(\al_1\!-\!\al_2)(\al_1\!-\!\al_2)(\al_2\!-\!\al_1\!-\!\psi_1).$$
The euler class of the obstruction bundle in this case is given~by
$$e\big(\E^*\!\otimes\!T_{p_2}\P^1(-p_2)\big)=\la_2\,.$$
By \cite[(7)]{GP}, the contribution of the fourth diagram in Figure~\ref{P1loc_fig2} 
to the second integral in~\eref{M21loc_e} is~thus
$$\int_{\ov\cM_{2,2}}\!\!\!\st^*\ka^4\,(\al_1\!-\!\al_2)^2\,
\frac{\la_2}{(\al_1\!-\!\al_2)^2(\al_2\!-\!\al_1\!-\!\psi_1)}=0.$$
Along with the two previous paragraphs, this provides a direct check
of the $d\!=\!1$ case of the second equality in~\eref{P1_e}.

\subsection{Genus 3 degree 1 primary invariants of~$\P^4$}
\label{P4eg_subs}

\noindent
We establish the second equality in~\eref{P4_e} by applying the symplectic sum formula, as stated 
in the second-to-last equation on page~201 in~\cite{Jun2}, to 
the absolute GW-invariant in~\eref{P4_e} via the decomposition~\eref{Xdecomp_e} with
$X\!=\!\P^4$ and $V\!=\!V_{\de}$, where $V_{\de}\!\subset\!\P^4$ 
is a smooth degree~$\de$ hypersurface.\\

\begin{figure}
\begin{pspicture}(0,0)(11,4)
\psset{unit=.3cm}
\psline[linewidth=.1](10,6)(22,6)\rput(9.5,9){$\P^4$}\rput(8.8,3){$\P_{\P^4}V_7$}
\rput(9.5,6.7){\sm{$V_7$}}\rput(8.7,5.1){\sm{$\P_{\P^4,\i}V_7$}}
\pscircle*(16,9){.25}\rput(18.2,9){\sm{$(3,1)$}}
\psline[linewidth=.05](16,9)(15,10.5)\rput(15,11.2){\sm{1}}
\pscircle*(11.5,6){.2}\pscircle*(13,6){.2}\pscircle*(14.5,6){.2}\pscircle*(16,6){.2}
\pscircle*(17.5,6){.2}\pscircle*(19,6){.2}\pscircle*(20.5,6){.2}
\psline[linewidth=.05](16,9)(11.5,6)\psline[linewidth=.05](16,9)(13,6)
\psline[linewidth=.05](16,9)(14.5,6)\psline[linewidth=.05](16,9)(16,6)
\psline[linewidth=.05](16,9)(17.5,6)\psline[linewidth=.05](16,9)(19,6)
\psline[linewidth=.05](16,9)(20.5,6)
\psline[linewidth=.05](11.5,6)(11.5,3)\psline[linewidth=.05](13,6)(13,3)
\psline[linewidth=.05](14.5,6)(14.5,3)\psline[linewidth=.05](16,6)(16,3)
\psline[linewidth=.05](17.5,6)(17.5,3)\psline[linewidth=.05](19,6)(19,3)
\psline[linewidth=.05](20.5,6)(20.5,3)
\pscircle*(11.5,3){.25}\pscircle*(13,3){.25}\pscircle*(14.5,3){.25}\pscircle*(16,3){.25}
\pscircle*(17.5,3){.25}\pscircle*(19,3){.25}\pscircle*(20.5,3){.25}
\rput(16,0){\rnode{A}{\sm{$(0,0)$}}}
\pnode(11.5,3){B1}\pnode(13,3){B2}\pnode(14.5,3){B3}\pnode(16,3){B4}
\pnode(17.5,3){B5}\pnode(19,3){B6}\pnode(20.5,3){B7}
\ncline[linewidth=.03,nodesep=3pt]{->}{A}{B1}
\ncline[linewidth=.03,nodesep=3pt]{->}{A}{B2}
\ncline[linewidth=.03,nodesep=3pt]{->}{A}{B3}
\ncline[linewidth=.03,nodesep=3pt]{->}{A}{B4}
\ncline[linewidth=.03,nodesep=3pt]{->}{A}{B5}
\ncline[linewidth=.03,nodesep=3pt]{->}{A}{B6}
\ncline[linewidth=.03,nodesep=3pt]{->}{A}{B7}
\psline[linewidth=.1](32,6)(44,6)\rput(31.5,9){$\P^4$}\rput(30.8,3){$\P_{\P^4}V_7$}
\rput(31.5,6.7){\sm{$V_7$}}\rput(30.7,5.1){\sm{$\P_{\P^4,\i}V_7$}}
\pscircle*(38,9){.25}\rput(40.2,9){\sm{$(0,1)$}}
\psline[linewidth=.05](38,9)(37,10.5)\rput(37,11.2){\sm{1}}
\pscircle*(33.5,6){.2}\pscircle*(35,6){.2}\pscircle*(36.5,6){.2}\pscircle*(38,6){.2}
\pscircle*(39.5,6){.2}\pscircle*(41,6){.2}\pscircle*(42.5,6){.2}
\psline[linewidth=.05](38,9)(33.5,6)\psline[linewidth=.05](38,9)(35,6)
\psline[linewidth=.05](38,9)(36.5,6)\psline[linewidth=.05](38,9)(38,6)
\psline[linewidth=.05](38,9)(39.5,6)\psline[linewidth=.05](38,9)(41,6)
\psline[linewidth=.05](38,9)(42.5,6)
\psline[linewidth=.05](33.5,6)(33.5,3)\psline[linewidth=.05](35,6)(35,3)
\psline[linewidth=.05](36.5,6)(36.5,3)\psline[linewidth=.05](38,6)(38,3)
\psline[linewidth=.05](39.5,6)(39.5,3)\psline[linewidth=.05](41,6)(41,3)
\psline[linewidth=.05](42.5,6)(42.5,3)
\pscircle*(33.5,3){.25}\pscircle*(35,3){.25}\pscircle*(36.5,3){.25}\pscircle*(38,3){.25}
\pscircle*(39.5,3){.25}\pscircle*(41,3){.25}\pscircle*(42.5,3){.25}
\rput(37.25,0){\rnode{A}{\sm{$(0,0)$}}}\rput(42.5,0){\rnode{A2}{\sm{$(3,0)$}}}
\pnode(33.5,3){B1}\pnode(35,3){B2}\pnode(36.5,3){B3}\pnode(38,3){B4}
\pnode(39.5,3){B5}\pnode(41,3){B6}\pnode(42.5,3){B7}
\ncline[linewidth=.03,nodesep=3pt]{->}{A}{B1}
\ncline[linewidth=.03,nodesep=3pt]{->}{A}{B2}
\ncline[linewidth=.03,nodesep=3pt]{->}{A}{B3}
\ncline[linewidth=.03,nodesep=3pt]{->}{A}{B4}
\ncline[linewidth=.03,nodesep=3pt]{->}{A}{B5}
\ncline[linewidth=.03,nodesep=3pt]{->}{A}{B6}
\ncline[linewidth=.03,nodesep=3pt]{->}{A2}{B7}
\end{pspicture}
\caption{Bipartite graphs~$\Ga$ that contribute to the absolute GW-invariant in~\eref{P4_e}
via the symplectic sum decomposition with respect to~$V_7$.}
\label{P4eg_fig}
\end{figure}

\noindent
Since the Poincare dual of the primary insertion in~\eref{P4_e}
vanishes on the hypersurface~$V_{\de}$
(the constraining point can be chosen outside of~$V_{\de}$),
$k_v\!=\!0$ for all $v\!\in\!\Ga_V$ (the marked point stays on the $X$-side)
if $\Ga$ is a bipartite graph as in Section~\ref{SympSum_subs}
contributing  to the absolute GW-invariant in~\eref{P4_e}.
Furthermore, $A_v\!=\!1$ for the unique vertex $v\!\in\!\Ga_X$.
By Section~\ref{SympSum_subs} or Lemma~\ref{relGW_lmm} (separately), 
the only graphs~$\Ga$ that may contribute to 
the absolute GW-invariant in~\eref{P4_e} satisfy
\begin{enumerate}[label=(\arabic*),leftmargin=*]

\item $(g_v,A_v,k_v)\!=\!(0,0,0)$ for all $v\!\in\!\Ga_V$ or
\item  $(g_v,A_v,k_v)\!=\!(3,0,0)$ for one element $v\!\in\!\Ga_V$ and
$(g_v,A_v,k_v)\!=\!(0,0,0)$ for the remaining elements $v\!\in\!\Ga_V$.

\end{enumerate}
There are other bipartite graphs $\Ga$, but they all contain a vertex $v\!\in\!\Ga_V$
with $g_v\!\in\!\{1,2\}$;
by Section~\ref{SympSum_subs}, such a graph does not contribute
to an absolute GW-invariant with primary insertions via the symplectic sum formula.
By Section~\ref{SympSum_subs} or Lemma~\ref{relGW_lmm},
the label of the edge leaving a vertex $v\!\in\!\Ga_V$ with $g_v\!=\!0$ 
in a contributing graph~$\Ga$ is~1 (and thus omitted in our diagrams).
By the proof of Lemma~\ref{relGW_lmm} in Section~\ref{SympSum_subs},
the same is the case if $g_v\!=\!3$;
otherwise, the fiber of the projection in~\eref{fMfibr_e} would have positive dimension,
while too many conditions would be imposed on the curve on the $X$-side.
The same conclusions can be drawn from Section~\ref{DirComp_subs}.\\ 

\noindent
In summary, there are only two graphs that may contribute to
the absolute GW-invariant in~\eref{P4_e} via the symplectic sum formula;
they are shown in Figure~\ref{P4eg_fig}.
Thus,
\BE{P4eg_e}
\GW_{3,1}^{\P^4}(\pt) = \frac{1}{\de!}\GW_{3,1;\1_{\de}}^{\P^4,V_{\de}}(\pt)
+\frac{\de}{\de!}\,\GW_{0,1;\1_{\de}}^{\P^4,V_{\de}}(1;\pt;1^{\de-1},\pt)\,
\GW_{3,F;(1)}^{\P_{\P^4}V_{\de},\P_{\P^4,\i}V_{\de}}(1;1;1)\,,\EE
where $F\!\in\!H_2(\P_{\P^4}V_{\de};\Z)$ is the fiber class.
The first insertion~1 in the last two relative invariants in~\eref{P4eg_e}
indicates that no constraint is imposed on the domain of the maps by pulling 
back a class~$\ka$ from a Deligne-Mumford space of curves.
The relative insertions $(1^{\de-1},\pt)$ and $1$ in these invariants 
(shown after the second semi-column in each case) arise from the Kunneth decomposition
of the diagonal~$\De_V$ in~$V^2$;
the point insertion on the first of these invariants corresponds 
to the pairing with the second invariant, which arises from 
a zero-dimensional relative moduli space.
It~is immediate from the $g\!=\!0$ part of the argument in Section~\ref{DirComp_subs}
that 
\BE{P4eg_e2}\frac{\de}{\de!}\,\GW_{0,1;\1_{\de}}^{\P^4,V_{\de}}(1;\pt;1^{\de-1},\pt)=
\GW_{0,1}^{\P^4}(\pt,\pt)=1.\EE
Combining~\eref{P4eg_e} with~\eref{P4eg_e2} and Lemma~\ref{M31_lmm} below,
we conclude~that 
$$\GW_{3,1}^{\P^4}(\pt) = \frac{1}{\de!}\GW_{3,1;\1_{\de}}^{\P^4,V_{\de}}(\pt)
+\frac{\lr{c_1(V_{\de})c_2(V_{\de})\!-\!c_3(V_{\de}),V_{\de}}}{362880} \,.$$ 
The second equality in~\eref{P4_e} now follows from 
$$c(V_{\de})=\big((1+x)^5(1\!+\!\de x)^{-1}\big)\big|_{V_{\de}}\in H^*(V_{\de};\Z),$$
where $x\!=\!c_1(\cO_{\P^4}(1))\in H^2(\P^4;\Z)$ is the standard generator.

\begin{lmm}\label{M31_lmm}
Let $(V,\om)$ be a compact symplectic manifold of real dimension~$6$ 
and $L\!\lra\!V$ be a complex line bundle.
With notation as at the beginning of Section~\ref{VMext_subs},  
the virtual dimension of the genus~3 relative moduli space 
$\ov\fM_{3,0;(1)}^{\P_{L,\i}}(\P_L,F)$ is~0 and 
$$\deg\big[\ov\fM_{3,0;(1)}^{\P_{L,\i}}(\P_L,F)\big]^{\vir}
=\frac{\lr{c_1(V)c_2(V)\!-\!c_3(V),V}}{362880}\,.$$
\end{lmm}

\begin{proof}
The first claim is immediate from the second equation in~\eref{virdim_e}.
In order to establish the second claim, we proceed as in Section~\ref{SympSum_subs}
by first choosing a generic deformation $\nu\!\in\!\Ga_{3,0}(V,J_V)$.
Lifting~$J_V$ and $\nu$ to $\P_L\!\lra\!V$ as in  Section~\ref{SympSum_subs}, 
we obtain a fibration
\BE{fMfibr_e2}
\pi_{L,V}\!:\ov\fM_{3,0;(1)}^{\P_{L,\i}}\big(\P_L,F;J,\pi_{X,V}^*\nu\big)
\lra \ov\fM_{3,0}\big(V,0;J_V,\nu\big)\EE
as in~\eref{fMfibr_e}.
In this case, the base is zero-dimensional.
Since the obstruction bundle for $\ov\fM_{3,0}(V,0)$ is given by~\eref{ObsBun_e}
with $(g_v,k_v\!+\!\ell_v)\!=\!(3,0)$, the degree of this base~is
\begin{equation*}\begin{split}
&\blr{e(\pi_1^*\E^*\!\otimes\!\pi_2^*TV),\ov\cM_3\!\times\!V}
=\lr{\la_1\la_2\la_3,\ov\cM_3}\blr{c_1(V)c_2(V)\!-\!3c_3(V),V}\\
&\hspace{1in}+\lr{\la_2^3,\ov\cM_3}\blr{c_3(V),V}
+\lr{\la_3^2,\ov\cM_3}\blr{c_1(V)^3\!-\!3c_1(V)c_2(V)\!+\!3c_3(V),V}\,.
\end{split}\end{equation*}
The three intersection numbers on $\ov\cM_3$ above are provided by 
Table~\ref{M2and3_tbl} and \eref{lag2_e}.
The second claim of Lemma~\ref{M31_lmm} now follows from Lemma~\ref{M30_lmm} below.
\end{proof}

\begin{lmm}\label{M30_lmm}
If $\st\!:\ov\fM_{3,0;(1)}^{\pt}(\P^1,1)\!\lra\!\ov\cM_3$ is the forgetful morphism
dropping the map to~$\P^1$ and the marked point, then
\BE{M30_e}\st_*\big[\ov\fM_{3,0;(1)}^{\pt}(\P^1,1)\big]^{\vir}
=4\big[\ov\cM_3\big].\EE\\
\end{lmm}

\noindent
Since $\ov\cM_3$ is smooth (as an orbifold) and irreducible,
\eref{M30_e} is equivalent~to 
\BE{M30_e2} \blr{\st^*\si,[\ov\fM_{3,0;(1)}^{\pt}(\P^1,1)]^{\vir}}=4,\EE
where $\si\!\in\!H^{12}(\ov\cM_3)$ is the Poincare dual of 
a generic element $\Si$ of~$\ov\cM_3$.
We give two proofs of~\eref{M30_e2} below,
which are similar to the two proofs of~\eref{M21_e2}.

\begin{proof}[{\bf{\emph{Proof 1 of \eref{M30_e2}}}}]
We continue with the localization setup in the first proof of~\eref{M21_e2}.
The fixed loci of the induced action on $\ov\fM_{3,0;(1)}^{\,p_2}(\P^1,1)$
again consist of maps sending components of positive genus to either the fixed point~$p_1$
or the rubber~$\P^1$ attached to the fixed point~$p_2$.
The four graphs describing these fixed loci, in the notation of \cite[Chapter~27]{MirSym}
and Figure~\ref{P1loc_fig1}, are shown in Figure~\ref{P4loc_fig1}.\\

\begin{figure}
\begin{pspicture}(1.2,1.2)(11,3)
\psset{unit=.3cm}
\pscircle*(16,9){.2}\rput(17.8,9){\sm{$(1,3)$}}
\pscircle*(16,5){.2}\rput(16.8,5){\sm{$2$}}
\psline[linewidth=.05](16,9)(16,5)
\psline[linewidth=.05](16,5)(15,4)\rput(14.5,4){$1$}
\pscircle*(26,9){.2}\rput(27.8,9){\sm{$(1,2)$}}
\pscircle*(26,5){.2}\rput(27.8,5){\sm{$(2,1)$}}
\psline[linewidth=.05](26,9)(26,5)
\psline[linewidth=.05](26,5)(25,4)\rput(24.5,4){$1$}
\pscircle*(36,9){.2}\rput(37.8,9){\sm{$(1,1)$}}
\pscircle*(36,5){.2}\rput(37.8,5){\sm{$(2,2)$}}
\psline[linewidth=.05](36,9)(36,5)
\psline[linewidth=.05](36,5)(35,4)\rput(34.5,4){$1$}
\pscircle*(46,9){.2}\rput(46.8,9){\sm{$1$}}
\pscircle*(46,5){.2}\rput(47.8,5){\sm{$(2,3)$}}
\psline[linewidth=.05](46,9)(46,5)
\psline[linewidth=.05](46,5)(45,4)\rput(44.5,4){$1$}
\end{pspicture}
\caption{The four graphs describing the $(\C^*)^2$-fixed loci 
of $\ov\fM_{3,0;(1)}^{\,p_2}(\P^1,1)$}
\label{P4loc_fig1}
\end{figure}
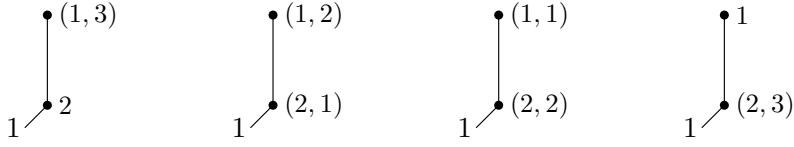

\noindent
The morphism~$\st$ takes the fixed loci represented by the two middle diagrams 
in Figure~\ref{P4loc_fig1} to the closure in $\ov\cM_3$ of the locus consisting 
of two-component maps.
Thus, $\st^*\si$ vanishes on these loci and the two middle diagrams do not contribute
to~\eref{M30_e2} via the virtual localization theorem of~\cite{GP}.\\

\noindent 
The locus represented by the first diagram in Figure~\ref{P4loc_fig1}
is isomorphic to $\ov\cM_{3,1}$ and is cut down by $\st^*\si$ to 
the curve~$\Si$ (which encodes the position of the node).
The space of deformations of this locus consists of moving the node and
of smoothing the node;
its euler class equals
$$(\al_1\!-\!\al_2)(\al_1\!-\!\al_2\!+\!c_1(T\Si))$$
after restricting to the cut-down space.
The obstruction bundle after cutting down by $\st^*\si$ is 
$$H^1(\Si;T_{p_1}\P^1)=\big(H^0(\Si;T^*\Si\!\otimes\!T_{p_1}^*\P^1)\big)^*
\approx T_{p_1}\P^1\oplus T_{p_1}\P^1\oplus T_{p_1}\P^1\,;$$
its equivariant euler class is $(\al_1\!-\!\al_2)^3$.
Thus, the contribution of the first diagram in Figure~\ref{P4loc_fig1} to~\eref{M30_e2}~is
$$\int_{\ov\cM_{3,1}}\st^*\si\,
\frac{(\al_1\!-\!\al_2)^3}{(\al_1\!-\!\al_2)(\al_1\!-\!\al_2\!+\!c_1(T\Si))}
=-\int_{\Si}c_1(T\Si)=4\,;$$ 
see \cite[(7)]{GP} or \cite[Theorem~3.6]{GV}.\\

\noindent 
The locus represented by the last diagram in Figure~\ref{P4loc_fig1}
is isomorphic~to  $\ov\fM_{3,0;(1),(1)}^{0,\i}(\P^1,1)_{\sim}$.
Since the virtual dimension of this moduli space is~5,
the restriction of~$\st^*\si$ to this fixed locus vanishes.
By \cite[Theorem~3.6]{GV}, the last diagram in Figure~\ref{P4loc_fig1}
thus does not contribute to~\eref{M30_e2}.
Combining this with the conclusion of the two previous paragraphs,
we obtain~\eref{M30_e2}.
\end{proof}

\begin{proof}[{\bf{\emph{Proof 2 of \eref{M30_e2}}}}]
Let $(\Si,\fj)$ be a generic element of $\ov\cM_3$ as before.
The first paragraph of the second proof of~\eref{M21_e2} applies
to the present situation; the only change is that the base in~\eref{TarBndl_e}
is replaced~by 
$$\Si\!\times\!\Hol_{\fj}(\P^1,1)\approx \Si\!\times\!\P^1\,.$$
The line bundle of smoothing parameters~\eref{DomBndl_e} now becomes
$$ T\Si\otimes T\P^1\lra \Si\!\times\!\P^1\,.$$
Analogously to the sentence containing~\eref{AffMap_e},
the solutions of the analogue of~\eref{Jnueq_e5} in this situation 
correspond to the zeros of the section~of 
$$\Obs/\Im\,L\lra \Si\!\times\!\Hol_{\fj}(\P^1,1),$$
with $L$ as before, induced by a generic admissible $\nu$, excluding  
the ones with $u(x_2)\!=\!\pt$.
Without the first restriction on~$\nu$ in~\eref{nuR_e}, the number of such zeros 
would have been 
$$\blr{e(\Obs/\Im\,L),\Si\!\times\!\Hol_{\fj}(\P^1,1)}
=\blr{c_1(\C^3/T\Si),\Si}\blr{c_1(T\P^1),\P^1}=8.$$
The contribution to this number from the vanishing of~$\bar\nu$ along $\Si\!\times\!\pt$
is the number of zeros of an affine bundle~map
$$\C\oplus T\Si\!\otimes\!T_{\pt}\P^1\lra \cH_{\Si}^{0,1}\!\otimes\!T_{\pt}\P^1$$
with an injective linear part.
Thus, the latter number~is
$$\blr{e(\C^3/(\C\!\oplus\!T\Si)),\Si}=4\,.$$
The number in~\eref{M30_e2} is the difference of the two numbers above. 
\end{proof}

\subsection{The $\de\!=\!0,1$ numbers in Example~\ref{P4_eg}}
\label{LocComp_subs}

\noindent
We now use the virtual localization theorem of~\cite{GP} to compute 
the absolute invariant and 
the virtual localization theorem of~\cite{GV} to compute 
the $\de\!=\!1$ case of the relative invariant
in Example~\ref{P4_eg}.\\

\noindent
We apply \cite[(7)]{GP} with the $\C^*$-action on~$\P^4$ given~by
$$c\cdot[Z_1,Z_2,Z_3,Z_4,Z_5]=[Z_1,cZ_2,cZ_3,c^{-1}Z_4,c^{-1}Z_5]$$
and its linear lift to $\cO_{\P^4}(1)$ defined in the same way.
The fixed locus of this action consists~of
$$p_1\equiv[1,0,0,0,0], \qquad 
\P_{23}^1\equiv\big\{[0,Z_2,Z_3,0,0]\in\P^4\big\},\qquad 
\P_{45}^1\equiv\big\{[0,0,0,Z_4,Z_5]\in\P^4\big\}.$$
Let
$$\al=c_1\big(\cO_{\P^4}(1)\big)\big|_{p_2}\in H^*_{\C^*}.$$
We denote~by 
$$\ov\fM_{3,1}(\P^4,1)_{p_1}\subset \ov\fM_{3,1}(\P^4,1) $$
the preimage of $p_1$ under the evaluation morphism~$\ev_1$. 
We will compute
\BE{M3loc_e}\begin{split} 
\int_{[\ov\fM_{3,1}(\P^4,1)_{p_1}]^{\vir}}\!\!1.
\end{split}\EE
The $\C^*$-fixed loci of this moduli space consist of maps sending 
the positive-genus components to~$p_1$ or a point on~$\P_{23}^1$ or~$\P_{45}^1$
with the image of a degree~1 rational component running between~$p_1$ and
a point on either~$\P_{23}^1$ or~$\P_{45}^1$.
The four types of  graphs possibly contributing  the integral in~\eref{M3loc_e} 
are shown in the left half of Figure~\ref{P4loc_fig2},
where $\pm$ on the bottom vertex indicates whether it lies on~$\P_{23}^1$ or~$\P_{45}^1$,
respectively.
In the computations below, we first assume that $i\!=\!+$.\\

\begin{figure}
\begin{pspicture}(2,1.2)(11,3)
\psset{unit=.3cm}
\pscircle*(14,9){.2}\rput(15.8,9){\sm{$(1,3)$}}
\pscircle*(14,5){.2}\rput(14.8,5){\sm{$i$}}
\psline[linewidth=.05](14,9)(14,5)
\psline[linewidth=.05](14,9)(13,10)\rput(13,10.7){\sm{$1$}}
\pscircle*(19,9){.2}\rput(20.8,9){\sm{$(1,2)$}}
\pscircle*(19,5){.2}\rput(20.8,5){\sm{$(i,1)$}}
\psline[linewidth=.05](19,9)(19,5)
\psline[linewidth=.05](19,9)(18,10)\rput(18,10.7){\sm{$1$}}
\pscircle*(24,9){.2}\rput(25.8,9){\sm{$(1,1)$}}
\pscircle*(24,5){.2}\rput(25.8,5){\sm{$(i,2)$}}
\psline[linewidth=.05](24,9)(24,5)
\psline[linewidth=.05](24,9)(23,10)\rput(23,10.7){\sm{$1$}}
\pscircle*(29,9){.2}\rput(29.8,9){\sm{$1$}}
\pscircle*(29,5){.2}\rput(30.8,5){\sm{$(i,3)$}}
\psline[linewidth=.05](29,9)(29,5)
\psline[linewidth=.05](29,9)(28,10)\rput(28,10.7){\sm{$1$}}
\pscircle*(39,9){.2}\rput(40.8,9){\sm{$(1,3)$}}
\pscircle*(39,5){.2}\rput(39.8,5){\sm{$i$}}
\psline[linewidth=.05](39,9)(39,5)
\psline[linewidth=.05](39,9)(38,10)\rput(38,10.7){\sm{$1$}}
\psline[linewidth=.05](39,5)(38,4)\rput(37.5,4){\sm{$2$}}
\pscircle*(44,9){.2}\rput(45.8,9){\sm{$(1,2)$}}
\pscircle*(44,5){.2}\rput(45.8,5){\sm{$(i,1)$}}
\psline[linewidth=.05](44,9)(44,5)
\psline[linewidth=.05](44,9)(43,10)\rput(43,10.7){\sm{$1$}}
\psline[linewidth=.05](44,5)(43,4)\rput(42.5,4){\sm{$2$}}
\pscircle*(49,9){.2}\rput(50.8,9){\sm{$(1,1)$}}
\pscircle*(49,5){.2}\rput(50.8,5){\sm{$(i,2)$}}
\psline[linewidth=.05](49,9)(49,5)
\psline[linewidth=.05](49,9)(48,10)\rput(48,10.7){\sm{$1$}}
\psline[linewidth=.05](49,5)(48,4)\rput(47.5,4){\sm{$2$}}
\pscircle*(54,9){.2}\rput(54.8,9){\sm{$1$}}
\pscircle*(54,5){.2}\rput(55.8,5){\sm{$(i,3)$}}
\psline[linewidth=.05](54,9)(54,5)
\psline[linewidth=.05](54,9)(53,10)\rput(53,10.7){\sm{$1$}}
\psline[linewidth=.05](54,5)(53,4)\rput(52.5,4){\sm{$2$}}
\end{pspicture}
\caption{The two sets of graphs possibly contributing to 
the integrals~\eref{M3loc_e} and~\eref{M3loc_e2}, with $i\!\in\!\{+,-\}$ 
in the first case and 
$i\!\in\!\{2,3,4,5\}$ in the second case.} 
\label{P4loc_fig2}
\end{figure}
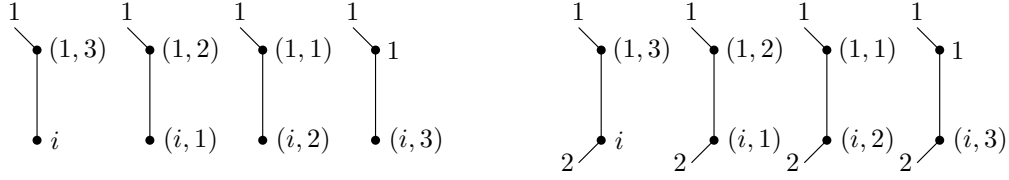

\noindent
The locus represented by the first diagram in Figure~\ref{P4loc_fig2}
is isomorphic to $\ov\cM_{3,2}\!\times\!\P^1$. 
The space of deformations of this locus consists of smoothing the node
and turning the line around~it away from~$\P_{23}^1$;
the equivariant euler class of the space of deformations is~thus
$$(-\al\!-\!x\!-\!\psi_2) (2\al\!+\!x)^2=-(\al\!+\!x\!+\!\psi_2) (2\al\!+\!x)^2\,,$$
where $x\!=\!c_1(\cO_{\P^1}(1))\!\in\!H^1(\P^1;\Z)$ is the standard generator.
The euler class of the obstruction bundle is given~by
$$e\big(\E_3^*\!\otimes\!T_{p_1}\P^4\big)
=\big(\la_3-\al\la_2+\al^2\la_1-\al^3\big)^2
\big(\la_3+\al\la_2+\al^2\la_1+\al^3\big)^2\,.$$
By \cite[(5.3)]{Mumford},
$$\big(\la_3-\al\la_2+\al^2\la_1-\al^3\big)
\big(\la_3+\al\la_2+\al^2\la_1+\al^3\big)=-\al^6\,.$$
The contribution of the first diagram in Figure~\ref{P4loc_fig2} 
to~\eref{M3loc_e} is~thus
\BE{M3loc_e1a}\begin{split}
&-\int_{\ov\cM_{3,2}\times\P^1}\frac{\al^{12}}{(\al\!+\!x\!+\!\psi_2) (2\al\!+\!x)^2}
=\frac52 \lr{x,\P^1}\blr{\psi_2^8,\ov\cM_{3,2}}
=\frac52 \blr{\psi_1^7,\ov\cM_{3,1}}= \frac{5}{165888}\,.
\end{split}\EE
The second equality above applies the dilaton equation \cite[Exercise~25.2.7]{MirSym};
the last follows from the first column in Table~\ref{M31_tbl}.\\

\noindent
The locus represented by the second diagram in Figure~\ref{P4loc_fig2}
is isomorphic to $\ov\cM_{2,2}\!\times\!\ov\cM_{1,1}\!\times\!\P^1$. 
The space of deformations of this locus consists of smoothing the two nodes
and moving the bottom node away from~$\P_{23}^1$;
the equivariant euler class of the space of deformations is~thus
$$(-\al\!-\!x\!-\!\pst)(\al\!+x\!-\!\psb)(\al\!+\!x)(2\al\!+\!x)^2
=-(\al\!+\!x\!+\!\pst)(\al\!+x\!-\!\psb)(\al\!+\!x)(2\al\!+\!x)^2\,,$$
where $\pst\!\in\!H^*(\ov\cM_{2,2})$ and $\psb\!\in\!H^*(\ov\cM_{1,1})$.
The euler class of the obstruction bundle is given~by
\begin{equation*}\begin{split}
&e\big(\E_2^*\!\otimes\!T_{p_1}\P^4\big)e\big(\E_1^*\!\otimes\!T\P^4|_{\P^1_{23}}\big)\\
&\qquad\qquad=\big(\la_2\!-\!\al\la_1\!+\!\al^2\big)^2\big(\la_2\!+\!\al\la_1\!+\!\al^2\big)^2
\big(\lb\!-\!2x\big)\big(\lb\!-\!(\al\!+\!x)\big)\big(\lb\!-\!(2\al\!+\!x)\big)^2\\
&\qquad\qquad=4\al\big(\la_2\!-\!\al\la_1\!+\!\al^2\big)^2\big(\la_2\!+\!\al\la_1\!+\!\al^2\big)^2
(3\lb x-2\al x\!+\!\al\lb)\big(\lb\!-\!(\al\!+\!x)\big),
\end{split}\end{equation*}
where $\lb\!\in\!H^*(\ov\cM_{1,1})$
By \cite[(5.3)]{Mumford},
$$\big(\la_2\!-\!\al\la_1\!+\!\al^2\big)\big(\la_2\!+\!\al\la_1\!+\!\al^2\big)=\al^4\,.$$
Since $\psi_1\!=\!\la$ on $\ov\cM_{1,1}$, 
the contribution of the second diagram in Figure~\ref{P4loc_fig2} 
to~\eref{M3loc_e} is~thus
\BE{M3loc_e1b}\begin{split}
\int_{\ov\cM_{2,2}\times\ov\cM_{1,1}\times\P^1}
\frac{4\al^9 (3\lb x-2\al x\!+\!\al\lb) }
{(\al\!+\!x\!+\!\pst)(\al\!+\!x)(2\al\!+\!x)^2}
&=5\lr{\la,\ov\cM_{1,1}}\lr{x,\P^1}\blr{\psi_2^5,\ov\cM_{2,2}}\\
&=\frac{5}{24}\blr{\psi_1^4,\ov\cM_{2,1}}=\frac{5}{27648}\,.
\end{split}\EE
The second equality above applies the dilaton equation \cite[Exercise~25.2.7]{MirSym};
the last follows from the first column in Table~\ref{M21_tbl}.\\

\begin{table}
\begin{center}
\begin{small}
\begin{tabular}{||c|c|c|c|c|c|c||}
\hline\hline
\!$\psi_1^7$\!& \!$\psi_1^6\la_1$\!& \!$\psi_5^2\la_1^2$\!& \!$\psi_1^5\la_2$\!& 
\!$\psi_1^4\la_1^3$\!&  \!$\psi_1^4\la_1\la_2$\!& \!$\psi_1^4\la_3$\!\\
\hline
\!$\frac{1}{82944}$\!& \!$\frac{7}{138240}$\!& \!$\frac{41}{290304}$\!& \!$\frac{41}{580608}$\!& 
\!$\frac{23}{96768}$\!& \!$\frac{23}{193536}$\!& \!$\frac{31}{967680}$\!\\
\hline\hline
\end{tabular}
\begin{tabular}{c}
\hspace{3in}${}$\\
\end{tabular}
\begin{tabular}{||c|c|c|c|c|c|c|c|c||}
\hline\hline
\!$\psi_1^3\la_1^4$\!& \!$\psi_1^3\la_1^2\la_2$\!& \!$\psi_1^3\la_1\la_3$\!& 
\!$\psi_1^3\la_2^2$\!&
\!$\psi_1^2\la_1^5$\!& \!$\psi_1^2\la_1^3\la_2$\!& \!$\psi_1^2\la_1^2\la_3$\!&
\!$\psi_1^2\la_1\la_2^2$\!& \!$\psi_1^2\la_2\la_3$\!\\
\hline
\!$\frac{41}{181440}$\!& \!$\frac{41}{362880}$\!&  \!$\frac{41}{1451520}$\!&
\!$\frac{41}{725760}$\!&
\!$\frac{1}{7560}$\!& $\frac{1}{15120}$& $\frac{1}{60480}$&
$\frac{1}{30240}$& $\frac{1}{120960}$\\
\hline\hline
\end{tabular}
\end{small}
\end{center}
\caption{The top intersections of $\la$-classes and $\psi_1^i$ with $i\!\ge\!2$
on $\ov\cM_{3,1}$; the intersections with $\psi_1^1$ are obtained by multiplying
the corresponding numbers in Table~\ref{M2and3_tbl} by~4.}
\label{M31_tbl}
\end{table}

\noindent
The locus represented by the third diagram in Figure~\ref{P4loc_fig2}
is isomorphic to $\ov\cM_{1,2}\!\times\!\ov\cM_{2,1}\!\times\!\P^1$. 
The euler class of its deformation space is as in the previous paragraph.
The euler class of the obstruction bundle is given~by
\begin{equation*}\begin{split}
&e\big(\E_1^*\!\otimes\!T_{p_1}\P^4\big)e\big(\E_2^*\!\otimes\!T\P^4|_{\P^1_{23}}\big)\\
&\qquad=\big(\lt\!-\!\al\big)^2\big(\lt\!+\!\al\big)^2
\big(\la_2\!-\!2x\la_1\big)
\big(\la_2\!-\!(\al\!+\!x)\la_1\!+\!(\al\!+\!x)^2\big)
\big(\la_2\!-\!(2\al\!+\!x)\la_1\!+\!(2\al\!+\!x)^2\big)^2,
\end{split}\end{equation*}
where $\lt\!\in\!H^*(\ov\cM_{1,2})$.
Since $\la^2\!=\!0$ on $\ov\cM_{1,2}$ and $\la_2^2,2\la_2\!-\!\la_1^2\!=\!0$ on $\ov\cM_2$,
the contribution of the third diagram in Figure~\ref{P4loc_fig2} 
to~\eref{M3loc_e} is~thus
\BE{M3loc_e1c}\begin{split}
&-\int_{\ov\cM_{1,2}\times\ov\cM_{2,1}\times\P^1}
\frac{4\al^8\la_1(2\al^2\la_1\!-\!4\al\la_1^2- 
(8\al^2\!-\!24\al\la_1\!+\!29\la_1^2)x)}
{(\al\!+\!x\!+\!\psi_1)(\al\!+x\!-\!\psi_2)(\al\!+\!x)(2\al\!+\!x)^2}\\
&\qquad\qquad=\blr{\psi_1^2,\ov\cM_{1,2}}\lr{x,\P^1}
\blr{\la_1^3\psi_1\!-\!8\la_1^2\psi_1^2\!+\!8\la_1\psi_1^3,\ov\cM_{2,1}}
=-\frac{1}{11520}\,.
\end{split}\EE
The second equality above applies the dilaton equation \cite[Exercise~25.2.7]{MirSym}
and uses the first column in Table~\ref{M2and3_tbl} and the third in Table~\ref{M21_tbl}.\\

\noindent
The locus represented by the fourth diagram in Figure~\ref{P4loc_fig2}
is isomorphic to $\ov\cM_{3,1}\!\times\!\P^1$. 
The space of deformations of this locus consists of smoothing the (bottom) node
and moving it from~$\P_{23}^1$;
the equivariant euler class of the space of deformations is~thus
$$(\al\!+x\!-\!\psi_1)(\al\!+\!x)(2\al\!+\!x)^2\,.$$
The euler class of the obstruction bundle is given~by
\begin{equation*}\begin{split}
e\big(\E_3^*\!\otimes\!T\P^4|_{\P^1_{23}}\big)
=\big(\la_3\!-\!2x\la_2\big)
&\big(\la_3\!-\!(\al\!+\!x)\la_2\!+\!(\al\!+\!x)^2\la_1\!-\!(\al\!+\!x)^3\big)\\
&\times\big(\la_3\!-\!(2\al\!+\!x)\la_2\!+\!(2\al\!+\!x)^2\la_1\!-\!(2\al\!+\!x)^3\big)^2.
\end{split}\end{equation*}
Since $\la_1^2\!=\!2\la_2$, $\la_2^2\!=\!2\la_1\la_3$, and $\la_3^2\!=\!0$ on $\ov\cM_{3,1}$,
the contribution of the fourth diagram in Figure~\ref{P4loc_fig2} 
to~\eref{M3loc_e} is~thus
\BE{M3loc_e1d}\begin{split}
&\int_{\ov\cM_{3,1}\times\P^1}
\frac{\al^6 (16\al^2\la_1^4\!-\!16\al\la_1^5\!+\!9\la_1^6\!-\!64\al^3\la_3)}
{(\al\!+x\!-\!\psi_1)(\al\!+\!x)(2\al\!+\!x)^2}\\
&\qquad+\int_{\ov\cM_{3,1}}
\frac{\al^2(128\al^4\la_1^2\!-\!256\al^3\la_1^3\!+\! 
   424\al^2\la_1^4\!-\!308\al\la_1^5\!+\!141\la_1^6\!-\!768\al^3\la_3)}
{8(\al\!-\!\psi_1)}\\
&=\frac18\lr{x,\P^1}
\blr{69\la_1^6\psi_1\!-\!148\la_1^5\psi_1^2\!+\!232\la_1^4\psi_1^3
\!-\!256\la_1^3\psi_1^4
\!+\!128\la_3\psi_1^4\!+\!128\la_1^2\psi_1^5,\ov\cM_{3,1}}
=-\frac{1}{2880}\,.
\end{split}\EE
Combining the numbers in~\eref{M3loc_e1a}-\eref{M3loc_e1d} and multiplying
the result by~2 (to account for $i\!=\!\pm$), we obtain the first equality
in~\eref{P4_e}.
This conclusion agrees with A.~Gathmann's {\it growi} program.\\

\noindent
We next apply \cite[Theorem~3.6]{GV} with the action of $\T\!\equiv\!(\C^*)^2$
on~$\P^4$ given~by
$$(c_1,c_2)\cdot[Z_1,Z_2,Z_3,Z_4,Z_5]=[Z_1,c_1Z_2,c_1^{-1}Z_3,c_2Z_4,c_2^{-1}Z_5]$$
and its linear lift to $\cO_{\P^4}(1)$ defined in the same way.
The fixed locus of this action consists~of the five~points
$$p_1\equiv[1,0,0,0,0], \qquad\ldots,\qquad p_5\equiv[0,0,0,0,1].$$
Let
$$\al_1=c_1\big(\cO_{\P^4}(1)\big)\big|_{p_2}\in H^*_{\T}\,,~~~
\al_2=c_1\big(\cO_{\P^4}(1)\big)\big|_{p_4}\in H^*_{\T}\,,\quad
V=\big\{[0,Z_2,Z_3,Z_4,Z_5]\!\in\!\P^4\big\}.$$
We denote~by 
$$\ov\fM_{3,1;(1)}^V(\P^4,1)_{p_1}\subset \ov\fM_{3,1;(1)}^V(\P^4,1)$$
the preimage of $p_1$ under the evaluation morphism~$\ev_1$. 
We will compute
\BE{M3loc_e2}\begin{split} 
\int_{[\ov\fM_{3,1;(1)}^V(\P^4,1)_{p_1}]^{\vir}}\!\!1\,.
\end{split}\EE
The $\C^*$-fixed loci of this moduli space consist of maps sending 
the positive-genus components to~$p_1$ and at most one of the fixed points~$p_i$ 
with $i\!=\!2,3,4,5$;
the image of the non-contracted degree~1 rational tail runs between~$p_1$ 
and one of the fixed points~$p_i$ with $i\!=\!2,3,4,5$.
The four types of  graphs possibly contributing to the integrals
in~\eref{M3loc_e2} are shown in the right half of Figure~\ref{P4loc_fig2}.
In the computations below, we first assume that $i\!=\!2$.\\

\noindent
The locus represented by the first diagram in the right half of Figure~\ref{P4loc_fig2}
is isomorphic to~$\ov\cM_{3,2}$. 
Its deformations consist of smoothing the node
and turning the line around~it away from~$p_2$;
the equivariant euler class of the space of deformations is~thus
$$(-\al_1\!-\!\psi_2)\big(\al_1\!-\!(-\al_1)\big)(\al_1\!-\!\al_2) 
\big(\al_1\!-\!(-\al_2)\big) 
=-2\al_1(\al_1^2\!-\!\al_2^2)(\al_1\!+\!\psi_2)\,.$$
The obstruction bundle is as for the first diagram in Figure~\ref{P4loc_fig2},
but its euler class  is now given~by
$$\prod_{j=1,2}\!\!\!\big((\al_j^3\!-\!\al_j^2\la_1\!+\!\al_j\la_2\!-\!\la_3)
(\al_j^3\!+\!\al_j^2\la_1\!+\!\al_j\la_2\!+\!\la_3)\big)
=\al_1^6\al_2^6\,;$$
the equality holds by \cite[(5.3)]{Mumford}.
The contribution of the fifth diagram in Figure~\ref{P4loc_fig2} 
to~\eref{M3loc_e2} is~thus
\BE{M3loc_e2a}\begin{split}
-\int_{\ov\cM_{3,2}}\frac{\al_1^6\al_2^6}{2\al_1(\al_1^2\!-\!\al_2^2)(\al_1\!+\!\psi_2)}
&=-\frac12\cdot\frac{\al_2^6}{\al_1^4(\al_1^2\!-\!\al_2^2)} \blr{\psi_2^8,\ov\cM_{3,2}}\\
&=-\frac12\cdot\frac{\al_2^6}{\al_1^4(\al_1^2\!-\!\al_2^2)} 
\blr{\psi_1^7,\ov\cM_{3,1}}=
-\frac12\cdot \frac{1}{82944}\cdot\frac{\al_2^6}{\al_1^4(\al_1^2\!-\!\al_2^2)} \,.
\end{split}\EE
The second equality above applies the dilaton equation \cite[Exercise~25.2.7]{MirSym};
the last follows from the first column in Table~\ref{M31_tbl}.\\

\noindent
The locus represented by the second diagram  in the right half of Figure~\ref{P4loc_fig2}
is isomorphic~to 
$$F_2\equiv\ov\cM_{2,2} \times \ov\fM_{1,0;(1),(1)}^{0,\i}(\P^1,1)_{\sim}\,.$$ 
The equivariant euler class of the space of deformations becomes
$$-2\al_1(\al_1^2\!-\!\al_2^2)(\al_1\!+\!\pst)(\al_1\!-\!\psb),$$
where $\pst\!\in\!H^*(\ov\cM_{2,2})$ and $\psb\!\!=\!\psi_{\i}$ is 
on the rubber moduli space; see \cite[Section~3.3]{GV}.
The euler class of the obstruction bundle is now given~by
\begin{equation*}\begin{split}
&e\big(\E_2^*\!\otimes\!T_{p_1}\P^4\big)e\big(\E_1^*\!\otimes\!T_{p_2}V\big)\\
&\qquad=\prod_{j=1,2}\!\!\!\big((\al_j^2\!-\!\al_j\la_1\!+\!\la_2)
(\al_j^2\!+\!\al_j\la_1\!+\!\la_2)\big)
\cdot(2\al_1\!-\!\lb)(\al_1\!-\!\al_2\!-\!\lb)(\al_1\!+\!\al_2\!-\!\lb)\\
&\qquad\cong-\al_1^4\al_2^4(5\al_1^2\!-\!\al_2^2)\lb
\mod H_{\T}^*\subset H_{\T^*}(F_2)\,,
\end{split}\end{equation*}
where $\lb\!\in\!H^*(F_2)$ is the pull-back of $\la\!\in\!H^*(\ov\cM_{1,1})$
by either forgetful morphism~$f$ from the second factor.
Since 
$$\ov\fM_{1,0;(1),(1)}^{0,\i}(\P^1,1)_{\sim}
\approx\ov\cM_{1,1}\times \ov\fM_{0,1;(1),(1)}^{0,\i}(\P^1,1)_{\sim}$$
as spaces and the last factor above is a point, $\psb$ vanishes on 
the virtual class of the second factor in~$F_2$.
The second proofs of~\eref{M21_e2} and~\eref{M30_e2} readily show~that 
$$f_*\big[\ov\fM_{1,0;(1),(1)}^{0,\i}(\P^1,1)_{\sim}\big]^{\vir}
=\big[\ov\cM_{1,1}\big]\,.$$
Thus, the contribution of the sixth diagram in Figure~\ref{P4loc_fig2} 
to~\eref{M3loc_e2} is
\BE{M3loc_e2b}\begin{split}
&\int_{[F_2]^{\vir}}\frac{\al_1^4\al_2^4(5\al_1^2\!-\!\al_2^2)\lb}
{2\al_1(\al_1^2\!-\!\al_2^2)(\al_1\!+\!\pst)\al_1}
=-\frac12\cdot
\frac{\al_2^4(5\al_1^2\!-\!\al_2^2)}{\al_1^4(\al_1^2\!-\!\al_2^2)}
\cdot\blr{\psi_2^5,\ov\cM_{2,2}}
\lr{\la,\ov\cM_{1,1}}\\
&\qquad\qquad=-\frac12\cdot
\frac{\al_2^4(5\al_1^2\!-\!\al_2^2)}{\al_1^4(\al_1^2\!-\!\al_2^2)}
\cdot\frac{1}{24}\blr{\psi_2^4,\ov\cM_{2,1}}
=-\frac12\cdot\frac{1}{27648}\cdot
\frac{\al_2^4(5\al_1^2\!-\!\al_2^2)}{\al_1^4(\al_1^2\!-\!\al_2^2)}\,.
\end{split}\EE
The second equality above applies the dilaton equation \cite[Exercise~25.2.7]{MirSym};
the last follows from the first column in Table~\ref{M21_tbl}.\\

\noindent
The locus represented by the third diagram  in the right half of Figure~\ref{P4loc_fig2}
is isomorphic~to 
$$F_3\equiv\ov\cM_{1,2} \times \ov\fM_{2,0;(1),(1)}^{0,\i}(\P^1,1)_{\sim}
\equiv F_{3;1}\times F_{3;2}\,.$$ 
The equivariant euler class of its deformation space is as in the previous paragraph.
The euler class of the obstruction bundle is now given~by
\begin{equation*}\begin{split}
&e\big(\E_1^*\!\otimes\!T_{p_1}\P^4\big)e\big(\E_2^*\!\otimes\!T_{p_2}V\big)
=\prod_{j=1,2}\!\!\!\big((\al_j\!-\!\lt)(\al_j\!+\!\lt)\big)
\cdot(4\al_1^2\!-\!2\al_1\la_1\!+\!\la_2)\\
&\hspace{2in}\times
\big((\al_1\!-\!\al_2)^2\!-\!(\al_1\!-\!\al_2)\la_1\!+\!\la_2\big)
\big((\al_1\!+\!\al_2)^2\!-\!(\al_1\!+\!\al_2)\la_1\!+\!\la_2\big)\\
&\qquad\cong-\al_1^3\al_2^2(9\al_1^2\!-\!\al_2^2)\la_1^3
+\frac12 \al_1^2\al_2^2(25\al_1^4\!-\!10\al_1^2\al_2^2\!+\!\al_2^4)\la_1^2
\mod H_{\T}^*\!\otimes\!H^2(F_{3;2})\subset H_{\T^*}(F_3)\,,
\end{split}\end{equation*}
where $\lt\!\in\!H^*(\ov\cM_{1,2})$ and $\la_1,\la_2\!\in\!H^*(F_{3;2})$
are the pull-backs of the Hodge classes $\la_1,\la_2\!\in\!H^*(\ov\cM_2)$
by the forgetful morphism~$f$ from the second factor.
Since 
$$\ov\fM_{2,0;(1),(1)}^{0,\i}(\P^1,1)_{\sim}
\approx\ov\cM_{2,1}\times\!\ov\!\fM_{0,1;(1),(1)}^{0,\i}(\P^1,1)_{\sim}
\cup 
\big(\ov\cM_{1,1}\!\times\!\ov\cM_{1,1}\!\times\! 
\ov\fM_{0,2;(1),(1)}^{0,\i}(\P^1,1)_{\sim}\big)\big/\Z_2$$
as spaces and the last factors in the two spaces on the RHS above
are zero- and one-dimensional, $\psb^2$ vanishes on the virtual class of~$F_{3;2}$.
Since $\la_1^3$ vanishes on the divisor in~$\ov\cM_2$ consisting of two-component curves and
\BE{RubbPsi_e} 
\blr{\psi_{\i}^{k-1},\ov\fM_{0,k;(1),(1)}^{0,\i}(\P^1,1)_{\sim}}=1
\qquad\forall~k\ge3,\EE
the second proofs of~\eref{M21_e2} and~\eref{M30_e2} readily show~that 
\begin{equation*}\begin{split}
\blr{\la_1^3, [\ov\fM_{2,0;(1),(1)}^{0,\i}(\P^1,1)_{\sim}]^{\vir}}
&=\blr{e(\C^2/T\Si_2),\Si_2}
\blr{\la_1^3,[\ov\cM_2]}=2\blr{\la_1^3,[\ov\cM_2]}\,;\\
\blr{\psi_{\i}\la_1^2,[\ov\fM_{2,0;(1),(1)}^{0,\i}(\P^1,1)_{\sim}]^{\vir}}
&=\blr{\la_1^2,\ov\cM_{1,1}}^2\,.
\end{split}\end{equation*}
Thus, the contribution of the seventh diagram in Figure~\ref{P4loc_fig2} 
to~\eref{M3loc_e2} is
\BE{M3loc_e2c}\begin{split}
&\int_{[F_3]^{\vir}}\frac{\al_1^3\al_2^2(9\al_1^2\!-\!\al_2^2)\la_1^3}
{2\al_1(\al_1^2\!-\!\al_2^2)(\al_1\!+\!\pst)\al_1}
-\frac12\int_{[F_3]^{\vir}}
\frac{\al_1^2\al_2^2(25\al_1^4\!-\!10\al_1^2\al_2^2\!+\!\al_2^4)\la_1^2\psi_{\i}}
{2\al_1(\al_1^2\!-\!\al_2^2)(\al_1\!+\!\pst)\al_1^2}\\
&=\frac{\lr{\pst^2,\ov\cM_{1,2}}}{2\,\al_1^4(\al_1^2\!-\!\al_2^2)}
\Big(\al_1^2\al_2^2(9\al_1^2\!-\!\al_2^2)\cdot2\blr{\la_1^3,[\ov\cM_2]}
-\al_2^2(25\al_1^4\!-\!10\al_1^2\al_2^2\!+\!\al_2^4)
\cdot\frac{\lr{\la,\ov\cM_{1,1}}^2}{2}\Big)\\
&=-\frac12\cdot\frac{1}{138240}\cdot\frac{\al_2^2(89\al_1^4\!-\!46\al_1^2\al_2^2\!+\!5\al_2^4)}
{\al_1^4(\al_1^2\!-\!\al_2^2)}\,;
\end{split}\EE
the last equality follows from Table~\ref{M2and3_tbl}.\\

\noindent
The locus represented by the last diagram in  Figure~\ref{P4loc_fig2}
is isomorphic~to 
$$F_4\equiv \ov\fM_{3,0;(1),(1)}^{0,\i}(\P^1,1)_{\sim}\,.$$ 
The equivariant euler class of its deformation space is reduced~to
$2\al_1(\al_1^2\!-\!\al_2^2)(\al_1\!-\!\psb)$.
The euler class of the obstruction bundle becomes
\begin{equation*}\begin{split}
&e\big(\E_3^*\!\otimes\!T_{p_2}V\big)=
(8\al_1^3\!-\!4\al_1^2\la_1\!+\!2\al_1\la_2\!-\!\la_3)
\big((\al_1\!-\!\al_2)^3\!-\!(\al_1\!-\!\al_2)^2\la_1\!+\!(\al_1\!-\!\al_2)\la_2\!-\!\la_3\big)\\
&\hspace{2.75in}
\big((\al_1\!+\!\al_2)^3\!-\!(\al_1\!+\!\al_2)^2\la_1\!+\!(\al_1\!+\!\al_2)\la_2\!-\!\la_3\big)\\
&\qquad\cong -\frac12\al_1^2(9\al_1^2\!-\!\al_2^2)\la_1^5
+\frac14 \al_1(45\al_1^4\!-\!14\al_1^2\al_2^2\!+\!\al_2^4)\la_1^4\\
&\qquad\quad-(18\al_1^6\!-\!20\al_1^4\al_2^2\!+\!2\al_1^2\al_2^4)\la_1^3
-(17\al_1^6\!+\!45\al_1^4\al_2^2\!+\!3\al_1^2\al_2^4\!-\!\al_2^6)\la_3
\mod H_{\T}^*\!\otimes\!H^{\le4}(F_4),
\end{split}\end{equation*}
where $\la_i\!\in\!H^*(F_4)$
is the pull-back of the Hodge class $\la_i\!\in\!H^*(\ov\cM_3)$
by the forgetful morphism~$f$.
Since 
\begin{equation*}\begin{split}
\ov\fM_{3,0;(1),(1)}^{0,\i}(\P^1,1)_{\sim}
\approx\ov\cM_{3,1}\times\!\ov\!\fM_{0,1;(1),(1)}^{0,\i}(\P^1,1)_{\sim}
&\cup  \ov\cM_{2,1}\!\times\!\ov\cM_{1,1}\!\times\! 
\ov\fM_{0,2;(1),(1)}^{0,\i}(\P^1,1)_{\sim}\\
&\cup  \big( (\ov\cM_{1,1})^3\!\times\! 
\ov\fM_{0,3;(1),(1)}^{0,\i}(\P^1,1)_{\sim}\big)\big/\S_3
\end{split}\end{equation*}
as spaces and the last factors in the three spaces on the RHS above
are zero-, one-, and two-dimensional, respectively, 
$\psb^3$ vanishes on the virtual class of~$F_4$.
Since $\la_1^5$ vanishes on the divisor in~$\ov\cM_3$ consisting of two-component curves and
and $\la_1^4$ vanishes on the subvariety  consisting of four-component curves
(three elliptic curves attached to a~$\P^1$),
\eref{RubbPsi_e} and the second proofs of~\eref{M21_e2} and~\eref{M30_e2} give 
\begin{equation*}\begin{split}
\blr{\la_1^5, [\ov\fM_{3,0;(1),(1)}^{0,\i}(\P^1,1)_{\sim}]^{\vir}}
&=\blr{\la_1^5 e(\E^*/T\Si_3),[\ov\cM_{3,1}]}
=\blr{\la_1^5\psi_1^2\!-\!\la_1^6\psi_1,[\ov\cM_{3,1}]}\,;\\
\blr{\psi_{\i}\la_1^4,[\ov\fM_{3,0;(1),(1)}^{0,\i}(\P^1,1)_{\sim}]^{\vir}}
&=8\blr{\la,\ov\cM_{1,1}}\blr{\la_1^3,\ov\cM_2};\\
\blr{\psi_{\i}^2\la_1^3,[\ov\fM_{3,0;(1),(1)}^{0,\i}(\P^1,1)_{\sim}]^{\vir}}
&=\blr{\la,\ov\cM_{1,1}}^3\,,\quad
\blr{\psi_{\i}^2\la_3,[\ov\fM_{3,0;(1),(1)}^{0,\i}(\P^1,1)_{\sim}]^{\vir}}
=\frac16\blr{\la,\ov\cM_{1,1}}^3\,.
\end{split}\end{equation*}
Thus, the contribution of the last diagram in Figure~\ref{P4loc_fig2} 
to~\eref{M3loc_e2} is
\begin{equation*}\begin{split}
&-\frac12\int_{[F_4]^{\vir}}\frac{\al_1^2(9\al_1^2\!-\!\al_2^2)\la_1^5}
{2\al_1(\al_1^2\!-\!\al_2^2)\al_1}
+\frac14\int_{[F_4]^{\vir}}\!\!\!
\frac{\al_1(45\al_1^4\!-\!14\al_1^2\al_2^2\!+\!\al_2^4)\la_1^4\psi_{\i}}
{2\al_1(\al_1^2\!-\!\al_2^2)\al_1^2}\\
&\hspace{2.5in}
-\frac16\int_{[F_4]^{\vir}}\!\!\!
\frac{(125\al_1^6\!-\!75\al_1^4\al_2^2\!+\!15\al_1^2\al_2^4\!-\!\al_2^6)\la_1^3\psi_{\i}^2}
{2\al_1(\al_1^2\!-\!\al_2^2)\al_1^3}\,.
\end{split}\end{equation*}
The preceding set of equations reduces this~to
\BE{M3loc_e2d}\begin{split}
&\frac{1}{2\,\al_1^4(\al_1^2\!-\!\al_2^2)}
\Bigg(-\frac12\al_1^4\big(9\al_1^2\!-\!\al_2^2\big)
\big(\blr{\la_1^5\psi_1^2,[\ov\cM_{3,1}]}\!-\!4\blr{\la_1^6,[\ov\cM_3]} \big)\\
&\hspace{1.1in}
+\al_1^2\big(45\al_1^4\!-\!14\al_1^2\al_2^2\!+\!\al_2^4\big)\cdot
2\blr{\la,\ov\cM_{1,1}}\blr{\la_1^3,\ov\cM_2}\\
&\hspace{1.1in}
-\big(125\al_1^6\!-\!75\al_1^4\al_2^2\!+\!15\al_1^2\al_2^4\!-\!\al_2^6\big)
\frac{\blr{\la,\ov\cM_{1,1}}^3}{6}\Bigg)\\
&\quad=-\frac12\cdot\frac{1}{2903040}\cdot
\frac{1747\al_1^6\!-\!1577\al_1^4\al_2^2\!+\!441\al_1^2\al_2^4\!-\! 
35\al_2^6}{\al_1^4(\al_1^2\!-\!\al_2^2)}\,.
\end{split}\EE
The sum of \eref{M3loc_e2a}, \eref{M3loc_e2b}, \eref{M3loc_e2c}, and~\eref{M3loc_e2d}
multiplied by~2 (to account for $i\!=\!3$) is
$$-\frac{1}{2903040}
\cdot\frac{1747 \al_1^6\!+\!292\al_1^4\al_2^2}{\al_1^4(\al_1^2\!-\!\al_2^2)}\,.$$
Adding in the same expression with~$\al_1$ and~$\al_2$ interchanged (to account for 
$i\!=\!4,5$), we find~that 
$$\GW_{3,1;(1)}^{\P^4,V_1}(\pt)
=-\frac{97}{193536}\,.$$
Along with the first equality in~\eref{P4_e}, this confirms the $\de\!=\!1$ case
of the second equality in~\eref{P4_e}.

\section{The Cieliebak-Mohnke approach to GW-invariants}
\label{CM_sec}

\noindent
Theorem~\ref{main_thm} and Examples~\ref{deg0_eg}-\ref{P4_eg} answer a key question
arising in recent attempts to adapt the idea of~\cite{CM} to constructing positive-genus
GW-invariants geometrically.
In this section, we review this approach and discuss its connections with 
Theorem~\ref{main_thm} and Examples~\ref{deg0_eg}-\ref{P4_eg}.\\

\noindent
Suppose $(X,\om)$ is a compact symplectic manifold such that $\om$ represents 
an integral cohomology class.
By \cite[Theorem~1]{D}, the Poincare dual of every sufficiently large integer multiple~$\de\om$
of~$\om$ can be represented by a symplectic hypersurface~$V$ in~$(X,\om)$.
If $A\!\in\!H_2(X;\Z)\!-\!\{0\}$ can be represented by a $J$-holomorphic map
$u\!:\Si\!\lra\!X$ for some $\om$-tame almost complex structure, then 
$$A\cdot V= \de\,\om(A)>0.$$ 
The idea of~\cite{CM} is to define the primary genus~0 GW-invariants by counting
$J$-holomorphic maps $\P^1\!\lra\!X$ that pass through generic representatives
of constraints of the appropriate total dimension and send $A\!\cdot\!V$ points of~$\P^1$
to~$V$ and dividing the resulting number by $(A\!\cdot\!V)!$. 
In order to ensure that the sets of maps being counted are finite,
the almost complex structure~$J$ on~$X$ is allowed to vary with 
the domain of the map in a coherent way. 
For $\de$ sufficiently large and a generic coherent family of~$J$'s,
every non-constant $J$-holomorphic map $u\!:\P^1\!\lra\!X$ of $\om$-energy
at most $\om(A)$ intersects $X\!-\!V$ and sends at least three distinct  points
of the domain to~$V$; see \cite[Proposition~8.13]{CM}.\\

\noindent
The almost complex structures~$J$ used in~\cite{CM} are required to be 
compatible with~$V$, in the sense that $J(TV)\!\subset\!TV$.
A coherent family of such complex structures can be viewed as 
a special case of a {\it single} pair~$(J,\nu)$, with~$\nu$ as~\eref{nucond0_e}
 satisfying the first condition in~\eref{nuVrestr_e}.
By a standard cobordism argument, the resulting count of $(J,\nu)$-maps 
is independent of a generic pair~$(J,\nu)$ compatible with~$(\om,V)$.
As in \cite[Section~10]{CM}, the independence of the counts on~$V$ can be shown 
by defining such counts with respect to two transverse Donaldson's hypersurfaces,
$V$ and~$V'$, that are compatible with the same $\om$-tame almost complex 
structure~$J$ on~$X$:
the dimension-counting argument at the beginning of Section~\ref{DirComp_subs} 
implies that the number of maps does not change if an additional $J$-holomorphic 
hypersurface~$V'$ is added.

\begin{rmk}\label{CM_rmk0}
The counts defined in~\cite{CM} have not been directly shown  to be invariant
under deformations of~$\om$, which is a central property of GW-invariants
in symplectic topology.
This could be established by showing that two Donaldson's divisors 
with respect to deformation equivalent symplectic forms and of the same degree
are deformation equivalent through Donaldson's divisors.
While it remains unknown whether this is the case, 
the counts of~\cite{CM} are directly shown  in~\cite{K} to be invariant under small deformations
of~$\om$. 
\end{rmk}

\begin{rmk}\label{CM_rmk}
Pairs~$(J,\nu)$ as in Section~\ref{RelDfn_sec} have been standard 
on the symplectic side of GW-theory at least since \cite{RT1,RT2}.
Using such pairs in~\cite{CM} would have avoided the need for 
an elaborate coherency condition on families of almost complex structure
and would have simplified the transversality issues,
likely cutting down the paper by half to two thirds.
It would have also extended the construction to genus~0 invariants with constraints
pulled back from the Deligne-Mumford space.
\end{rmk}

\noindent
By \cite[Section~3.2]{LR}, any topological component of the preimage of
a $J$-holomorphic hypersurface~$V$ in~$X$
under  the limit $u\!:\!\Si\!\lra\!X$ of a sequence of $J$-holomorphic maps 
$u_k\!:\Si_k\!\lra\!X$ 
from smooth domains meeting~$X\!-\!V$ comes with a holomorphic section of 
the pull-back of the normal bundle to~$V$.
By \cite[Section~6]{IPrel}, this conclusion also applies to $(J,\nu)$-holomorphic maps,
if $J(TV)\!=\!TV$ and $\nu$ satisfies the first condition in~\eref{nuVrestr_e}.
If $J$ and~$\nu$ also satisfy~\eref{NijenCond_e} and
the second condition in~\eref{nuVrestr_e},
spaces of maps from stable domains that satisfy this limiting condition 
are of dimension at least two less than the space of maps from smooth domains
which meet~$X\!-\!V$.
If all relevant domains are stable, invariants of~$(X,\om,V)$ can then 
be defined by counting such maps; see Section~\ref{RelDfn_sec}.\\

\noindent
The attempts in~\cite{G,IPvfc} to extend the approach of~\cite{CM} to 
positive-genus GW-invariants utilize the ideas outlined in the previous paragraph.
A crucial claim of \cite{G,IPvfc} is that the resulting counts 
of relative genus~$g$ degree~$d$ $(J,\nu)$-maps are independent of the choice
of $(g,A)$-hollow hypersurface~$V$ (at least, if it is a Donaldson's hypersurface).
As illustrated by Theorem~\ref{main_thm} and Examples~\ref{deg0_eg}-\ref{P4_eg},
this claim is often, but not always, true.
As illustrated by the direct proof of Theorem~\ref{main_thm} in Section~\ref{DirComp_subs}, 
it is true precisely 
when the ideas outlined in the previous paragraph are {\it not} needed
to define the relative counts.
In these cases, the argument in Section~\ref{DirComp_subs} implies that the counts
do not change when a second $J$-holomorphic hypersurface~$V_2$ is added.

\begin{rmk}\label{Nij_rmk}
The Nijenhuis condition~\eref{NijenCond_e} on~$J$ and 
the second restriction on~$\nu$ in~\eref{nuVrestr_e} are central to the setups
in~\cite{G} and~\cite{IPvfc}.
However, neither~\cite{G} nor~\cite{IPvfc} shows that  
these conditions can be satisfied with respect to two Donaldson's hypersurfaces simultaneously
(and even more complicated combinations of hypersurfaces are used in~\cite{IPvfc});
for transversality reasons, the second condition in~\eref{nuVrestr_e} needs to be
achieved 
for any specified restriction of~$\nu$ to the intersection of the two hypersurfaces.
These properties are used to show that the defined counts are independent of 
the choice of Donaldson's divisor~$V$. 
The attempt in~\cite{G2} to address the condition~\eref{NijenCond_e} on~$J$ 
fails at the first and  basic step, for which the reader was referred to 
the proof of \cite[Theorem~6.17]{MS1}; 
the author of~\cite{G2} has agreed with this and withdrawn his claim.
However, it appears plausible that some version of the intended argument in~\cite{G2} could be carried~out.
\end{rmk}

\begin{rmk}\label{Nij_rmk3}
There have been extensive discussions of~\cite{G} and~\cite{IPvfc}
at and since the SCGP workshop which took place during the week of March 17-21, 2014.
Examples~\ref{P1_eg} and~\ref{P4_eg} directly contradict the main premise
behind the approaches in~\cite{G} and~\cite{IPvfc}, formulated as Axiom~A4
in \cite[Definition~1.5]{IPvfc} and \cite[Theorem~11.1]{IPvfc};
these are relevant to~\cite{G} after evaluating on appropriate insertions
(Example~\ref{deg0_eg} contradicts~\cite{IPvfc}, but is excluded in~\cite{G}).
This leaves little of substance in either~\cite{G} or~\cite{IPvfc},
as far as an alternative construction of counts of $J$-holomorphic maps is concerned.
While the crucial issues raised in Remark~\ref{Nij_rmk} are not even mentioned 
and the independence issue is addressed incorrectly in~\cite{G} because of the author's
apparent misunderstanding of \cite{RT1,RT2},
this 100-page manuscript consists mostly of setting up the approach
based on the analytic construction of the Deligne-Mumford moduli space in~\cite{RS}.
In contrast, the setup in~\cite{IPvfc} is carried out in a standard way,
similarly to our Section~\ref{RelDfn_sec}.
On the other hand, \cite{IPvfc} contains two additional central premises that 
are intended to take the approach of~\cite{CM} to the level of virtual classes
(instead of just numbers as in~\cite{G}):
\cite[Lemma~7.4]{IPvfc}, claiming that certain (never defined) moduli spaces are manifolds,
and \cite[(11.4)]{IPvfc}, a symplectic sum formula for virtual classes 
(instead of numbers);
the reader is referred to \cite{IPrel,IPsum} for the proofs of these statements.
However, the two statements require gluing maps with rubber components;
such gluing constructions are skipped and claimed to be unnecessary precisely because 
of a pre-print like~\cite{IPvfc} is {\it in preparation} (first cited in~2001).
These issues are discussed in more detail in \cite[Section~2]{SympSum};
unfortunately, the video of the discussion of~\cite{IPvfc} itself 
at the SCGP workshop has not been made publicly available due to T.~Parker's veto. 
\end{rmk}

\noindent
In principle, the approach of~\cite{CM} could be adapted to constructing 
positive-genus GW-invariants
by subtracting lower-genus contributions with appropriate coefficients
if the real dimension of the target~$X$ is~8 or less. 
These coefficients are determined by the chern classes of the divisor~$V$,
the top intersections of $\la$-classes on $\ov\cM_g$,
and the relative GW-theory of~$\P^1$.
While all of these are computable in some sense, it does not appear that 
the resulting coefficients would have reasonably simple expressions.


\begin{thebibliography}{99}

\bibitem{BF} K.~Behrend and B.~Fantechi, {\it The intrinsic normal cone},
Invent.~Math.~128 (1997), no.~1, 45–-88.

\bibitem{CM} K.~Cieliebak and K.~Mohnke,
{\it Symplectic hypersurfaces and transversality in Gromov-Witten theory},
J.~Symplectic Geom.~5 (2007), no.~3, 281--356.

\bibitem{D} S.~Donaldson,
{\it Symplectic submanifolds and almost-complex geometry},
J.~Differential Geom.~44 (1996), no.~4, 666–-705. 

\bibitem{FP} K.~Faber and R.~Pandharipande, 
{\it Relative maps and tautological classes}, J.~EMS~7 (2005), no.~1, 13--49.

\bibitem{SympSum} M.~Farajzadeh  Tehrani and A.~Zinger,
{\it On Symplectic Sum Formulas in\\ Gromov-Witten Theory},
arXiv/1404.1898. 

\bibitem{Faber} C.~Faber, 
{\it Algorithms for computing intersection numbers on moduli spaces of curves, 
with an application to the class of the locus of Jacobians}, 
{\it New Trends in Algebraic Geometry} (Warwick, 1996), 93–-109,
London Math.~Soc.~Lecture Note, Ser.~264, Cambridge Univ.~Press, 1999.

\bibitem{FaberP} C.~Faber and Pandharipande, 
{\it Relative maps and tautological classes}, 
J.~Eur.~Math.~Soc.~7 (2005), no.~1, 13-–49.

\bibitem{FO} K.~Fukaya and K.~Ono, {\it Arnold conjecture and Gromov-Witten invariant}, 
Topology 38 (1999), no.~5, 933–-1048.

\bibitem{G} A.~Gerstenberger,
{\it Geometric transversality in higher genus Gromov-Witten theory},
arXiv/1309.1426.

\bibitem{G2} A.~Gerstenberger,
{\it Addendum to ``Geometric transversality in higher genus Gromov-Witten theory"},
pre-print~2014.

\bibitem{GP} T.~Graber and R.~Pandharipande, 
{\it Localization of virtual classes}, 
Invent.~Math.~135 (1999), 487--518.

\bibitem{GV} T.~Graber and R.~Vakil,
{\it Relative virtual localization and vanishing of tautological classes 
on moduli spaces of curves}, 
Duke Math.~J.~130 (2005), no.~1, 1–-37. 

\bibitem{MirSym} K.~Hori, S.~Katz, A.~Klemm, R.~Pandharipande, 
R.~Thomas, C.~Vafa, R.~Vakil, and E.~Zaslow, {\it Mirror Symmetry},
Clay Math.~Inst., AMS, 2003. 

\bibitem{HLR} J.~Hu, T.-J.~Li, and Y.~Ruan, 
{\it Birational cobordism invariance of uniruled symplectic manifolds}, 
Invent.~Math.~172 (2008), no.~2, 231–-275.

\bibitem{Inc} E.~Ionel,
{\it GW-invariants relative normal crossings divisors},
arXiv/1103.3977.

\bibitem{IPrel} E.~Ionel and T.~Parker, 
{\it Relative Gromov-Witten invariants},
Ann.~of Math.~157 (2003), no.~1, 45--96.

\bibitem{IPsum} E.~Ionel and T.~Parker, 
{\it The symplectic sum formula for Gromov-Witten invariants},
Ann.~of Math.~159 (2004), no.~3, 935--1025.

\bibitem{IPvfc} E.~Ionel and T.~Parker, 
{\it A natural Gromov-Witten virtual fundamental class},
arXiv/1302.3472. 

\bibitem{K} A.~Krestiachine, 
{\it Donaldson hypersurfaces and Gromov-Witten invariants},
Ph.D. Thesis HU Berlin, in preparation.

\bibitem{LR} A.-M.~Li and Y.~Ruan, 
{\it Symplectic surgery and Gromov-Witten invariants of Calabi-Yau 3-folds},
Invent.~Math.~145 (2001), no.~1, 151--218. 

\bibitem{Jun1} J.~Li,
{\it Stable morphisms to singular schemes and relative stable morphisms}, 
J.~Differential Geom.~57 (2001), no.~3, 509--578.

\bibitem{Jun2} J.~Li,
{\it A degeneration formula for GW-invariants}, 
J.~Differential Geom.~60 (2002), no.~1, 199--293.

\bibitem{Jun3} J.~Li, 
{\it Lecture notes on relative GW-invariants}, 
{\it Intersection Theory and Moduli}, ICTP Lect.~Notes, XIX, 41–-96,
ICTP, 2004. 

\bibitem{LT} J.~Li and G.~Tian,
{\it Virtual moduli cycles and Gromov-Witten invariants of general symplectic manifolds},
in {\it Topics in Symplectic 4-Manifolds}, 47--83, Internat.~Press 1998.

\bibitem{Lo} E.~Looijenga, {\it Smooth Deligne-Mumford compactifications 
by means of Prym level structures}, J.~Algebraic Geom.~3 (1994), 283–-293.

\bibitem{MP} D.~Maulik and R.~Pandharipande, 
{\it A topological view of Gromov-Witten theory}, Topology~45 (2006), 
no.~5, 887--918.

\bibitem{MS} D.~McDuff and D.~Salamon,
{$J$-holomorphic curves and quantum cohomology}, 
University Lecture Series~6, AMS~1994.

\bibitem{MS1} D.~McDuff and D.~Salamon, 
{\it Symplectic Topology}, 2nd Ed., Oxford University Press, 1998.

\bibitem{MS2} D.~McDuff and D.~Salamon, 
{\it $J$-Holomorphic Curves and Symplectic Topology}, 2nd Ed., 
AMS Colloquium Publications~52, AMS~2012.

\bibitem{Mumford} D.~Mumford, 
{\it Towards an enumerative geometry of the moduli space of curves}, 
{\it Arithmetic and Geometry}, Vol.~II, 271–-328, 
Progr.~Math.~36, Birkh\"auser, 1983.

\bibitem{OP} A.~Okounkov and R.~Pandharipande, 
{\it Virasoro constraints for target curves}, 
Invent.~Math.~163 (2006), no.~1, 47–-108.

\bibitem{RS} J.~Robbin and D.~Salamon, 
{\it A construction of the Deligne-Mumford orbifold}, 
J.~Eur.~Math.~Soc.~8 (2006), no.~4, 611–-699.

\bibitem{RT1} Y.~Ruan and G.~Tian,
{\it A mathematical theory of quantum cohomology}, 
J.~Differential Geom.~42 (1995), no.~2, 259–-367.

\bibitem{RT2} Y.~Ruan and G.~Tian,
{\it Higher genus symplectic invariants and sigma models coupled with gravity},
Invent.~Math.~130 (1997), no.~3, 455--516.

\bibitem{ObsAnal} A.~Zinger,
{\it Enumerative vs.~symplectic invariants and obstruction bundles},
J.~Symplectic Geom.~2 (2004), no.~4, 445–-543.

\bibitem{g2n2and3} A.~Zinger,
{\it Enumeration of genus-two curves with a fixed complex structure in $\P^2$ and~$\P^3$}, 
J.~Differential Geom.~65 (2003), no.~3, 341–-467. 

\bibitem{anal} A.~Zinger, {\it Basic Riemannian geometry and Sobolev estimates
used in symplectic topology}, arXiv/1012.3980.


\end{thebibliography}
\end{document}